\documentclass[reqno,11pt]{amsart}
\pdfoutput=1

\usepackage[utf8]{inputenc}
\usepackage[italian,english]{babel}
\usepackage{amsfonts}
\usepackage{amssymb}
\usepackage{tikz}
\usepackage{svg}
\usepackage[right=3cm,bottom=2.5cm,footskip=30pt]{geometry}
\usepackage[autostyle]{csquotes}
\MakeOuterQuote{"}
\usepackage{color}
\usepackage{mathrsfs}
\usepackage{amsmath}
\usepackage{mathtools}
\usepackage{enumitem}
\usepackage{nicefrac}
\usepackage{multirow}
\usepackage{aligned-overset}

\makeatletter
\renewcommand\paragraph{\@startsection{paragraph}{4}{\z@}%
  {3.25ex \@plus1ex \@minus.2ex}%
  {-.5em}%
  {\normalfont\normalsize\bfseries\ignorespaces}}
\makeatother

\usetikzlibrary{svg.path,arrows,calc,decorations.pathreplacing}
\tikzset{>=latex'}

\usepackage[toc]{appendix}
\usepackage{etoolbox}
\cslet{blx@noerroretextools}\empty
\usepackage[maxbibnames=9,giveninits=true,style=alphabetic,
  sorting=nyt,backend=biber]{biblatex}
\usepackage{bibsettings}
\usepackage[colorlinks, citecolor=citegreen, linkcolor=refred,urlcolor=blue, unicode,psdextra,hypertexnames=false]{hyperref}
\usepackage{cleveref}
\usepackage{orcidlink}

\crefname{enumi}{}{}
\crefname{enumii}{}{}
\expandafter\def\csname ver@etex.sty\endcsname{3000/12/31}

\usepackage{headings}
\usepackage{accents}
\usepackage{autonum}
\usepackage{gasymbols}

\definecolor{citegreen}{rgb}{0,0.3,0}
\definecolor{refred}{rgb}{0.5,0,0}
\usepackage[equation=section]{theorems}

\author[L.~Benatti]{Luca Benatti \orcidlink{0000-0002-4685-7443}}
\address{L.~Benatti, University of Vienna,
Oskar-Morgenstern-Platz 1, 1090, Vienna, Austria}
\email{luca.benatti@unive.ac.at}

\author[A.~Pluda]{Alessandra Pluda \orcidlink{0000-0003-4714-4119}}
\address{A.~Pluda, Universit\`a di Pisa,
Largo Bruno Pontecorvo 5, 56127 Pisa, Italy}
\email{alessandra.pluda@unipi.it}

\newcommand{\pas}{\partial_s}
\newcommand{\pat}{\partial_t}
\newcommand{\pax}{\partial_x}

\newcommand{\nodes}{\mathcal{P}}

\newcommand{\E}{\mathcal{E}}

\newcommand{\Net}{\gamma}

\newcommand{\Bcnd}{\mathcal{B}}

\DeclareMathOperator{\id}{Id}
\DeclareMathOperator{\rnk}{rnk}
\DeclareMathOperator{\diag}{diag}
\DeclareMathOperator{\lspan}{span}

\title[One thousand and one higher-order geometric flows of networks]{One thousand and one higher-order\\ geometric flows of networks}

\begin{document}

\begin{abstract}
    We show short-time existence and uniqueness up to reparametrization for a large class of higher-order geometric flows of curves and networks obtained as $L^2$-gradient flow of higher-order functionals involving derivatives of the curvature. Additionally, we prove a long-time existence result via energy methods.

\end{abstract}

\maketitle


\smallskip

\noindent \underline{\smash{Keywords}}: higher-order geometric flows, networks, motion by curvature, local existence and uniqueness, long-time existence, gradient flows.
\smallskip

\tableofcontents

\section{Introduction}

    The study of geometric evolutions typically proceeds through three stages, which arise in a natural order: establishing the existence of the flow, at least for a short time; understanding which obstructions prevent the flow from existing for all time; and analyzing its asymptotic behavior. 
    
    One of the first major successes of this program is the complete description of the evolution of embedded closed planar curves by curve shortening flow. While short-time theory is comparatively robust, reflecting the second-order parabolic nature of the evolution, the long-time and asymptotic behavior are particularly remarkable: every simple closed curve eventually becomes convex and then shrinks to a round point in finite time \cite{GaHa86,Gr87}. No singularities develop before the natural extinction time $T=A_0/2\pi$, which is determined solely by the area $A_0$ enclosed by the initial curve.

    This clean result has motivated extensive efforts to generalize the theory in several directions. One direction involves replacing the classical second-order curve shortening flow with higher-order geometric evolutions. To the author's knowledge, Polden \cite{poldenthesis}, closely followed by Dziuk--Kuwert--Schätzle \cite{DzKuSc02}, was the first to carry out a complete analysis of closed curves under such a higher-order flow, from short-time existence to asymptotic behavior. Another direction aims to extend the theory from smooth curves to less regular objects. The simplest examples in this class are networks, namely one-dimensional connected sets composed of finitely many curves meeting at junctions. Their evolution under curve shortening flow has developed into a rich theory in its own right; for a comprehensive account, we refer to the monograph \cite{MNPS24}.

    This paper brings these two directions together in a systematic way. We develop a general approach to higher-order geometric flows on networks by identifying broad classes of junction constraints under which the evolution is well-posed. Specifically, we prove short-time and (conditional) long-time existence results for geometric gradient flows associated with arbitrary $m$-th order energy functionals for networks of the form
    \begin{equation}
    \E(\gamma) \coloneqq \int_{\gamma} \frac12\abs{\pas^m\gamma}^2\dif\gamma + \frac12\ell(\gamma), \qquad m \geq 2.
    \end{equation}
    For $m=1$, the functional $\E$ reduces to length, and its geometric gradient flow is precisely the curve shortening flow discussed above. This case is not included in the present work. The first case covered by our analysis is $m=2$, which corresponds to the elastic flow, another central and extensively studied geometric evolution of curves (see, e.g., \cite{LaSi85, Li89, poldenthesis, DzKuSc02, Li12, NoOk14, DaPoSp16, NoOk17, Sp17, MaPo20, Po20, KM24, Miura25, MW25, MMR25}).

        Before addressing the three stages mentioned, here we have face a preliminary challenge: setting the class of evolutionary problems we want to study. This task is driven by the variational structure. The underlying variational problem naturally yields the evolution equation, which is encoded in the first variation of $\E$. An energy functional of order $m$ gives rise to a parabolic equation of order $2m$. The genuinely delicate issue is the choice of boundary conditions. From a PDE perspective, it is clear that arbitrary boundary conditions cannot be allowed, as the problem is bound to be ill-posed in certain classes. At the same time, there may be several possible choices of boundary constraint equations for each fixed order $m$. The difficulty therefore lies not in selecting a convenient set of boundary conditions, but in identifying a sufficiently large and meaningful family of admissible constraint sets that still lead to a well-posed evolution.
        
        A possible choice is to impose a full set of Dirichlet boundary conditions. For an energy of order $m$, one expects a Dirichlet boundary conditions is of order at most $m-1$. We distinguish between topological conditions and higher-order Dirichlet boundary conditions (\cref{def:admissible_dirichlet_0,def:admissible_dirichlet}). The former are vector-valued conditions of order zero. Their role is to preserve the relations between the supports of the curves along the flow, basically imposing constraints on their endpoint (e.g. fixed endpoints, junctions, ...). The latter are scalar-valued and involve geometric quantities such as tangents and curvatures. Dirichlet boundary conditions restricts the class of networks and consequently the class of admissible variations for the network. Imposing a full set of such conditions makes the boundary terms in the first variation vanish, thereby allowing us to interpret the velocity of the evolution as the $L^2$-gradient of $\E$. 
        
        Imposing a full set of Dirichlet conditions is only one possible way of canceling the boundary terms in the first variation. If the aim is to treat a broader class of boundary value problems, this choice is too restrictive. This leads to Neumann-type boundary conditions. Intuitively, these Neumann conditions play the role of the orthogonal complement of the prescribed Dirichlet ones. More precisely, there is a natural coupling mechanism by which the absence of a Dirichlet condition of order $i$ gives rise to a Neumann condition of order $2m-1-i$. This pairing is crucial for proving long-time existence and marks the point at which the variational structure becomes essential in our analysis. 

        From the analytic point of view, the discussion above completes the picture. However, this is precisely where the geometric nature of the problem comes into play. We are interested in the geometry of the evolution, namely in the support of the curves involved, and not in the way their parametrizations change along the flow. The additional degrees of freedom arising from this geometric invariance have three main consequences. First, uniqueness can only be understood up to reparametrization: this is what we call geometric uniqueness of the flow. Second, a purely geometric evolution can prescribe only the normal velocity, so that the resulting equation is not strictly parabolic, but degenerate. Third, whenever a specific parametrization is selected, the choice saturates $m-1$ degree of freedom at the boundary, making the analytic set of boundary conditions described above redundant with respect to the underlying geometric problem. The first main contribution of the present paper is to develop a systematic procedure for deriving the full list of admissible boundary conditions for the geometric $L^2$-gradient flow associated with the energy $\E$.

        With the class of problems now fully set up, we turn to the analysis of well-posedness for the resulting flow for every order $m\geq 2$ and set of boundary conditions. In rough terms, the main result is the following.
        
\begin{theorem}\label{thm:buona-positura}
Let $\mathcal{C}$ be a class of networks with a fixed topology, satisfying a set of admissible boundary conditions up to order $2m-1$ and let $\gamma_0$  be an admissible initial network. Then, there exist $T > 0$ and a solution $\left(\gamma(t)\right)_{t\in[0,T]}$ to the geometric gradient flow of $\E(\Net)$ in $\mathcal{C}$ with
 \begin{equation}
     \gamma^i \in W_p^1\left((0,T); L_p\right) \cap L_p\left((0,T); W_p^{2m}\right). 
 \end{equation}
    The solution is geometrically unique and admits a regular parametrization that is smooth for all $t > 0$.
\end{theorem}

    To prove this theorem we need first to break the invariance by reparametrization by fixing a suitable tangential velocity, thereby removing the tangential degrees of freedom. We conveniently choose the tangential velocity and obtain the Special Flow and we add the $m-1$ remaining boundary conditions.  We then linearize this system around the initial datum.  Using the classical theory of Solonnikov \cite{solonnikov2}, we prove well-posedness for the associated linear system in \cref{thm:existence_Solonnikov}. Applying Banach's Fixed Point Theorem then yields a short-time existence for the Special Flow  shown in \cref{short time existence}. The existence of solutions to the geometric gradient flow follows immediately. 
    
    To prove uniqueness, we still have to show that any arbitrary tangential choices can be absorbed by finding a suitable family of reparametrizations. This is the content of \cref{sec:uniqueness}. Finally, since the initial data lie only in a fractional Sobolev space ($W^{2m-\nicefrac{2m}{p}}_p$), we employ Angenent's parameter trick to show that the solution instantaneously becomes $C^\infty$.

    While the proof of this result is technical, it relies on well-established techniques. In contrast, the analysis of the long-time behavior of the flow requires substantially new ideas.

\begin{theorem}\label{thm:longtime-intro}
Let $\mathcal{C}$ be a class of networks with a fixed topology, satisfying boundary conditions up to order $2m-1$ and let $\gamma_0$  be an admissible initial network.  Suppose that $\gamma_t$ is a maximal solution $\left(\gamma(t)\right)_{t\in[0,T)}$ to the geometric gradient flow of $\E(\Net)$ in $\mathcal{C}$. If $\gamma(t)$ does not degenerate as $t \to T$, then $T=+\infty$.
\end{theorem}

    The leading idea is to prove that under the above assumption every finite $T$ is an admissible starting point for the problem. To show that, one must bound the norm of $\gamma$ in the initial data space as $t \to T$. This bound follows from the fact that the kinetic energy of the flow, namely the $L^2$-norm of the velocity, has bounded variation and therefore cannot blow up in finite time. Since the velocity is a differential operator of order $2m$ along each curve, this estimate naturally provides the required control in the initial data space $W^{2m - 2m/p}_{p}$. 

    A first obstruction to controlling the kinetic energy is that its variation along the flow has no evident sign and it is of order $4m$, which is far too high with respect to the initial data space. For closed curves, where no boundary is present, this issue can be overcome integrating by parts $m$ times: the leading term becomes of order $3m$ and, more importantly, has a manifestly negative sign. In our setting, however, boundary terms of order up to $4m-1$ emerge. These terms cannot be directly reabsorbed in the leading bulk term by means of interpolation inequalities. One would like to exploit that boundary conditions are preserved along the flow, and hence that their time derivatives vanish, in order to reduce the order of these terms. However, identifying this structure within the lengthy expression produced by this differentiation is far from straightforward. The existing literature offers little guidance in this direction, as it typically treats each specific problem by means of \textit{ad hoc} cancellations dictated by its particular structure. This is precisely where the preliminary work devoted to the Dirichlet--Neumann coupling becomes essential. Up to lower-order terms, the time variations of the boundary conditions satisfy the same orthogonality structure as the original conditions, and this reveals the cancellations needed to control the boundary terms in our non specific setting.

Among the hypotheses of the theorem, we required the solution to be non-degenerate as $t \to T$. We essentially requires that both the lengths of all curves in the network are bounded away from zero and the topological Dirichlet boundary condition hold uniformly for $t$ close to $T$ (see \cref{thm:longtime}). This requirement makes the result conditional. However, this description of the long-time behavior is also optimal, since there are examples of degeneracy along the flow.

We point out that the statement of the long-time behavior is exactly the same for every order $m\geq 2$. This might come as a surprise, since the long-time descriptions of the two best-studied cases ($m=1$ for the curve shortening flow and $m=2$ for the elastic flow) differ significantly from each other. Moreover, in the previous literature, there was only one example of long-time existence for higher-order flow of networks \cite{DaChPo19,garcke-menzel-pluda-20}. For the elastic flow of networks with natural boundary conditions, 
non-degeneracy was expressed by requiring that at each junction all concurrent curves do not align during the evolution, the angles formed by two subsequent unit tangent vectors at the junction must remain strictly positive. Stated in this way, non-degeneracy seems to affect first-order derivatives, inducing a wrong belief that the higher the order of the flow, the higher the derivatives involved in the conditional result. Ultimately, it required a deeper understanding of the flow to express this same condition only in terms of the topological condition.

To achieve a complete description of the flow, once established that that flow exists globally in time, one seeks to describe the shape of the evolving object as time approaches $+\infty$, and to prove stability and convergence of the solutions. The present paper does not address this final question, leaving it as the subject of future research.

\subsection{Comparing with the existing literature}
Interest in higher-order geometric flows has grown steadily over the past fifteen years, and a community of researchers has flourished. We try to give an accurate description of the vast existing literature that inspired and guided our work.

For closed curves, Langer and Singer \cite{LaSi85} pioneered the variational study of elastic curves (see also \cite{Ko96}). Then, early work by Linnér \cite{Li89} gave numerical evidence that 
curves evolving under elastic flow can develop self-intersections or lose convexity. Polden \cite{poldenthesis} and Dziuk, Kuwert, and Schätzle \cite{DzKuSc02} proved the well-posedness and long-time existence for both penalized and free elastic flow. Finally, Mantegazza and Pozzetta \cite{MaPo20,Po20} established the uniqueness of the asymptotic limit of the penalized elastic flow of closed curves, while Miura and Wheeler completed the analysis of the free elastic flow \cite{MW25}. Additionally, the sharp energy threshold below which closed curves strictly preserve their initial embeddedness was determined in \cite{MMR25}. Okabe instead analyzed the elastic flow of closed curves under different constraints \cite{Ok07,Ok08}.

Passing to flows of order grater than four, Polden \cite{poldenthesis} gave a quite complete description for closed curves and McCoy and Wheeler also studied flows of order $2m$ in the recent paper \cite{McCW26} (see also \cite{McWhWu17} for the case $m=3$).

As detailed explaned in the first part of the introduction, when we consider open curves, there is more than one choice of boundary conditions. For Navier boundary conditions, Novaga and Okabe established short-time and long-time existence, as well as asymptotic behavior \cite{NoOk14,NoOk17}. Under clamped boundary conditions, short-time existence was proved by Spener \cite{Sp17}, while Lin \cite{Li12} addressed their long-time behavior; Dall’Acqua, Pozzi, and Spener \cite{DaPoSp16} completed the analysis with  convergence and asymptotic properties. We mention also the case of a partial free boundary studied in \cite{Diana,WhWh}. The case of non-compact curves is challenging. Nevertheless, it has been studied by several authors: Novaga and Okabe \cite{NoOk14} provided results on both long-time behavior and asymptotic analysis. A new energy method was introduced in \cite{MR26} to analyze the elastic flow of infinite complete curves, while the preservation of their embeddedness and graphicality was established in \cite{MR25b}. As above, by elastic flow, we mean the flow with length penalization; the $L^2$-gradient flow of $\int\vert k \vert^2\,\mathrm{d}s$ when the curve is instead subjected to a fixed length is studied in \cite{DaLiPo14,DaLiPo17,DzKuSc02,RuSp20}. 

Related to the elastic energy we also have 
a second-order gradient flows obtained via the tangent indicatrix \cite{We93, We95, LiYa18, LiYaSc15, NoPo, OkPoWh}. Furthermore,  obstacle problems for the elastic flow and bending energy was also investigates \cite{Ken25}: Yoshizawa  established the local-in-time existence of weak solutions.  Other notable examples of fourth-order flows include the Helfrich flow \cite{DaPo14, Wh15}, the curve diffusion flow \cite{AbBu19,AbBu20}, the length-constrained diffusion flow \cite{McWhWu19}, and heterogeneous elastic wires \cite{heterogeneous-elastic1, heterogeneous-elastic2,La25}, which continue to enrich the literature.  Anisotropic extensions have also attracted attention: Barrett, Garcke, and Nürnberg \cite{anisotropic1} developed parametric approximations for the anisotropic elastic flows of open and closed curves, while Bellettini, Kholmatov, and Novaga \cite{BKN25} focused on the crystalline case.

Shifting focus from single curves to networks, the analysis of higher-order flows has concentrated on the elastic flow. When we fix only the topology of the network, the problem is coupled with second- and third-order Neumann boundary conditions. In this case, short-time existence has been established in different settings: Hölder solutions in $\R^2$ \cite{garcke-menzel-pluda-19} and $\R^n$ \cite{DaChPo20}, and Sobolev solutions in $\R^2$ \cite{garcke-menzel-pluda-20}. The transition to long-time behavior is addressed both in \cite{DaChPo19} and in \cite{garcke-menzel-pluda-20}. As in the case of curves, one can fix both the structure of the network and the angles at the junctions. Short-time existence for Hölder solutions of this problem was shown in \cite{garcke-menzel-pluda-19}, but no long-time existence result is currently available in the literature. We emphasize that our paper covers the analysis of the long-time behavior for this latter case as well. 

Finally, once the existence of the flow is solved, one can focus on beautiful geometric questions. For example, Huisken's problem asks whether there is an elastic flow of closed planar curves that is initially contained in the upper half-plane but ``migrates'' to the lower half-plane at a positive time. See the recent developments in the case of open curves under natural boundary conditions \cite{KM24, Miura25}.

\subsection*{Acknowledgment}
This research was funded in part by the Austrian Science Fund (FWF) [grant DOI \href{https://www.fwf.ac.at/en/research-radar/10.55776/EFP6}{10.55776/EFP6}]. For open access purposes, the authors have applied a CC BY public copyright license to any author-accepted manuscript version arising from this submission.

Luca Benatti and Alessandra Pluda were partially supported by the BIHO Project "NEWS – NEtWorks and Surfaces
evolving by curvature" and by the MUR Excellence Department Project awarded to the Department
of Mathematics of the University of Pisa. 

The authors are members of the INDAM–GNAMPA.
Alessandra Pluda is partially supported by the INdAM - GNAMPA Project (codice CUP E53C25002010001) "ANGELS".

Alessandra Pluda gratefully acknowledges support from the Erwin Schr\"odinger International Institute for Mathematics and Physics. Some of the research for this paper was carried out during the thematic program on \textit{Free Boundary Problems}.

Artificial intelligence was used exclusively to improve grammar and clarity of expression. It was not used to derive, verify, or prove any of the results presented.

\section{Curves and networks}

A planar regular curve $\gamma$ is a continuous map $\gamma:[0,1]\to \R^2$, differentiable on $(0,1)$ with $\abs{\pax \gamma(x)}\neq 0$ for every $x\in (0,1)$. By an abuse of terminology, we use $\gamma$ to denote the parametrized curve, its equivalence class, and its support. 

We denote by $s$ the arclength parameter. Whenever we write a function  $f(s)$ depending on $s$, we actually mean $f(x(s))$. Additionally, we have $\pas\coloneqq \frac{1}{\abs{\pax \gamma}}\pax$. We extensively use the arclength measure when integrating with respect to the volume element $\mu_g$ on $[0,1]$ induced by a regular rectifiable curve $\gamma$, namely,
given an integrable function $f$ on $[0,1]$, we have
    \begin{equation}
        \int_{\gamma} f \dif \gamma\coloneqq \int_0^1 f(x) \abs{\pax\gamma(x)} \dif x.
    \end{equation}
    
Assuming $\gamma$ is of class $H^2$, we denote by $\tau \coloneqq \pas \gamma$ the unit tangent vector, by $\nu$ the unit normal vector (the counterclockwise rotation of $\frac{\pi}{2}$ of $\tau$) and by $\kappa=\pas^2\gamma$ the curvature vector of the curve $\gamma$, which is defined at almost every point.  In $\mathbb{R}^2$ we can write the curvature vector as $\kappa=k\nu$ with $k$ the oriented curvature.

\begin{definition}
    A network is a finite collection of curves $\gamma= (\gamma^1, \ldots, \gamma^N)$. 
\end{definition}

A generic curve in $\gamma$ will be denoted by $\gamma^i$; when there is no risk of confusion, we drop the $i$.

\medskip

As a direct generalization of what we just said for curves, given a network $\gamma = (\gamma^1, \ldots, \gamma^N)$, we denote by $\varphi$ the collection $\varphi = (\varphi^1, \ldots, \varphi^N)$ where each function $\varphi^i$ acts on $\gamma^i$. Again, when there is no possibility of confusion, we will omit indices.
When we integrate along a network, we write
    \begin{equation}
        \int_{\gamma} \varphi \dif \gamma = \sum_{i=1}^N \int_{\gamma^i} \varphi^i \dif \gamma^i.
    \end{equation}

 Given a point $P\in \R^2$, we say that $P \in \gamma^i$ if $P$ belongs to the support of $\gamma^i$ and $P\in \gamma$ means that $P \in \gamma^i$ for some $i \in \set{1,\ldots,N}$. For the set $\set{\gamma^i(0), \gamma^i(1)}$ of the endpoints of $\gamma^i$ we use the symbol $\nodes^i$ and $\nodes$ is the union of $\nodes^i$ over $i \in \set{1,\ldots,N}$.  A point $P \in \nodes$ is called $d$-junction if $\#\set{\gamma^{-1}(P)}=d$. For consistency with the case of a single curve, if $d=1$, we still call $P$ an endpoint. 
 
In this paper, it is particularly convenient to introduce an "orientation-free" evaluation at the boundary. Assume that $\gamma$ is a curve, not a loop, i.e. $\gamma(0)\neq \gamma(1)$. If $\gamma^{-1}(P) \in \set{0,1}$ we define
    \begin{equation}
      \evaluat{\pax^k \varphi}_P \coloneqq (-1)^{k\left(\gamma^{-1}(P) +1\right)}\pax^k \varphi\left(\gamma^{-1}(P)\right).
    \end{equation}
If $\gamma^{-1}(P) \not\in \set{0,1}$ the evaluation will be intended to be $0$. The product of expressions evaluated at different points is defined to be $0$.
If $\gamma(0)=\gamma(1)$, the above definition is ambiguous since $\gamma^{-1}(P)=\set{0,1}$. In this case, we will perform computations as there were two copies of the curve - labeled differently - for which $P$ is uniquely determined as the image of $0$ or $1$, and the other endpoint is different from $P$. Subsequently, the second copy will be replaced by the curve itself.    
Note that this evaluation is distributive with respect to the product and the sum. 

We  also introduce 
\begin{equation}
 \evaluat{\pax^k\varphi}_{\nodes}  \coloneqq \sum_{P\in\nodes}\left(\evaluat{\pax^k \varphi}_P \right).
\end{equation}

\begin{definition}[Star network]
Let $\gamma=(\gamma^1,\ldots,\gamma^N)$ be a network. We say that $\gamma$ is a \emph{star network with center $P$}, for $P \in \nodes$, if $P\in\nodes ^i$ for every $ i \in \set{1, \ldots, N}$.
\end{definition}

If $\gamma=(\gamma^1,\ldots,\gamma^N)$ is a star network with center $P$, then $N\leq \#\set{\gamma^{-1}(P)}\leq 2N$. Furthermore, there are examples of star networks with two distinct centers.

\medskip

We conclude this section with a definition that will be crucial in the sequel.
\begin{definition}
We denote by $\pol_\sigma^h(\gamma)$
a  polynomial in $\pas \gamma,\dots,\pas^h \gamma$ with constant 
coefficients in $\R$ such that every monomial can be bounded by
\begin{equation}
a \prod_{l=1}^h	\abs{\pas^l \gamma}^{\beta_l}\quad\text{ with} \quad \sum_{l=1}^h(l-1)\beta_l \leq\sigma,
\end{equation}
where $\beta_l\in\mathbb{N}_0$ for $l\in\set{1,\dots,h-1}$ and $a\in \R$. With abuse of notation, we will say that $A(\gamma) \in \R^k$ satisfies $A(\gamma)= \pol_\sigma^h(\gamma)$ if all entries of $A(\gamma)$ can be estimated by a polynomial $\pol_\sigma^h(\gamma)$.
\end{definition}

The polynomials $\pol_\sigma^h(\gamma)$ satisfy these basic calculus rules:
\begin{align}
\partial_s\left(\pol_\sigma^h(\gamma)\right)&=\pol_{\sigma+1}^{h+1}(\gamma),\\
\pol_{\sigma_1}^{h_1}(\gamma)+\pol_{\sigma_2}^{h_2}
(\gamma)&=\pol_{\max\set{\sigma_1, \sigma_2}}^{\max\set{h_1,h_2}}(\gamma),\\
\pol_{\sigma_1}^{h_1}(\gamma)\pol_{\sigma_2}^{h_2}
(\gamma)&=\pol_{\sigma_1+\sigma_2}^{\max\set{h_1,h_2}}(\gamma).
\label{calcpol}
\end{align}

With abuse of notation, we will also use the symbol $\pol^h(\gamma)$ for a polynomial in $\pax \gamma,\dots,\pax^h \gamma$ with constant coefficients in $\R$ where the $x$-derivative of $\gamma$ appears with order at most $h$. Accordingly, we set $\pol^0(\gamma)=0$.

\section{Preliminaries on function spaces}

\subsection{Parabolic Sobolev spaces}
Let $m\in\mathbb{N}\smallsetminus \set{0}$, $p\in(1,\infty)$  and $T>0$. We denote by $W_p^{1,2m}$ the parabolic Sobolev space
\begin{equation}\label{eq:spazio-di-Sobolev-utilizzato}
W_p^{1,2m}\coloneqq W_p^1\left((0,T);L_p\left(0,1\right)\right)\cap L_p\left((0,T);W_p^{2m}\left(0,1\right)\right)
\end{equation}
endowed with the norm
\begin{equation}
\norm{}_{W_p^{1,2m}}\coloneqq\norm{}_{W_p^1\left((0,T);L_p\left(0,1\right)\right)}+\norm{}_{ L_p\left((0,T);W_p^{2m}\left(0,1\right)\right)}.
\end{equation} 
For a detailed description of the space, we refer to \cite{Amannvectorvalued,yosida}. We note that this definition is consistent with the notation used by Solonnikov \cite[Section 20]{solonnikov2} (see also \cite[Section 2]{garcke-menzel-pluda-20}). 

Elements of the space $W_p^{1,2m}$ 
are functions $f:(0,T)\to L_p(0,1)$ with $f\in L_p\left((0,T);L_p(0,1)\right)$ that possess a distributional derivative with respect to time $\partial_t f\in L_p\left((0,T);L_p(0,1)\right)$. Moreover, for almost every $t\in(0,T)$, the function $f(t)$ lies in $W_p^{2m}\left(0,1\right)$ with $2m$ spatial derivatives $\partial_x (f(t)),\ldots,\partial_x ^{2m}\left(f(t)\right)\in L_p\left(0,1\right)$. The functions $t\mapsto \partial_x^k(f(t))$, $k\in\{1,\ldots,2m\}$ lie in $L_p\left((0,T);L_p(0,1)\right)$. By means of the isometric isomorphism
\begin{equation}
J: L_p\left((0,T);L_p(0,1)\right)\to L_p\left((0,T)\times(0,1)\right), \qquad (Jf)(t,x):=\left(f(t)\right)(x),
\end{equation}
the space $W_p^{1,2m}$  can be identified with functions $f:(0,T)\times(0,1)\to\mathbb{R}$ having one distributional derivative with respect to time and $2m$ distributional derivatives with respect to space almost everywhere in $(0,T)\times(0,1)$, all of which belong to $L_p\left((0,T)\times(0,1)\right)$.

\subsection{Sobolev-Slobodeckij spaces} Once the evolution problem is formulated in the space $W_p^{1,2m}$, it is crucial to characterize the space of the initial data. This is achieved by identifying the temporal trace of $W_p^{1,2m}$.

\begin{definition}[Sobolev Slobodeckij spaces]
  Let $p\in (1,\infty)$ and $\theta\in(0,1)$. The Slobodeckij seminorm of $f\in L_p\left(0,1\right)$ is defined by
	\begin{equation}
\left[f\right]_{\theta,p}:=\left(\int_0^1\int_0^1\frac{\left\lvert f(x)-f(y)\right\rvert^p}{\lvert x-y\rvert^{\theta p+1}}\dif x\dif y\right)^{\nicefrac{1}{p}}.
	\end{equation}  
Let $s\in(0,\infty)$ be non-integer. The Sobolev-Slobodeckij space $W_p^s\left(0,1\right)$ is defined by
	\begin{equation}
	W_p^s\left(0,1\right)\coloneqq\set{f\in W_p^{\lfloor s\rfloor}\left(0,1\right):\left[f\right]_{W^s_p}:=\left[\partial_x^{\lfloor s\rfloor}f\right]_{s-\lfloor s\rfloor,p}<\infty }.
	\end{equation}    
\end{definition}

The Slobodeckij space $W^s_p(0,1)$ is a Banach space endowed with the norm 
\begin{equation}
\left\lVert\cdot\right\rVert_{W_p^s(0,1)}:=\lVert\,\cdot\,\rVert_{W_p^{\lfloor s\rfloor}}+\left[\,\cdot\,\right]_{W^s_p}.
\end{equation}

Let $p\in[1,\infty)$, $s\in(0,1)$. If $s-\frac{1}{p}>0$
the space $W^s_p(0,1)$ is a Banach algebra.
In particular, for any $f,g\in W_p^s(0,1)$, the product $fg$ belongs to $W_p^s(0,1)$ and satisfies the inequality
	\begin{equation}
	\left\lVert fg\right\rVert_{W_p^s(0,1)}\leq \kst(s,p)\left(\left\lVert f\right\rVert_{\CS^0([0,1])}\left\lVert g\right\rVert_{W_p^s(0,1)}+\left\lVert g\right\rVert_{\CS^0([0,1])}\left\lVert f\right\rVert_{W_p^s(0,1)}\right).
	\end{equation}

For $p\in (1,\infty)$, the temporal trace of $W^{1,2m}_p$ (at $t=0$) is precisely the space $W^{2m-\nicefrac{2m}{p}}_p\left(0,1\right)$;
see \cite[Theorem 4.5]{DenkSaalSeiler}.

\subsection{Embeddings}\label{emb}

A key tool for our analysis is the following compact embedding: for $p\in(1,4m+2)$ we have
\begin{equation}\label{eq:sobolevembeddingclassico}
	W_2^{2m}\left(0,1\right)\hookrightarrow W_p^{2m-\nicefrac{2m}{p}}\left(0,1\right).
\end{equation}
By the Sobolev embedding theorem, for $p\in(2m+1,\infty)$ and $\alpha\in\left(0,1-\nicefrac{(2m+1)}{p}\right)$ we have \begin{equation}
	W_p^{2m-\nicefrac{2m}{p}}\left(0,1\right)\hookrightarrow \CS^{2m-1,\alpha}\left([0,1]\right).
\end{equation}
Furthermore, combining \cite[Theorem 4.4]{DenkSaalSeiler} and \cite[Theorem 4.6.1.(e)]{Triebel}, for  $p\in(2m+1,\infty)$, $\alpha\in\left(0,1-\nicefrac{(2m+1)}{p}\right)$ we have continuous embeddings
	\begin{equation}W_p^{1,2m}\left((0,T)\times(0,1)\right)\hookrightarrow \CS^0\left([0,T];W_p^{2m-\nicefrac{2m}{p}}\left(0,1\right)\right)\hookrightarrow \CS^0\left([0,T];\CS^{2m-1+\alpha}\left([0,1]\right)\right).
	\end{equation}

\begin{lemma}\label{hölderspacetime}
    	Let $T$ be positive, $p\in(2m+1,\infty)$ and $\theta\in\left(\frac{1+\nicefrac{1}{p}}{2m-\nicefrac{2m}{p}},1\right)$. Then,
	\begin{equation}\label{emb3}
	W_p^{1,2m}\left((0,T)\times(0,1)\right)\hookrightarrow \CS^{(1-\theta)\left(1-\nicefrac{1}{p}\right)}\left([0,T];\CS^{1}\left([0,1]\right)\right)
	\end{equation}
    with continuous embedding.
\end{lemma}
\begin{proof}
First, we recall from \cite[Theorem 4.4]{DenkSaalSeiler} and \cite[Corollary~26]{Simon} that
\begin{align}
    W_p^{1,2m}\left((0,T)\times(0,1)\right)&\hookrightarrow \CS^0\left([0,T];W_p^{2m-\nicefrac{2m}{p}}\left(0,1\right)\right)\\
       W_p^{1,2m}\left((0,T)\times(0,1)\right)&\hookrightarrow \CS^{1-\nicefrac{1}{p}}\left([0,T];L_p\left(0,1\right)\right). 
\end{align}
For all $\theta\in(0,1)$ the real interpolation method (see \cite{Triebel}) gives 
\begin{equation}\label{emb1}
W_p^{\theta\left(2m-\nicefrac{2m}{p}\right)}\left((0,1)\right)=\left(L_p\left((0,1)\right), W_p^{2m-\nicefrac{2m}{p}}\left((0,1)\right)\right)_{\theta,p}
\end{equation}
and for all $f\in W_p^{2m-\nicefrac{2m}{p}}\left((0,1)\right)$ it holds
\begin{equation}\label{emb2}
    \norm{f}_{W_p^{\theta(2m-\nicefrac{2m}{p})}}\leq \kst \norm{f}^{1-\theta}_{L_p}\norm{f}^{\theta}_{W_p^{2m-\nicefrac{2m}{p}}}.
\end{equation}
For all Banach spaces $X_0$, $X_1$ and $Y$ such that $X_0\cap X_1\subset Y$ and $\left\lVert y\right\rVert_Y\leq \kst\left\lVert y\right\rVert_{X_0}^{1-\theta}\left\lVert y\right\rVert^{\theta}_{X_1}$
for all $y\in X_0\cap X_1$, one has the continuous embedding
\begin{equation}
\CS^0\left([0,T];X_1\right)\cap \CS^{\alpha}\left([0,T];X_0\right)\hookrightarrow \CS^{(1-\theta)\alpha}\left([0,T];Y\right). 
\end{equation}
In our context, applying \cref{emb1} and \cref{emb2} yields
\begin{equation}
	W_p^{1,2m}\left((0,T)\times(0,1)\right)\hookrightarrow \CS^{(1-\theta)\left(1-\nicefrac{1}{p}\right)}\left([0,T];W_p^{\theta\left(2m-\nicefrac{2m}{p}\right)}\left(0,1\right)\right).
\end{equation}
The choice  $\theta \in \left(\frac{1+\nicefrac{1}{p}}{2m-\nicefrac{2m}{p}},1\right)$ ensures that $W_p^{\theta\left(2m-\nicefrac{2m}{p}\right)}\left(0,1\right)\hookrightarrow \CS^1([0,1])$ via Sobolev embedding, concluding the proof.
\end{proof}

Finally, we recall that for $p\in(1,\infty)$, $T>0$ and $0<\alpha<\beta<1$ we have  the embedding 
\begin{equation}
\CS^\beta\left([0,T]\right)\hookrightarrow W_p^{\alpha}\left(0,T\right)
\end{equation} 
and there exist positive constants $\sigma=\sigma(\alpha,\beta)$, $\kst\left(p,\alpha,\beta\right)$ such that for all $f\in \CS^\beta\left([0,T]\right)$,
\begin{equation}\label{embeddingsobolevhölder}
\left\lVert f \right\rVert_{W_p^{\alpha}\left((0,T)\right)}\leq \kst\left(p,\alpha,\beta\right)T^{\sigma(\alpha,\beta)}\left\lVert f\right\rVert_{\CS^\beta\left([0,T]\right)}.
\end{equation}

From \cite[Theorem 5.1]{solonnikov2} we also know the continuity of the operator
\begin{equation}\label{continuity_boundary_op}
\begin{split}
 \evaluat{\pax^k\cdot}_{x=0}:  W_p^{1,2m}\left((0,T)\times(0,1)\right)&\to W_p^{\frac{1}{2m}-\frac{1}{2mp}}\left(0,T\right)\\
        \gamma &\mapsto \evaluat{\pax^{k}\gamma}_{x=0}
    \end{split}.
\end{equation}

From now on, we fix $m\geq 2$, $p\in (2m+1,4m+2)$, $\theta\in \left(\frac{1+\nicefrac{1}{p}}{2m-\nicefrac{2m}{p}},\frac{2m-1}{2m}\right)$, $\alpha=\frac{1}{2m}-\frac{1}{2mp}$ and $\beta=\left(1-\theta\right)(1-\nicefrac{1}{p})$.
Note that in this regime $\alpha<\beta$.

\medskip

Later in the discussion, we will need several estimates to be uniform in time. To this end, we endow the relevant spaces with a norm different from (but equivalent to) the standard one.

\begin{definition}
    Let $T>0$. We define $\boldsymbol{E}_T$ as the space $W_p^{1,2m}$ endowed with the norm 
	\begin{equation}
    \norm{f}_{\boldsymbol{E}_T}\coloneqq\norm{f}_{W_p^{1,2m}\left((0,T)\times(0,1)\right)}+\norm{f(0)}_{W_p^{2m-\nicefrac{2m}{p}}\left(0,1\right)}.
	\end{equation}
    Similarly, we define $\boldsymbol{F}_T$ as the space $W_p^{\frac{1}{2m}-\frac{1}{2mp}}\left(0,T\right)$ endowed with the norm 
	\begin{equation}
	\norm{b}_{\boldsymbol{F}_T}\coloneqq \norm{b}_{W_p^{\frac{1}{2m}-\frac{1}{2mp}}\left(0,T\right)}+\abs{b(0)}.
	\end{equation}
\end{definition}

A standard technique in the literature for obtaining estimates independent of the time interval is the use of extension operators. Let $0<T<T_1$. Then, there exist  linear operators
	\begin{align}
	    E_i:& W_p^{1,2m}\left((0,T)\times(0,1)\right)\to W_p^{1,2m}\left((0,T_1)\times(0,1)\right)\\
        E_b:& W_p^{\nicefrac{1}{2m}-\nicefrac{1}{2mp}}\left(0,T\right)\to W_p^{\nicefrac{1}{2m}-\nicefrac{1}{2mp}}\left(0,T_1\right)
	\end{align}
	such that for all $f\in W_p^{1,2m}\left((0,T)\times(0,1)\right)$, $\left(E_i f\right)_{|(0,T)}=f$, for all $g\in  W_p^{\nicefrac{1}{2m}-\nicefrac{1}{2mp}}\left(0,T\right)$, $\left(E_b g\right)_{|(0,T)}=g$. Crucially, these operators satisfy
	\begin{align}
	\norm{E_i f}_{W_p^{1,2m}\left((0,T_1)\times(0,1)\right)}& \leq \kst(p,T_1)\norm{f}_{\boldsymbol{E}_T},\\
    \norm{E_b g}_{ W_p^{\nicefrac{1}{2m}-\nicefrac{1}{2mp}}\left(0,T_1\right)}&\leq \kst(p,T_1)\norm{ g}_{\boldsymbol{F}_T},\label{estensione_controllata}
	\end{align}
where the constant $\kst(p,T_1)$ remains bounded as $T\to 0$.

\begin{lemma}[Uniform embedding]\label{uniformcalphac1}
Let $T_1>0$ and $\theta\in\left(\frac{1+\nicefrac{1}{p}}{2m-\nicefrac{2m}{p}},1\right)$, $\beta=(1-\theta)\left(1-\nicefrac{1}{p}\right)$. Then, there exist constants $\kst$
depending on $T_1,p,\theta$ such that for all $T\in (0,T_1]$ and all $f\in W_p^{1,2m}\left((0,T)\times(0,1)\right)$ the following uniform estimates hold:
	\begin{align}
	    \norm{f}_{\CS^0\left([0,T];\CS^{2m-1}\left([0,1]\right)\right)}&\leq \kst(p)\norm{f}_{\CS^0\left([0,T];W_p^{2m-\nicefrac{2m}{p}}\left(0,1\right)\right)}\leq \kst\left(T_1,p\right)
    \norm{f}_{\boldsymbol{E}_T},\label{stima-buc}\\
    \norm{f}_{\CS^{\beta}\left([0,T];\CS^1\left([0,1]\right)\right)}&\leq \kst\left(T_1,p,\theta\right)\norm{f}_{\boldsymbol{E}_T}.\label{stima-hölder}
    \end{align}
\end{lemma}    
\begin{proof}
 Let $f\in W_p^{1,2m}\left((0,T)\times(0,1)\right)$ and let $E_i$ be the extension operator defined above.
 Then $E_if$ lies in $W_p^{1,2m}\left(\left(0,T_1\right)\times(0,1)\right)$ and 
 \begin{align}
 &\norm{f}_{\CS^0\left([0,T];W_p^{2m-\nicefrac{2m}{p}}\left(0,1\right)\right)}\leq \norm{E_if}_{\CS^0\left(\left[0,T_1\right];W_p^{2m-\nicefrac{2m}{p}}\left(0,1\right)\right)}\\
 &\leq \kst\left(T_1,p\right)\norm{E_if}_{W_p^{1,2m}\left(\left(0,T_1\right)\times(0,1)\right)}\leq \kst\left(T_1,p\right)\norm{f}_{\boldsymbol{E}_T}.
 \end{align}
 The second estimate \cref{stima-hölder} follows by applying the same extension argument to the interpolation result in \cref{emb3}.
\end{proof}

\section{Derivation of the flow}

The most recurring functional in this paper is the following energy functional on curves.
\begin{definition}
    Let $m\geq 1$ and $\gamma \in H^m((0,1),\R^2)$. We define the energy 
    \begin{equation}\label{eq:main-functional}
        \E_m(\gamma) \coloneqq\int_{\gamma}\frac{\abs{\pas^m\gamma}^2}{2}\dif \gamma.
    \end{equation}
\end{definition}
\subsection{First variation}
We begin by considering a single curve. Let $\gamma\in H^m((0,1), \R^2)$ be the curve and $\varphi\in\CS^\infty([0,1], \R^2)$. An external variation of $\gamma$ with respect to $\varphi$ is the one-parameter semigroup of curves given by $\gamma + \varepsilon \varphi$. For simplicity, we will often say that $\varphi$ is an external variation of $\gamma$. The first variation of $\E_m(\gamma)$ in direction $\varphi$ at $\gamma$ is defined as
\begin{equation}
    \delta_\varphi \E_m(\gamma)\coloneqq \evaluat{\frac{\dd}{\dd \varepsilon} \E_m(\gamma + \varepsilon \varphi)}_{\varepsilon=0}.
\end{equation}

Variations of derivatives of $\gamma$ appear in the process of the explicit computation of $\delta_{\varphi} \E_m$. If $\gamma$ is parametrized with $x \in [0,1]$, then $\delta_{\varphi}$ commutes with $\pax$. For example, $\delta_\varphi \gamma = \varphi$ but also $\delta_{\varphi} \pax^j \gamma = \pax^j \delta_{\varphi} \gamma$. The variation of $\pas^j \gamma$ is more delicate. Indeed, the arclength parameter of $\gamma$ is different from the arclength parameter of $\gamma +\varepsilon \varphi$. The following lemma gives the computation rule.
\begin{lemma}\label{lem:commuteDpas} Let $\varphi$ be an external variation of $\gamma$. Then, for every $j\in\mathbb{N}\smallsetminus\{0\}$, it holds
\begin{equation}\label{eq:commuteDpas}
    \delta_\varphi \pas^j \gamma = \pas \delta_\varphi \pas^{j-1}\gamma - \ip{\pas \gamma, \pas \varphi} \pas^{j} \gamma.
\end{equation}
\end{lemma}

\begin{proof} 
Since $\pas = \pax / \abs{ \pax \gamma} $, one gets
\begin{align}
    \delta_\varphi ( \pas^j \gamma ) &= \delta_\varphi\left(\frac{1}{\abs{\pax \gamma}}\right)\pax \pas^{j-1} \gamma + \frac{1}{\abs{\pax \gamma}} \delta_\varphi ( \pax \pas^{j-1}\gamma)\\ &=- \frac{1}{ \abs{\pax \gamma}^3} \ip{\pax \varphi , \pax \gamma } \pax \pas^{j-1} \gamma + \frac{\pax}{\abs{\pax \gamma}} \delta_\varphi ( \pas^{j-1}\gamma)\\
    &= - \ip{\pas \varphi , \pas \gamma }\pas^j\gamma + \pas \delta_\varphi (\pas^{j-1} \gamma).\qedhere
\end{align}
\end{proof}

Moreover, since $\dd \gamma = \abs{\pax \gamma} \dd x $
\begin{equation}\label{eq:variationdifs}
    \delta_\varphi (\dd \gamma ) = \ip{\pas \varphi , \pas \gamma }\dif \gamma.
\end{equation}

We can now compute the first variation of our functional in direction $\varphi$ at $\gamma$.

\begin{proposition}\label{prop:first_variation_energy}
Let $m\geq 1$, $\gamma\in H^{2m}((0,1),\mathbb{R}^2)$ and $\E_m$ be the functional defined by \cref{eq:main-functional}. For $\varphi$ external variation of $\gamma$ we have
\begin{align}\label{eq:var-prima-base}
    \delta_\varphi \E_m(\gamma)&= -\int_{\gamma}\ip{\vec{V}_m(\gamma), \varphi} \dif \gamma  +\evaluat{\ip{\psi^0_m(\gamma), \varphi}}_{\nodes}+\sum_{j=0}^{m-2} (-1)^j  \evaluat{\ip{\delta_\varphi\pas^{m-j-1} \gamma, \pas^{m+j}\gamma }}_{\nodes},
\end{align}
    where $\vec{V}_m$ and $\psi^0_m$ are operators depending on $\gamma$ defined as follows
        \begin{align}
          \label{eq:normalvelocityoperator}
        \vec{V}_m(\gamma)&\coloneqq(-1)^{m+1} \pas^{2m}\gamma - \sum_{j=1}^{m-1} (-1)^j \pas(\ip{\pas^{m-j}\gamma, \pas^{m+j}\gamma} \pas \gamma) - \frac{1}{2} \pas \left[\abs{\pas^m \gamma}^2 \pas \gamma\right],\\
        \psi^0_m(\gamma) &\coloneqq (-1)^{m-1} \pas^{2m-1} \gamma - \sum_{j=1}^{m-1} (-1)^j \ip{\pas^{m-j} \gamma,\pas^{m+j}\gamma} \pas \gamma - \frac{1}{2}\abs{\pas^m \gamma}^2\pas \gamma.\nonumber
    \end{align}
\end{proposition}

\begin{proof}
 Applying \cref{eq:variationdifs} and integrating by parts, we have
    \begin{align}
        \delta_\varphi \E_m(\gamma) \overset{\cref{eq:variationdifs}}&{=}\int_{\gamma} \frac12\delta_\varphi  \abs{\pas^{m}\gamma}^2 +  \frac12\abs{\pas^{m}\gamma}^2\ip{\pas \varphi, \pas \gamma}\dif \gamma\\
          &= \int_{\gamma} \ip{ \delta_\varphi \pas^{m}\gamma , \pas^{m}\gamma } - \frac12\ip{\varphi, \pas (\abs{\pas^m \gamma}^2\pas \gamma)}\dif \gamma + \frac12\evaluat{{{\abs{\pas^{m}\gamma}^2}} \ip{ \varphi, \pas \gamma}}_{\nodes}.
    \end{align}
    Now we iterate the following procedure $m$ times. Isolating the integral part containing $\delta_\varphi$, applying \cref{lem:commuteDpas} and integrating by parts
    \begin{align}
        \int_{\gamma} \ip{ \delta_\varphi \pas^{m-j}\gamma , \pas^{m+j}\gamma }\dif \gamma &= \int_{\gamma} \ip{\pas \delta_\varphi \pas^{m-j-1} \gamma, \pas^{m+j}\gamma}-\ip{\pas^{m-j}\gamma, \pas^{m+j} \gamma}  \ip{\pas \varphi, \pas \gamma }\dif \gamma\\
        &=-\int_{\gamma} \ip{\delta_\varphi \pas^{m-j-1} \gamma, \pas^{m+j+1}\gamma}- \ip{\varphi, \pas \left(\ip{\pas^{m-j}\gamma, \pas^{m+j} \gamma}  \pas \gamma \right)}\dif \gamma\\
        &\qquad + \evaluat{\ip{\delta_\varphi \pas^{m-j-1} \gamma, \pas^{m+j}\gamma}}_{\nodes} - \evaluat{\ip{\pas^{m-j}\gamma, \pas^{m+j} \gamma}\ip{ \varphi,  \pas \gamma}}_{\nodes}
    \end{align}
\end{proof}

\begin{remark}
 Observe that $\vec{V}_m(\gamma)$ is purely normal. Indeed, every tangential part cancels out, and we have 
\begin{equation}
    \vec{V}_m(\gamma) = (-1)^{m+1} (\pas^{2m} \gamma)^\perp - \sum_{j=0}^{m-1} (-1)^j\ip{\pas^{m-j} \gamma , \pas^{m+j} \gamma}\pas^2 \gamma - \frac{1}{2} \abs{\pas^m \gamma}^2 \pas^2 \gamma.
\end{equation}
For this reason, we will denote
\begin{equation}
\label{eq:normalvelocity}
    V_m\nu=\vec{V}_m(\gamma)
\end{equation}
omitting the subscript $m$ whenever there is no risk of confusion.  
\end{remark}

Previous computations can be extended to a network of curves $\gamma = (\gamma^1, \ldots,\gamma^N)$. Given $\varphi=(\varphi^1, \ldots, \varphi^N)$ the variation of $\gamma$ with respect to $\varphi$ is the one-parameter
semigroup of networks given by $\gamma + \varepsilon \varphi$ where each curve $\gamma^i$ varies in direction $\varphi^i$. Observe that a variation can modify the topology of $\gamma $ in principle. Given that 
\begin{equation}
    \E_m(\gamma) = \sum_{j =1}^{N} \E_m(\gamma^j)
\end{equation}
its first variation in direction $\varphi$ will be
\begin{equation}
    \delta_{\varphi}\E_m(\gamma)  = \sum_{j=1}^{N} \delta_{\varphi^j}\E_m(\gamma_j). 
\end{equation}

Regardless of whether a single curve or a network is considered, for the velocity $V\nu$ in \cref{eq:normalvelocity} to be interpreted as the energy gradient with respect to the $L^2$-inner product, the boundary terms in the variation must vanish:

\begin{equation}\label{eq:boundary_couple}
 \evaluat{\ip{\psi^0_m(\gamma), \varphi}}_{\nodes}+
    \sum_{j=0}^{m-2} (-1)^j  \evaluat{\ip{\delta_\varphi\pas^{m-j-1} \gamma, \pas^{m+j}\gamma }}_{\nodes}=0.
\end{equation}

\subsection{The effective geometric structure}\label{effective-structure} So far, computations have been performed by fixing a specific parametrization of the curve -- namely, the arc-length parametrization. The energy $\E_m$ depends only on the length of the curve, the curvature and its derivative;  consequently, the variation of the energy should depend solely on their variations. However, at first glance, the boundary terms \cref{eq:boundary_couple} appear to depend on the variation of the $2m$ parameters $ \{ \pas^{j}\gamma\}$, $ j \in \set{0, \ldots, m-1}$. The contribution related to the position of the curve in space corresponds to the boundary component $\ip{\psi^0_m(\pas^m \gamma), \varphi}$. Among the remaining $2m-2$ parameters, the only significant ones are those related to geometric quantities.  The apparent redundancy of the free parameters arises from the fact that the natural relations between the derivatives of $\gamma$, induced by the parametrization, have not been imposed yet. 
Once these relations are enforced, this redundancy disappears and the expression in \cref{eq:boundary_couple} simplifies. We now turn to the derivation of these relations.

\medskip

Under arc-length parametrization, we have $\abs{\pas \gamma} = 1$, which implies $\partial^{j-1} \abs{\pas\gamma} = 0$ for $j \in \set{2, \ldots, m-1}$ providing $m-1$ constraints among the derivatives of $\gamma$. Computing the variation of these identities, we get
\begin{equation}
    0=\frac{1}{2}\delta_\varphi \pas^{j-1} \left(\abs{\pas \gamma}^2\right) =\sum_{k=1}^{j} \binom{j-1}{k-1} \ip{\pas^{j-k+1}\gamma,\delta_\varphi \pas^{k} \gamma}.
\end{equation}
Let $b_{jk} \in \R^{2\times 2}$, $ j,k \in \set{1, \ldots, m-1}$ be defined as
\begin{equation}
    b_{jk} \coloneqq \begin{cases}[@{\kern1mm}c@{\kern5mm}l]\displaystyle\binom{j-1}{k-1} (\pas \gamma) (\pas^{j-k+1} \gamma)^\dagger & \text{for } k\leq j\\[.3cm]
    0 &\text{otherwise}
    \end{cases}.
\end{equation}
We recursively define for $ \alpha \in \set{1, \ldots, m-1}$
\begin{align}
    b^{(m-1)}_{jk} &\coloneqq \delta_{jk}\id_{\R^2}\label{eq:b0} \\
    b^{(\alpha-1)}_{jk} &\coloneqq b^{(\alpha)}_{jk}- b_{\alpha j}^\dagger b_{\alpha k}^{(\alpha)}.\label{eq:ba}
\end{align}

\begin{lemma}\label{lem:structureb}
    The following properties hold
    \begin{enumerate}
        \item $b^{(\alpha)}_{jk}=\delta_{jk}\id_{\R^2}$ for all $k\leq \alpha$; 
        \item $b^{(\alpha)}_{jk} = (\id_{\R^2}- b_{jj}^\dagger)b^{(j)}_{jk}$ for all $\alpha < j$.
    \end{enumerate}
\end{lemma}

\begin{proof}
    (1) If $\alpha=m-1$ the property is true by \cref{eq:b0}. Assume the property is true for some $\alpha$ and pick $k\leq \alpha-1$. By \cref{eq:ba} we have
    \begin{equation}
        b^{(\alpha-1)}_{jk} =  b^{(\alpha)}_{jk}- b_{\alpha j}^\dagger b_{\alpha k}^{(\alpha)} = \delta_{jk}\id_{\R^2} - b_{\alpha j}^\dagger\delta_{\alpha k} \id_{\R^2} =\delta_{jk}\id_{\R^2},
    \end{equation}
    since $k\leq \alpha-1 $ implies $k\neq \alpha$.

    (2) Observe that $b_{\alpha j} = 0$ for $\alpha<j$, hence 
    \begin{equation}
        b^{(\alpha)}_{jk} = b^{(j-1)}_{jk} = b^{(j)}_{jk} - b^\dagger_{jj} b^{(j)}_{jk}.\qedhere
    \end{equation}
\end{proof}

We now set
\begin{equation}
    b^\perp_{jk}  \coloneqq b^{(0)}_{jk} \qquad \text{for } j,k \in \set{1, \ldots, m-1}.
\end{equation}

\begin{remark}\label{rmk:bisaprojection}
In view of \cref{lem:structureb}, the matrix $b^\perp$ projects onto the normal component. Indeed, $b_{jj} = (\pas \gamma)(\pas \gamma)^\dagger$ for all $j \in \set{1, \ldots, m-1}$, thus
    \begin{equation}
        \ip{b_{jk}^\perp z_k, \pas \gamma} =\ip{b^{(j)}_{jk} z_k, (\id_{\R^2} - b_{jj})\pas \gamma} = 0
    \end{equation}
for every $z \in (\R^{2})^{m-1}$. In particular, $b_{jk}^\perp z_k = \ip{b_{jk}^\perp z_k,\nu} \nu$.
\end{remark}

\begin{lemma}\label{lem:complement_of_b}
    Let $y,z \in (\R^2)^{m-1}$ with $y \in \ker b$, then $\ip{z,y} = \ip{b^\perp z,y}$. Moreover, $\ker b^\perp = (\ker b)^\perp$.
\end{lemma}

\begin{proof}
    Let $y \in \ker b$. Then
    \begin{align}
        \ip{z,y} &= \sum_{j,k=1}^{m-1} [b^{(m-1)}_{jk} z_k]^\dagger y_j-  [b^{(m-1)}_{(m-1)k} z_k ]^\dagger b_{(m-1)j}y_j= \sum_{j,k=1}^{m-1} [b^{(m-2)}_{jk} z_k]^\dagger y_j\\
        &=\ldots = \sum_{j,k=1}^{m-1} [b^\perp_{jk} z_j]^\dagger y_k = \ip{b^\perp z,y}.
    \end{align}
    In particular, $\ker b^\perp \subseteq (\ker b)^\perp$. We observed in \cref{rmk:bisaprojection} that $\rnk b^\perp\leq m-1$. On the other hand, by \cref{lem:structureb} $b^\perp_{jj} \nu = (\id_{\R^2} - b_{jj}^\dagger)\nu=\nu$ for all $ j \in \set{1, \ldots, m-1}$. Then, $\dim \ker b^\perp \leq m-1$ and $\ker b^\perp = (\ker b)^\perp$.
\end{proof}

    \begin{definition} \label{def:definition_of_psi}
    We define
        \begin{equation}\label{psi}
            \psi^j_m(\gamma)\coloneqq\sum_{k=1}^{m-1} (-1)^{m-k-1}b^\perp_{jk}\pas^{2m-k-1}\gamma.
        \end{equation}
    \end{definition}

    \begin{remark}\label{rmk:boundary_couple_bis}
        Since $\delta_\varphi\pas^k \gamma \in \ker b$, the higher order boundary terms in \cref{prop:first_variation_energy} can be rewritten as 
        \begin{equation}\label{eq:boundary_couple_bis}
             \sum_{j=0}^{m-2} (-1)^j  \evaluat{\ip{\delta_\varphi\pas^{m-j-1} \gamma, \pas^{m+j}\gamma }}_{\nodes} = \sum_{j=1}^{m-1} \evaluat{\ip{\delta_\varphi \pas^j \gamma ,\psi^j_m(\gamma)}}_{\nodes}
        \end{equation}
    In particular, since $\psi^j_m$ is purely normal by \cref{rmk:bisaprojection} we reduced the apparent $2m-2$ parameters into the  $m-1$ significant geometric ones.
    \end{remark}

    \begin{lemma}     The structure of the operators $b_{jk}^\perp$ in terms of the derivatives of $\gamma$ is as follows:
        $b_{jk}^\perp = \pol^{k-j+1}_{k-j}(\gamma)$.
    \end{lemma}

    \begin{proof}
        Observe that $b^{(m-1)}_{jk} = \delta_{jk}\id_{\R^2}=\pol^{k-j+1}_{k-j}(\gamma)$. Assume now that the property is true for some $\alpha$. Then,
        \begin{equation}
            b^{(\alpha-1)}_{jk} =  b^{(\alpha)}_{jk}- b_{\alpha j}^\dagger b_{\alpha k}^{(\alpha)} = \pol^{k-j+1}_{k-j}(\gamma) -\pol^{\alpha-j+1}_{\alpha-j}(\gamma)\pol^{k-\alpha+1}_{k-\alpha} (\gamma)=\pol^{k-j+1}_{k-j}(\gamma)
        \end{equation}
        since either $\max\set{k-\alpha+1, \alpha- j +1} \leq k-j+1$, $b_{\alpha k}^{(\alpha)}=0$ or $b_{\alpha j}=0$.
    \end{proof}

    \begin{remark}
    In the case of the Willmore functional/ elastic flow ($m=2$) the above procedure is redundant. For $m=2$, the only boundary term involving a derivative of the variation is $\delta_\varphi \pas \gamma$, which is already purely normal.
    \end{remark}

\subsection{Boundary conditions} 
Following the simplification procedure detailed in the previous section, the velocity $V\nu$ in \cref{eq:normalvelocity} can be understood as the $L^2$-gradient of the energy $\E_m$ provided that the boundary terms vanish.

We want to work in several different classes of networks and curves, where we impose conditions on the topology of the network and possibly on higher-order geometric quantities (angles at the junctions and curvature, for instance). The imposed boundary conditions on a specific point $P\in \nodes$ are not influenced by our choices at the other points. So we can choose any $d$-junction $P$ and work on it separately from the others.

\medskip

In \cref{rmk:boundary_couple_bis}, we notice that higher-order terms have a precise structure. This is because the simplifications in the previous section, which resulted from the choice of parametrization, relate to the derivatives of $\gamma$.
The higher-order conditions are of a geometric nature and scalar-valued. At each $d$-junction, $P$ one can impose at most $d$ conditions for each order $i \in \set{1, \ldots, m-1}$. We denote their expressions by $B_{ij}(\gamma)$.

\medskip 

On the other hand, the choice of the zero-order  Dirichlet boundary condition will be made to preserve the network's structure throughout the evolution; thus, we must enforce a topological constraint on $\gamma$ itself rather than on its derivatives (see \cref{sec:topological} for examples). This constraint concerns the images of the endpoints under the map $\gamma$. Consequently, it is both analytic in nature -- since it involves the map itself -- and vector-valued, consistently with the expression in \cref{prop:first_variation_energy}. We will denote their expressions by $B_{0j}(\gamma)$.

The full set of Dirichlet boundary conditions we impose at $P$ is given by $\evaluat{B_{ij}(\gamma)}_P =0$ for $ i \in \set{0, \ldots, m-1}$ and $j \in \set{1, \ldots, d}$. 
Consequently,
\begin{equation}\label{eq:admissibility_system}
    \evaluat{\delta_{\varphi}B_{ij}(\gamma)}_{P} = 0\in \R^{d(m+1)}
\end{equation}
gives the constraints on $\varphi$ and its derivatives up to order $i$.

\medskip

We assume that the imposed conditions are independent. Clearly, the more conditions we impose, the more \cref{eq:admissibility_system} enforces relations on $\varphi$ and its derivatives, thereby restricting the class of variations that preserve the boundary conditions. In the extreme case where every $B_{ij}$ is nontrivial and independent, the admissible variations must vanish at $P$ together with all their derivatives, thus ensuring that the terms in \cref{eq:boundary_couple_bis} vanish.

To treat a more general setting, we allow for cases where only a subset of the conditions $\evaluat{B_{ij}(\gamma)}_P$ is imposed. For simplicity of notation, we will still retaining all the indices so that the order of the conditions is preserved, in this case the expression for $B_{ij}(\gamma)$ is set to zero so that  $\evaluat{B_{ij}}_{P}=0$ is trivially satisfied.
In this case, to ensure that \cref{eq:boundary_couple_bis} still vanishes, we need to impose additional restrictions on $\psi^l$. Morally speaking, $\psi^l$ must belong to a suitable orthogonal complement of the space of admissible variations with respect to the scalar product appearing in \cref{eq:boundary_couple_bis}.

This requirement leads to higher-order conditions involving $\psi^l$. We denote these conditions by
\begin{equation}
    B_{(2m-1-i)j}(\gamma,\psi) \in \mathbb{R}^{d(m+1)}
\end{equation}
for $ i \in \set{0, \ldots, m-1}$ and $j \in \set{1, \ldots, d}$. The index $(2m-1-i)$ highlights that $B_{(2m-1-i)j}$ contains derivatives of $\gamma$ up to order $(2m-i-1)$.  Moreover, $B_{(2m-1-i)j}$ is the complementary condition to $B_{ij}$, hence the $0$-order topological condition is paired with the highest-order condition. This is also reflected in the structure of the conditions: the only vector-valued expressions are the ones for $i=0$.

We denote a specific set of conditions by
\begin{equation}
 \evaluat{\Bcnd(\gamma,\psi)}_P = 0 \in \mathbb{R}^{2d(m+1)} .
\end{equation}
Here again, we keep all indices to maintain their correspondence with the order of the condition.

\begin{definition}\label{def:admissible-variation}
    A variation $\varphi$ for $\gamma$ is admissible for a given set of conditions $\evaluat{\Bcnd(\gamma, \psi)}_{P}=0$ if it satisfies \cref{eq:admissibility_system} for all $i \in \set{0, \ldots, m-1}$, and for all $ j \in \set{1, \ldots, d}$ at the $d$-junction $P$.
\end{definition}

We now assume for simplicity that $P$ is a $d$-junction and $\gamma=(\gamma^1,\ldots, \gamma^d)$ is the star network with center $P$. We also assume there are no loops. 
We define the constraints on the class of possible conditions by defining the structure of $\delta_\varphi B_{ij}$.

\subsection{Topological boundary conditions}\label{sec:topological}
\newcommand{\ndA}{A}
We now introduce the admissible boundary conditions of order zero. These conditions are described by a matrix $a = (a_{ij})$, where $a_{ij} \in \mathbb{R}^{2 \times 2}$. 

To detect the suitability of a given condition, we introduce the operator $\ndA:(\S^{1})^d \to \R^{2d\times 2d}$ defined as 
\begin{equation}\label{eq:nondegeneracy_system}
    \ndA(\nu)\coloneqq a^\dagger a + \diag(\nu) \diag(\nu)^\dagger,
\end{equation}
where $\diag(\nu)$ is the $d\times d$ matrix whose diagonal entries are $\nu^1, \ldots, \nu^d \in \R^2$.

\begin{definition}[Topological boundary conditions]\label{def:admissible_dirichlet_0}
  The conditions $B_{0i}(\gamma)$ for $ i \in \set{1, \ldots, d}$ are admissible if their variation can be written as
    \begin{align}\label{eq:admissible_dirichlet_0}
        \delta_\varphi B_{0i} = \sum_{j=1}^d a_{ij}\delta_{\varphi } \gamma^j,
    \end{align}
    where $a_{ij} \in \R^{2 \times 2}$ fulfill the property
    \begin{equation}
        \set{\nu \in (\S^1)^d : \det A(\nu) \neq 0} \neq \varnothing.
    \end{equation}
\end{definition}

\begin{remark}
Although this condition may appear abstract, it merely requires the existence of at least one choice of unit normal vectors to the curves at the junction for which the associated matrix \(A(\nu)\) is nonsingular.

Such a condition is in fact quite natural and readily satisfied in all classical situations. For example, in the case of a fixed endpoint, one has \(\det A(\nu)\neq 0\) for every \(\nu\in\mathbb{S}^1\). At a junction, the condition is also satisfied: indeed, the singular set of \(A(\nu)\) is characterized by the equation $\pm \nu^1 = \ldots = \pm\nu^d$. \(A(\nu)\) is nonsingular whenever the tangent directions of the curves are not all aligned. More detailed explanations and examples will be given below.
\end{remark}

Given a network $\gamma$ and a $d$-junction $P$, assume that the unit normal vectors $(\nu^1, \ldots, \nu^d)=\nu$ at the junction satisfy $\det\ndA=\det\ndA(\nu)\neq0$. This implies $\dim (\ker a) \leq d$. Indeed, $\rnk (\ndA) =2d$ and $\rnk(\diag(\nu)) = d $. For any $\xi \in \ker a$, we get
\begin{equation}
    \xi = \ndA^{-1}\ndA\xi = \ndA^{-1}\diag(\nu) \diag(\nu)^\dagger \xi
\end{equation}
Consequently, each element in $\ker a$ is in $\mathrm{Im}\ndA^{-1}\diag(\nu)$. Imposing that \cref{eq:admissible_dirichlet_0} vanishes is equivalent to saying that any admissible variation $\varphi$ belongs to $\ker a$. The previous argument shows that each admissible variation $\varphi$ is fully determined by its normal components. Moreover, 
$\varphi \in \ker a$ is equivalent to $\diag(\nu)^\dagger \varphi \in \ker(a \ndA^{-1} \diag(\nu)) $. Indeed,
\begin{align}
    \rnk(a\ndA^{-1}\diag(\nu)) =
  \rnk(\ndA^{-1}\diag(\nu))- \dim (\ker a ) = d - \dim (\ker a )
\end{align}
since the image of $\ndA^{-1}\diag(\nu)$ contains $\ker a$.

Let $\pi: (\R^2)^d \to (\R^2)^d$ be the orthogonal projection onto $\ker a$. We then have
\begin{align}
    \sum_{j=1}^d \ip{\psi^0(\gamma^j), \delta_{\varphi } \gamma^j} &=  \sum_{j,l,k=1}^d \ip{\pi_{lj}\psi^0(\gamma^j), \ndA^{-1}_{kl}\ip{\nu^k, \delta_{\varphi } \gamma^k}\nu^k}\\
    &=\sum_{j,l,k=1}^d \ip{ \ndA^{-1}_{lk}\pi_{lj}\psi^0(\gamma^j),\nu^k}\ip{\nu^k, \delta_{\varphi } \gamma^k}.
\end{align} 
The above formula shows that it is enough to constrain $\psi^0(\gamma)$ to lie in $\ker (\diag(\nu)^\dagger \ndA^{-1} \pi)$. However, in order to adopt it as a condition, we must ensure that it is not weaker than the requirement $\psi^0(\gamma) \in \ker \pi$, that is
\begin{equation}
    \rnk  (\diag(\nu)^\dagger \ndA^{-1} \pi)  = \rnk (\pi ) = \dim (\ker a).
\end{equation}
This identity is true if $\mathrm{Im}\pi\cap\ker(\diag(\nu)^\dagger \ndA^{-1} ) = \set{0}$. To prove it, assume $\xi\in \mathrm{Im}\pi\cap\ker(\diag(\nu)^\dagger \ndA^{-1} )$. Since $\mathrm{Im}\pi  = \ker a $, then $\xi = \ndA^{-1} \diag(\nu)\diag( \nu)^\dagger \xi$, hence 
\begin{equation}
    \abs{\xi}^2  = \ip{ \xi,\ndA^{-1}  \diag(\nu)\diag(\nu)^\dagger \xi} = \ip{\diag(\nu)^\dagger\ndA^{-1} \xi, \diag(\nu)^\dagger  \xi} =0.
\end{equation}

We are now ready to introduce our maximal order condition.

\begin{definition}[Maximal order condition]\label{def:Neumann_highest}
    Let $\pi_{jl}$ be the orthogonal projection onto $\ker a $ and assume that $\det \ndA(\nu) \neq 0$. We define the maximal order condition as
    \begin{equation}
        B_{(2m-1)j}(\gamma, \psi) = \sum_{i,k=1}^{d} \ip{ \ndA(\nu)^{-1}_{ji}\pi_{ik}\psi^0(\gamma^k), \nu^j}
    \end{equation}
\end{definition}

\subsubsection{Examples of topological conditions}\label{sec:top-bc}
As mentioned previously, we always impose a zero-order Dirichlet boundary condition. We distinguish the cases of endpoints and junctions. 

\paragraph{Fixed endpoints}  Let  $\gamma^1 \in \gamma$, let  $P$ be an endpoint of $\gamma^1$, and $P_0 \in \R^2$ a fixed point. The fixed point condition is given by $\evaluat{\gamma^1}_{P}=P_0$. If we write $B_{01}= \gamma^1 - P_0$, the class of admissible variations must be constrained with $a_{11} = \id_{\R^2}$. The condition $\det\ndA(\nu)\neq 0$ is satisfied for every $\nu \in \S^1$ since $a_{11}$ is nonsingular. Moreover, $\pi_{11}=0\in \R^{2\times 2}$

\paragraph{Concurrency condition} Let $P$ be a $d$-junction. The concurrency condition says that a $d$-junction remains a $d$-junction along the flow. This condition can be expressed as $B_{0j} =\gamma^{j}-\gamma^{j-1}$ for all $j>1 $ and $B_{01} =0$ (as explained before, $B_{01} =0$ means that we are not imposing a condition, $B_{01}$ is there only to retain the order). Then,
\begin{equation}
    a_{ij} =\begin{cases}[@{\kern.1cm}l@{\kern.5cm}l]
        \id_{\R^2} & \text{if $i=j$ and $i>1$,}\\
        -\id_{\R^2} & \text{if $i=j-1$ and $i>1$,}\\
        0 & \text{otherwise.}
    \end{cases}
\end{equation}
In this case, $a = (a_{ij})$ is singular; indeed, the rank of the matrix is $2d-2$. The  condition $\det \ndA(\nu)\neq 0$ is equivalent to asking $\lspan\set{\nu^1, \ldots, \nu ^d} = \R^2$. Indeed, $\ker a = \set{(\xi, \ldots, \xi) : \xi \in \R^2}$.  Therefore, $\ker\diag(\nu)^\dagger\cap \ker a = \ker \nu^\dagger$ which is trivial if and only if  $\lspan\set{\nu^1, \ldots, \nu ^d} = \R^2$. Moreover, $\pi_{ij} = \frac{1}{d} \id_{\R^2}$.

\paragraph{Endpoint constrained on a straight line} Let  $\gamma^1 \in \gamma$, let  $P$ be an endpoint of $\gamma^1$. Let $\Gamma$ be a straight line and $\nu_\Gamma$ the normal unit vector to it. Assume without loss of generality that $0 \in \Gamma$, then $B_{01} = \ip{\gamma^1, \nu^{\vphantom{\dagger}}_{\Gamma}}\nu^{\vphantom{\dagger}}_{\Gamma}$. Hence $\varphi^1$ is constrained with $a_{11} = \nu_{\Gamma}^{\vphantom{\dagger}} \nu_{\Gamma}^\dagger$. Arguing as for the concurrency condition, $a_{ij}$ is admissible if $\lspan\{\nu^{\vphantom{\dagger}}_{\Gamma}, \nu^1\} = \R^2$. Clearly, $\pi_{11} = \tau^{\vphantom{\dagger}}_{\Gamma} \tau_{\Gamma}^\dagger$, where $\tau^{\vphantom{\dagger}}_{\Gamma}$ is the unit vector tangent to $\Gamma$.
\medskip

While not explicitly treated in this paper, the results for a straight line $\Gamma$ could likely be extended to curves or attachment conditions on a surface.

\subsection{Dirichlet-Neumann boundary conditions}

\begin{definition}[Admissible Dirichlet boundary conditions]\label{def:admissible_dirichlet}
  The conditions $B_{ij}(\gamma) = \pol^i_{i-1}(\gamma)$ for $ i \in \set{1, \ldots, m-1}$ and $ j \in \set{1, \ldots, d}$ are admissible Dirichlet boundary conditions if their variation can be written as
    \begin{align}
        \delta_\varphi B_{ij} = \sum_{k=1}^{m-1} \sum_{l=1}^{d} a_{ijkl}(\gamma)\ip{\delta_\varphi \pas^k \gamma^l,\nu^l},
    \end{align}
    where the coefficients $a_{ijkl}=a_{ijkl}(\gamma)\in\R$  satisfy  the following properties:
\begin{itemize}
    \item \underline{\smash{non-degeneracy}}: $B_{ij}=0$ whenever $a_{ijij}=0$;
    \item \underline{\smash{well-ordering}}:  $a_{ijkl}=0$ for $kd +l > id +j$;
    \item \underline{\smash{normalization}}: $a_{ijij} \in \{0,1\}$;
    \item \underline{\smash{homogeneity}}: $a_{ijkl} \in \pol^{i-k+1}_{i-k}$.
\end{itemize}
\end{definition}
\begin{remark}
    Observe that the only truly restrictive property is homogeneity. Non-degeneracy is required to prevent a condition of order $i$ from being equivalently rewritten without the $i$-th derivative of the curve. Well-ordering serves a similar purpose to non-degeneracy: the condition $B_{ij}$ is of order $i$ and provides non-trivial information about the $j$-th component of the curve. Finally, since homogeneity implies that $a_{ijil}$ is constant, normalization can always be achieved by dividing by this constant.
\end{remark}

Given $a_{ijkl}= a_{ijkl}(\gamma)$ for $i,k \in \set{1, \ldots, m-1}$ and $ j,l \in \set{1, \ldots, d}$ we recursively define 
\begin{align}
    a^{(m-1,d)}_{ijkl}&\coloneqq\delta_{ik}\delta_{jl}\label{eq:a0}\\
    a^{(\alpha, \beta-1)}_{ijkl} &\coloneqq a^{(\alpha,\beta)}_{ijkl}-a^{\vphantom{\beta}}_{\alpha\beta ij}a^{(\alpha, \beta)}_{\alpha \beta kl}&& \text{for } \alpha \in \set{1, \ldots, m-1}\text{ and }  \beta \in \set{1, \ldots, d}\label{eq:ab}\\
    a^{(\alpha -1,d)}_{ijkl} &\coloneqq
    a^{(\alpha,0)}_{ijkl} &&\text{for } \alpha \in \set{1, \ldots, m-1}.\label{eq:aa}
\end{align}

We now set
\begin{equation}
    a^\perp_{ijkl} \coloneqq a^{(0,0)}_{ijkl} \qquad \text{for }i,k \in \set{1, \ldots, m-1} \text{ and }  j,l \in \set{1, \ldots, d}.
\end{equation}
$a^\perp$ is the matrix associated with the orthogonal complement of the kernel of $a$. 
\begin{lemma}\label{lem:iterationstrcuture} 
    The following properties hold
    \begin{itemize}
        \item $a^\perp_{ijkl}= (1-a_{ijij})\delta_{ik}\delta{jl}$ for $id+j \geq kd+l$.
        \item $a_{ijkl}^\perp =0$ whenever $a_{ijij}^\perp =0$.
        \item if $\overline{a}_{ijkl}=\delta_{ik}a_{ijkl}$ then $\overline{a}^\perp_{ijkl}=\delta_{ik} a^\perp_{ijkl}$.
    \end{itemize}
\end{lemma}

\begin{proof}
    Observe that if $\alpha d + \beta < id +j$, then
    \begin{equation}
        a^{(\alpha, \beta-1)}_{ijkl} =a^{(\alpha, \beta)}_{ijkl} -a_{\alpha\beta ij}a^{(\alpha, \beta)}_{\alpha\beta kl}=a^{(\alpha, \beta)}_{ijkl}
    \end{equation}
    since $a_{\alpha \beta ij}=0$.  Applying it and \cref{eq:aa} repeatedly, we have
    \begin{equation}\label{eq:zzzformulaforperp}
        a^\perp_{ijkl} = a^{(i,j-1)}_{ijkl}  =a^{(i,j-1)}_{ijkl} = a^{(i,j)}_{ijkl} - a_{ijij} a^{(i,j)}_{ijkl} = (1- a_{ijij} )a^{(i,j)}_{ijkl}.
    \end{equation}
    By induction, one can also show that $a^{(\alpha,\beta)}_{ijkl}=0$ for $id+j>kd+l$. Indeed, if $\alpha = m-1$ and $\beta =d$ the statement is true. Additionally, if the statement is true for some $\alpha$ and $\beta$, then
    \begin{equation}
        a_{\alpha \beta ij}a^{(\alpha,\beta)}_{\alpha \beta kl} =0,
    \end{equation}
    because either $id+ j>\alpha d + \beta$ and $a_{\alpha \beta ij}=0$ or $\alpha d + \beta \geq id +j>kd +l$ and $a^{(\alpha,\beta)}_{\alpha \beta kl} =0$.

    \medskip
    (1) Since $a^{(\alpha, \beta)}_{\alpha\beta kl} =0$ for all $\alpha d+ \beta > id + j \geq kd +l$, \cref{eq:zzzformulaforperp} yields
    \begin{equation}
        a^\perp_{ijkl} = (1- a_{ijij} )a^{(i,j)}_{ijkl} =(1- a_{ijij} )a^{(m-1,d)}_{ijkl}=(1- a_{ijij} )\delta_{ik}\delta_{jl}
    \end{equation}
    
    (2) By the previous step, if $a^\perp_{ijij}=0$ then $a_{ijij}=1$, therefore $a^\perp_{ijkl}=0$ by \cref{eq:zzzformulaforperp}.

    (3) Observe that $\overline{a}^{(m-1,d)}_{ijkl}=\delta_{ik}\delta_{jl} =\delta_{ik}a^{(m-1,d)}_{ijkl}$. Assume that $a^{(\alpha, \beta)}_{ijkl}=\delta_{ik}a^{(\alpha,\beta)}_{ijkl}$ for some $\alpha$ and $\beta$. Then,
    \begin{equation}
        \overline{a}^{(\alpha,\beta-1)}_{ijkl} = \delta_{ik}a^{(\alpha,\beta)}_{ijkl} - \delta_{\alpha i}a_{\alpha \beta ij} a^{(\alpha,\beta)}_{\alpha \beta kl} \delta_{\alpha k}.
    \end{equation}
    Observe that $\delta_{i\alpha}\delta_{\alpha k} = \delta_{ik}$ if $i \neq k$ and $i=k =\alpha$. If $i=k\neq \alpha$ either $\alpha>i$ and $a_{\alpha \beta ij} =0$ or $\alpha<k$ and $a^{(\alpha,\beta)}_{\alpha \beta kl}=0$. The conclusion follows.
\end{proof}

\begin{lemma}
    Let $y,z \in \R^{d(m-1)}$ with $y \in \ker a$, then
    $\ip{z,y} = \ip{a^\perp z,y}$. Moreover, $\ker a^\perp = (\ker a )^\perp $. 
\end{lemma}
\begin{proof}
    Let $y \in \ker a $.
    \begin{align}
        \ip{z,y} &= \sum_{i,k=1}^{m-1} \sum_{j,l=1}^{d} [a^{(m-1,d)}_{ijkl}z_{kl}]^\dagger y_{ij}-[a^{(m-1,d)}_{(m-1)dkl}z_{kl}]^\dagger a_{(m-1)dij}y_{ij}\\
        &=\sum_{i,k=1}^{m-1} \sum_{j,l=1}^{d} [a^{(m-1,d-1)}_{ijkl}z_{kl}]^\dagger y_{ij}
        = \ldots =\ip{a^\perp z,y}.
    \end{align}
    This shows that $\ker a^\perp \subseteq (\ker a )^\perp$. \cref{lem:iterationstrcuture} implies $\rnk a^\perp = \# \set{ a_{ijij} =0 } = \dim \ker a$. In other words, $\ker a^\perp = (\ker a)^\perp$.
\end{proof}

\begin{definition}[Neumann boundary conditions]\label{def:neumann_boundary}
Let $B_{ij}$ be as in the \cref{def:admissible_dirichlet} and $a_{ijkl}=a_{ijkl}(\gamma)$ the coefficients associated to it. The Neumann boundary conditions are
\begin{equation}
    B_{(2m-i-1)j}(\gamma,\psi) \coloneqq \sum_{k=0}^{m-1}\sum_{l=1}^{d}a^\perp_{ijkl}(\gamma)\ip{\psi^k(\gamma^l),\nu^l}.
\end{equation}
\end{definition}

\begin{lemma} The structure of the operators $a^\perp_{ijkl}$ in terms of the derivatives of $\gamma$ is as follows:
    $a^\perp_{ijkl}=\pol^{k-i+1}_{k-i}$.
\end{lemma}

\begin{proof}
    Observe that $a^{(m-1,d)}_{ijkl} = \delta_{ik}\delta_{jl} = \pol^{k-i+1}_{k-i}$. Assume now that $a^{(\alpha,\beta)}_{ijkl} = \pol^{k-i+1}_{k-i}$. If $\beta =0$, by \cref{eq:aa} we have $a^{(\alpha-1,\beta)}_{ijkl} = \pol^{k-i+1}_{k-i}$. Assume $\beta>0$. Then,
    \begin{equation}
        a^{(\alpha,\beta-1)}_{ijkl}=a^{(\alpha,\beta)}_{ijkl}-a^{\vphantom{\beta}}_{\alpha\beta ij}a^{(\alpha, \beta)}_{\alpha \beta kl} = \pol^{k-i+1}_{k-i} - \pol^{\alpha-i+1}_{\alpha-i} \pol^{k-\alpha +1}_{k-\alpha} = \pol_{k-i}^{k-i+1}
    \end{equation}
    since $\max\set{k-\alpha+1, \alpha-i+1}\leq k-i+1$, $\alpha < i$ or $\alpha >k$.
\end{proof}

\subsubsection{Examples of higher-order conditions}\label{sec:higher-bc} We list reasonable admissible Dirichlet boundary conditions and we compute the coefficients that turn nonzero when such a boundary condition is imposed.

\paragraph{Fixed angle at an endpoint} Let  $\gamma^1 \in \gamma$ and let  $P$ be an endpoint of $\gamma^1$. Let $w \in \S^1$. The fixed angle condition at $P$ can be written as $B_{11}= \ip{\pas \gamma^1 , R_{1,1} w } $, where $R_{1,1}$ is the (fixed) rotation of $\R^2$ sending $w$ to $\evaluat{\nu^1}_{P}$. Consequently, $a_{1111}= 1$, since
\begin{equation}
    \delta_{\varphi}B_{11} = \ip{\delta_\varphi \pas \gamma^1, R_{1,1} w} = \ip{\delta_\varphi \pas \gamma^1, \nu^1}.
\end{equation}

\paragraph{Fixed curvature at an endpoint} Let  $\gamma^1 \in \gamma$ and let  $P$ be an endpoint of $\gamma^1$, and $k^1\in \R$. The fixed curvature condition at $P$ is $B_{21}= \ip{\pas^2 \gamma^1 , \nu^1 } - k^{1} $. Consequently, 
\begin{equation}
    \delta_\varphi B_{21} = \ip{\delta_\varphi \pas^2 \gamma^1, \nu^1}+ \ip { \pas^2 \gamma^1, \delta_\varphi \nu^1} =  \ip{\delta_\varphi \pas^2 \gamma^1, \nu^1}
\end{equation}
since $\ip{\delta_\varphi \nu^1, \nu^1}=0$. Then, $a_{2121}=1$ and $a_{2111}=0$.

\paragraph{Preserved angles at a junction} Let $P$ be a $d$-junction. The preserved angle condition is $B_{1i}=\ip{\pas\gamma^i,R_{i,i-1}\pas\gamma^{i-1}}$ for every $i>1$, where $R_{i,i-1}$ is the (fixed) rotation of $\R^2$ sending $\evaluat{\pas \gamma^{i-1}}_P$ to $\evaluat{\nu^i}_P$. Hence,
\begin{align}
        \delta_\varphi B_{1i} = \ip{\delta_\varphi \pas \gamma^i, \nu^i} + \ip{ \pas \gamma^i, R_{i,i-1}\delta_\varphi \pas \gamma^{i-1}} = \ip{\delta_\varphi \pas \gamma^i, \nu^i} -\ip{\delta_\varphi \pas \gamma^{i-1}, \nu^{i-1}} 
\end{align}
for all $i>1$. Here, we used $\ip{ \pas \gamma^i, R_{i,i-1}\delta_\varphi \pas \gamma^{i-1}}=-\ip{\nu^{i-1}, \delta_\varphi \pas \gamma^{i-1}}$. Hence, $a_{1i1j} = \delta_{ij} - \delta_{(i-1)j}$ for all $i>1$ and $a_{111j}=0$.

\paragraph{Preserved curvature at a junction}  Let $P$ be a $d$-junction. The preserved curvature condition is $B_{2i}=\ip{\pas^2\gamma^i,\nu^i}-\ip{\pas^2\gamma^{i-1}, \nu^{i-1}}$ for every $i>1$. Hence,
\begin{equation}
    \delta_{\varphi} B_{2i} = \ip{\delta_\varphi \pas^2 \gamma^i, \nu^i}- \ip{\delta_\varphi \pas^2 \gamma^{i-1}, \nu^{i-1}}
\end{equation}
for all $i>1$. Hence, $a_{2i2j} = \delta_{ij} - \delta_{(i-1)j}$ for all $i>1$, $a_{212j}=0$, and $a_{2i1j}=0$.

\subsection{Set of conditions} We can finally describe the list of boundary conditions of the geometric gradient flow of $\E_m$ and define the set of admissible initial data.

\begin{definition}
    A set of conditions $\Bcnd(\gamma, \psi)$ is  admissible if for every $d$-junction $P$ it can be written as $B_{ij}$ for $i \in \set{0, \ldots, 2m-1}$ and $ j \in \set{1, \ldots, d}$, where
    \begin{itemize}
        \item $B_{0j}$ and $B_{ij}$ for $i \in \set{1, \ldots, m-1}$ are as in \cref{def:admissible_dirichlet_0} and \cref{def:admissible_dirichlet} respectively;
        \item $B_{(2m-1)j}$ and $B_{(2m-i-1)j}$ for $i \in \set{1, \ldots, m-1}$ are obtained as in \cref{def:Neumann_highest} from $B_{0j}$ and \cref{def:neumann_boundary} from $B_{ij}$ respectively.
    \end{itemize}
\end{definition}

\begin{definition}[Non-degeneracy]
    A network $\gamma$ satisfying a set of conditions $\Bcnd(\gamma,\psi)$ is non-degenerate if 
    \begin{enumerate}
        \item each curve $\gamma^i \in \gamma$ has positive length, and
        \item at every $d$-junction $P$, we have $\det \ndA(\nu) \neq 0$ where $\nu=(\nu^1, \ldots, \nu^d)$ are the unit normal vectors of $\gamma$ at $P$ and $\ndA(\nu)$ is defined in \cref{eq:nondegeneracy_system}.
    \end{enumerate}
    We will denote by $\mathcal{C}(\Bcnd)$ the class of all non-degenerate networks $\gamma$ that satisfy the prescribe boundary conditions $\Bcnd(\gamma, \psi )=0$.
\end{definition}

\begin{theorem}\label{thm:gradient_flow}
    Let $\Bcnd(\gamma,\psi)$ be a set of admissible conditions. If $\gamma \in \mathcal{C}(\Bcnd)$, then
    \begin{equation}
        \delta_{\varphi} \E_m(\gamma) = - \int_\gamma \ip{\vec{V}_m(\gamma) , \varphi} \dif \gamma
    \end{equation}
    for every admissible variation $\varphi$.
\end{theorem}

\begin{proof}
    Let $P$ be a $d$-junction. Since $\varphi$ is an admissible variation and $\Bcnd$ is admissible,
    \begin{equation}
        \sum_{l=1}^d a_{jl} \evaluat{\delta_\varphi \gamma^l }_{P} = 0, \qquad \text{and} \qquad
        \sum_{k=1}^{m-1} \sum_{l=1}^d \evaluat{a_{ijkl}(\gamma) \ip{\delta_\varphi \pas^{k}\gamma^l, \nu^l}}_{P} =0.
    \end{equation}
 Then, by construction of the Neumann boundary conditions (see \cref{def:Neumann_highest,def:neumann_boundary}), we have
    \begin{equation}
        \sum_{i=0}^{m-1} \sum_{j=1}^{d} \evaluat{\ip{\delta_\varphi \pas^i \gamma^j, \psi^i_m(\gamma^j)}}_{P} =0.
    \end{equation}
    The claim follows from \cref{prop:first_variation_energy} and \cref{rmk:boundary_couple_bis}.
\end{proof}

\subsection{Definition of the flow}

We aim to define the gradient flow of $\E_m$. From now on, we fix $m\geq 2$, $T>0$, and  $p\in (2m+1,4m+2)$. We consider the Sobolev space $  W_p^{1,2m}$ defined in \cref{eq:spazio-di-Sobolev-utilizzato}.

\medskip

For clarity of exposition, we suppose that 
$\gamma$ is a star network without loops composed of $d$ curves that meet at one junction $P$ and with $P^1, \ldots,P^d\in\mathbb{R}^2$ fixed endpoints. We impose a set of admissible conditions $\Bcnd(\gamma, \psi)$ and denote by $\mathcal{C}(\Bcnd)$ the class of all non-degenerate networks $\gamma$ that satisfy the prescribed boundary conditions $\Bcnd(\gamma, \psi )=0$.

\begin{definition}[Admissible initial datum]\label{def:admissible-datum}
    We say that a network $\gamma_0$ is an admissible initial datum of the gradient flow of $\E_m$ for the set of boundary conditions $\Bcnd(\gamma, \psi)$ if it belongs to $\mathcal{C}(\Bcnd)$ with each curve of class  $ W^{2m-\nicefrac{2m}{p}}_p(0,1)$.
\end{definition}

\begin{definition}[Geometric gradient flow]\label{def:flow}
    A solution to the geometric gradient flow of $\E_m$ in the class $\mathcal{C}(\Bcnd)$ in the time interval $[0,T)$ with admissible initial datum $\gamma_0$ is a time-dependent family of networks $\gamma_{t\in [0,T)}$, composed of curves in $W^{1,2m}_p$ that evolves according to 
    \begin{equation}\label{flow}
        \begin{cases}
            (\partial_t \gamma)^\perp & =& V\nu & \text{for all } t\in [0,T),x\in (0,1),\\
          \evaluat{\Bcnd(\gamma,\psi)}_{P_t}&=&0& \text{for all } P_t\in \nodes_t,\\
          \evaluat{\gamma}_{t=0} &=& \gamma_0& \text{for all } x\in [0,1].
        \end{cases}
    \end{equation}
\end{definition}

\begin{definition}[Uniqueness]\label{def:uniqueness}
    We say that a solution to the geometric gradient flow of $\E_m$ in the class $\mathcal{C}(\Bcnd)$ is unique if it is unique up to reparametrization, namely if any two distinct solutions to the flow coincide up to a time-dependent reparametrization.
\end{definition}

\section{Short-time existence}

Polden devoted his PhD thesis \cite{poldenthesis} to the study of system \cref{flow} for closed curves. 
The first question we would like to answer here is whether the imposed boundary conditions at the junctions lead to a well-posed evolution problem as well.

\subsection{De Turck's trick}\label{DeTurck}

As customary in geometric flows, only the normal velocity is imposed, and the motion in the tangential direction is not specified by the problem itself. This fact is reflected in the main equation of the system \cref{flow}: the motion equation turns out to be degenerate. A now-standard method to fix this issue is the so-called De Turck's trick: we choose a suitable tangential velocity to obtain a well-posed parabolic PDE for the parametrization in such a way that the geometric problem is invariant. 

Since we want to treat the problem as a PDE -- thus regarding each curve $\gamma$ as a function -- it is convenient to work on a fixed domain. For this reason, we fix a generic parametrization $x\in[0,1]$ for each curve in the network. Accordingly with the definition of the flow for a star network, we write $\gamma=(\gamma^1,\ldots,\gamma^d)$. Without loss of generality, we impose $P=\gamma^1(0)=\ldots=\gamma^d(0)$, and $P^i=\gamma^i(1)$ for $i\in\set{1,\ldots,d}$. 

We choose the tangential component of the velocity to be
\begin{equation}\label{tang-vel}
T^i=(-1)^{m+1}\ip{\frac{\pax^{2m}\gamma^i}{\abs{\pax \gamma^i}^{2m}},\frac{\pax \gamma^i}{\abs{\pax \gamma^i}}}.
\end{equation}
With this choice, the evolution equation becomes
\begin{equation}
\partial_t \gamma^i =V^i\nu^i+T^i\tau^i
=\frac{(-1)^m}{\abs{\pax\gamma^i}^{2m}}\pax^{2m}\gamma^i+\pol^{2m-1}(\gamma).
\end{equation}

\medskip
A vector-valued PDE of order $2m$ requires $2m$ boundary conditions. The conditions encoded in $\Bcnd(\gamma,\psi)$ provide $m+1$ of them. To complete the system, we impose $m-1$ additional conditions for each curve, namely
\begin{equation}
\ip{\pax^{m+j-1}\gamma^i(t,y),\tau^i(t,y)}=0,
\end{equation}
for every $j\in\set{1,\ldots,m-1}$, $y\in\set{0,1}$, and $t\in[0,T)$. Consistently with the notation used for the geometric boundary conditions, we denote this additional set of conditions by
$\evaluat{\Bcnd^\ast(\gamma)}_{P_t}=0$.

\begin{definition}[Gradient flow]\label{def:special-flow}
        A solution to the gradient flow of $\E_m$ in the class $\mathcal{C}(\Bcnd)$ in the time interval $[0,T)$ with admissible initial datum $\gamma_0$ is a time-dependent collection of functions $\gamma(t,x)\in W^{1,2m}_p$ that satisfy the following system
       \begin{equation}\label{gradient_flow}
        \begin{cases}
            \partial_t \gamma& =& V\nu+T\tau& \text{in } [0,T)\times (0,1),\\
          \evaluat{\Bcnd(\gamma,\psi)}_{P_t}&=&0& \text{for all } P_t\in \nodes_t,t\in [0,T),\\
        \evaluat{\Bcnd^\ast(\gamma)}_{P_t}&=&0&\text{for all } P_t\in \nodes_t,t\in [0,T),\\
\evaluat{\gamma}_{t=0}         &=& \gamma_0& \text{for all } x\in [0,1].
        \end{cases}
    \end{equation}
\end{definition}

Our first task is to show that there exists a unique solution to the gradient flow \cref{gradient_flow} in a small time interval $[0,T]$. The existence of a solution for the gradient flow \cref{gradient_flow} implies the existence of a solution for the geometric gradient flow \cref{flow}. We postpone the discussion of the uniqueness of solutions to \cref{flow} in \cref{sec:uniqueness}.

\subsection{Linear system}

To derive the underlying linear system, we linearize both the motion equation and the boundary conditions. Let us first focus on the evolution equation $\pat\gamma =V\nu +T\tau$ where the tangential velocity \cref{tang-vel} is specified in the previous section. Omitting the dependence on $(t,x)$,  
we can write the motion equation of each curve  as  
\begin{equation}
\pat\gamma=(-1)^{m+1}\frac{\pax^{2m}\gamma}{\vert \pax\gamma\vert^{2m}}
+\sum_{j=2m+2}^{6m} \frac{\pol^{2m-1}(\gamma)}{\abs{\pax \gamma}^j}.
\end{equation}
A suitable linearization of such an equation is obtained by linearizing only the highest-order term:
\begin{align}\label{lin-main-eq}
    \pat\gamma+(-1)^{m}\frac{\pax^{2m}\gamma}{\vert \pax h\vert^{2m}}&=\left(\frac{(-1)^m}{\abs{\pax h}^{2m}}-\frac{(-1)^m}{\abs{\pax \gamma}^{2m}}\right)\pax^{2m}\gamma
+\sum_{j=2m+2}^{6m} \frac{\pol^{2m-1}(\gamma)}{\abs{\pax \gamma}^j}.
\end{align}

Consider now the list of boundary conditions in which the highest-order terms have orders $i$ varying from $0$ to $2m-1$. We define the linearized boundary operator $\evaluat{\mathcal{L}_{h}\Bcnd(\gamma,\psi)}_{\mathcal{P}}$
linearizing each $i$-order boundary condition as follows.
\begin{itemize}
    \item For $i=0$, the condition is already linear/affine, so there is no need for a linearization: $\mathcal{L}_{h}B_{0j} \coloneqq B_{0j} $ for every $ j \in \set{1, \ldots, d}$.
    \item If $ i \in \set{1, \ldots, m-1}$ we define
    \begin{equation}\label{eq:linearized}
         \mathcal{L}_{h}B_{ij} \coloneqq  \sum_{l=1}^da_{ijil}(h)\ip{\frac{\pax^i\gamma^l}{\abs{\pax h^l}^i},\nu^l_h},
    \end{equation}
    where $a_{ijil}(h)=\pol^1_0(h)$.
    \item Similarly, for $i \in \set{1, \ldots, m-1}$ we define
    \begin{equation}
        \mathcal{L}_{h}B_{(2m-i-1)j} \coloneqq  \sum_{l=1}^d(-1)^{m-i+1} a^{\perp}_{ijil}(h)\ip{\frac{\pax^{2m-i-1}{\gamma^l}}{\abs{\pax h^l}^{2m-i-1}},\nu_h^l},
    \end{equation}
    where $a^{\perp}_{ijil}(h) = \pol^{1}_0(h)$
      \item Accordingly, we define the linearized maximal order condition as
    \begin{equation}
        \mathcal{L}_{h} B_{(2m-1)j} \coloneqq (-1)^{m+1}\sum_{i,k=1}^{d} \ip{ \ndA(\nu_h)^{-1}_{ji}\pi_{ik} \nu_h^k, \nu^j_h} \ip{\frac{\pax^{2m-1}{\gamma^k}}{\abs{\pax h^k}^{2m-1}},\nu_h^k}
    \end{equation}
    \item Finally, we should linearize the extra conditions $\evaluat{\Bcnd^\ast(\gamma)}_{P_t}=0$. For $ i \in \set{1, \ldots, m-1}$, we define
    \begin{equation}
         \mathcal{L}_{h}B^\ast_{(m+i-1)j}  \coloneqq  \ip{\pax^{(m+i-1)}\gamma^j, \tau_h^j}.
    \end{equation}
\end{itemize}

We also set 
\begin{equation}\label{lin-bdry}
    g_{ij} = \mathcal{L}_{h}B_{ij} - B_{ij} \quad\text{and}\quad 
      g_{ij}^{\ast} = \mathcal{L}_{h}B_{ij}^\ast - B_{ij}^\ast.
\end{equation}  
Note that $g_{ij}, g_{ij}^{\ast}  \in W_p^{\frac{1}{2m}-\frac{1}{2mp}}(0,T)$. 

\begin{remark}
    Observe that $a^\perp_{ijil}(h)$ is the orthogonal complement of $a_{ijil}(h)= \delta_{ik} a_{ijkl}(h)$ in virtue of \cref{lem:iterationstrcuture}.
\end{remark}

\begin{remark}
 For all $i \in \set{1, \ldots, 2m-1}$, the linearized boundary operator is obtained by retaining only the highest-order terms in $\gamma$ appearing in the full linearization. If $i=1$, the dominant term is $\pax \gamma$ that appears also as a non-linearity at the denominator (the tangents read as $\frac{\pax \gamma}{\abs{\pax \gamma}}$). We hence have a fully nonlinear condition. Although, $\mathcal{L}_hB_{1j}$ is the full linearization of $B_{1j}$ since the linearization of $\pas \gamma = \tau$ is given by 
 \begin{equation}
      \frac{\pax \gamma}{\abs{\pax h}}-\frac{\pax h\ip{\pax \gamma, \pax h}}{\abs{\pax h^l}^3} = \ip{\nu_h, \frac{\pax \gamma}{\abs{\pax h}}}\nu_h.
 \end{equation}
\end{remark}

\begin{definition}[Admissible initial datum]\label{linear-comp}
Let $g$ and $g^\ast$ be as above. We say that a collection of parametrized curves $h=(h^1,\ldots, h^d)$ with $h^i:[0,1]\to \mathbb{R}^2$ is an admissible initial datum for 
the linearized system \cref{linearisedproblem} associated with the gradient flow of $\E_m$ if each curve $h^i$ is of class $W^{2m-\nicefrac{2m}{p}}_p$, is regular, has positive length, and 
satisfies the linear compatibility conditions with respect to $(g,g^\ast)$ 
\begin{equation}\label{eq:linear-comp}
    \evaluat{\mathcal{L}_h\Bcnd(h)}_{\nodes}=\evaluat{g(0)}_{\nodes}\quad \text{and}\quad
     \evaluat{\mathcal{L}_h\Bcnd^\ast(h)}_{\nodes}=\evaluat{g^\ast(0)}_{\nodes}
\end{equation}
and  $\det \ndA(\nu_h) \neq 0$, where $\nu_h=(\nu^1_h(0), \ldots, \nu^d_h(0))$ are the unit normal vectors of $h$ at $x=0$ and $\ndA(\nu)$ is defined in \cref{eq:nondegeneracy_system}. 
\end{definition}

\begin{definition}[Linearized problem]
   The linearized system associated with the gradient flow of $\E_m$ in the class $\mathcal{C}(\Bcnd)$ is 
    \begin{equation}\label{linearisedproblem}
            \begin{cases}
\partial_t\gamma(t,x)+(-1)^m\frac{\pax^{2m}\gamma(t,x)}{\abs{\pax h(x)}^{2m}}& =& f(t,x)& \text{on } t\in [0,T],x\in (0,1),\\
   \evaluat{\mathcal{L}_h\Bcnd(\gamma)}_{\nodes_t}&=&\evaluat{g(t)}_{\nodes_t}& \text{on }[0,T],\\
       \evaluat{\mathcal{L}_h\Bcnd^\ast(\gamma)}_{\nodes_t}&=&\evaluat{g^\ast(t)}_{\nodes_t} &\text{for all } t\in [0,T],\\
          \evaluat{\gamma}_{t=0}(x) &=& h(x)& \text{on }[0,1],
            \end{cases}
    \end{equation}
where $(f,g,g^\ast,h)$ is a quadruple of given functions with $f\in  L_p\left((0,T); L_p(0,1)\right)$, $g=(g_{ij}), g^\ast= (g_{ij}^{\ast})  \in W_p^{\frac{1}{2m}-\frac{1}{2mp}}(0,T)$ and  $h$ an admissible initial datum in the sense of \cref{linear-comp}.
\end{definition}

Our objective is to establish an existence result for the linear system by employing the theory developed by Solonnikov \cite{solonnikov2}, which treats coupled boundary conditions. We apply \cite[Theorem 4.9]{solonnikov2} by verifying the necessary assumptions on the parabolic and boundary operators, as well as the initial conditions.
    
\subsubsection{Parabolic system in the sense of Solonnikov}\label{parabolicity}

We use the notation of \cite[Chapter 1]{solonnikov2}: let $b=m$ and $r=2d$, $\gamma^i=(u^i,v^i)$ and $h^i=(u_0^i, v_0^i)$.
We write the linearized evolution equation as
\begin{equation}
\begin{cases}
u^i_t+\frac{(-1)^m}{\abs{\pax h^i}^{2m}}\pax^{2m}u^i=a^i,\\
v^i_t+\frac{(-1)^m}{\abs{\pax h^i}^{2m}}\pax^{2m}v^i=b^i.
\end{cases}
\end{equation}
We can equivalently write the system as
\begin{equation}
    \mathcal{D}X=f
\end{equation}
where $X$ and $f$ are the vectors of components
$(u^1,v^1,\ldots,u^d,v^d)$ and $(a^1,b^1,\ldots,a^d,b^d)$
respectively and $\mathcal{D}(x,t,\partial_x,\partial_t)$ is the $2d\times 2d$ matrix with entries
\begin{equation}
    \begin{array}{lll}
    c_{jj}&=  \partial_t+\left(\frac{(-1)^m}{\abs{\pax h^{\lceil \frac{j}{2}\rceil}}^{2m}}\right)\pax^{2m}&\text{ for }j\in\{1,\ldots,2d\}\\
    c_{kj}&=0&\text{ for }k,j \in \{1,\ldots,2d\},\, k\neq j.
\end{array}
\end{equation}
We notice that each linear differential operator $c_{kj}$ coincides with its principal part. 

With the choice $s_k=0$ for every $k\in\{1,\ldots, 2d\}$ and $t_j=2b=2m$ for every $j\in\{1,\ldots, 2d\}$
the conditions of~\cite[page 8]{solonnikov2} are satisfied.

The determinant of the diagonal matrix $\mathcal{D}$ is simply the product of the elements of the diagonal:
\begin{equation}
    \det\mathcal{D}(x,t,i\xi,p)= \prod _{j=1}^d\left(p+  \frac{1}{\abs{\pax h^j}^{2m}}\xi^{2m}\right)^2.
\end{equation}
This is a polynomial of degree $2d$ in $p$ with roots of multiplicity two of the form $p_j=\frac{-1}{\abs{\pax h^j}^{2m}}\xi^{2m}$ with $j\in \set{1,\ldots,d}$. Choosing $\delta = \min\left\{\frac{1}{\abs{\pax h^j}^{2m}}: j\in \set{1,\ldots,d}\right\}$ the system is parabolic in the sense of Solonnikov as defined in~\cite[page 9]{solonnikov2}. Note that the system is parabolic also in the sense of Petrovski\u{i} (see \cite[page 8]{solonnikov2}).

\subsubsection{Lopatinski\u{i}-Shapiro condition}

As for the case of the linear operator, the boundary operator must fulfill certain conditions. The  Lopatinski\u{i}-Shapiro condition ensures that the boundary operators are compatible with the parabolic operator.

\begin{definition}[Lopatinski\u{i}-Shapiro condition]
    Let $\lambda\in\mathbb{C}$ with $ \Re(\lambda)>0$. We say that the  Lopatinski\u{i}-Shapiro condition for system \cref{linearisedproblem} is satisfied at a boundary point if every solution $(\gamma^i)_{i \in \set{1, \ldots, d}}\in C^{2m}([0,\infty),(\mathbb{C}^2)^d)$ to the system of ODEs 
\begin{equation}\label{LopatinskiiShapirosystem}
\begin{cases}
\begin{array}{llll}
\lambda \gamma^i(x)+\frac{(-1)^m}{\abs{\pax h^i(0)}^{2m}}\pax^{2m}\gamma^i(x)&=0&\;x\in[0,\infty),
i\in\{1,\ldots,d\}&\;\text{motion,}\\
    \begin{aligned}
       & \evaluat{\mathcal{L}_h\mathcal{B}(\gamma)}_{x=0} \\
       & \evaluat{\mathcal{L}_h\mathcal{B}^\ast(\gamma)}_{x=0}
    \end{aligned} 
    &     \begin{aligned}
      =0\\
       = 0
    \end{aligned} &\; &\;\text{boundary conditions}\\
\end{array}
\end{cases}
\end{equation}

which satisfies $\lim_{x\to\infty}\abs{\gamma^i(x)}=0$ is the trivial solution. 
\end{definition}

\begin{lemma}\label{LopatinskiiShapiro}
The  Lopatinski\u{i}-Shapiro condition is satisfied at each $d$-junction.
\end{lemma}
\begin{proof}
    Let $(\gamma^{i})_{i \in \set{1, \ldots, d}}$ be a solution to \cref{LopatinskiiShapirosystem} such that $\lim_{x\to\infty}\lvert \gamma^i(x)\rvert=0$.
    Because of the specific exponential form of the solutions to \cref{LopatinskiiShapirosystem}, all the derivatives of $\gamma^{i}$ also  decay to zero as $x$ tends to infinity. We test the motion equation by $\abs{\pax h^i(0)}\overline{\ip{\gamma^i(x) , \nu^i_h}}\nu^i_h$, integrate by parts $m$-times and sum to get 
\begin{align}
    0=&\sum_{i=1}^d \lambda\abs{\pax h^i(0)}\int_0^{+\infty} \abs{\ip{\gamma^i(x) , \nu^i_h}}^2 \dif x +\int_0^{+\infty} \frac{\abs{\ip{\pax^m\gamma^i(x) , \nu^i_h}}^2}{\abs{\pax h^i(0)}^{2m-1}}\dif x +\\&\qquad +\sum_{j=0}^{m-1} (-1)^{m-j-1} \evaluat{\ip{\frac{\pax^{2m-j-1} \gamma^i}{\abs{\pax h^i(0)}^{2m-j-1}}, \nu_h^i}\overline{\ip{\frac{\pax^{j} \gamma^i}{\abs{\pax h^i(0)}^j},\nu_h^i}}}_{x=0}.
\end{align}
Since $\evaluat{\mathcal{L}_h\Bcnd(\gamma)}_{x=0}=0$, the boundary term vanishes. Hence, $\abs{\ip{\gamma^i(x) , \nu^i_h}}^2 =0$, thus $\gamma^i(x)= \ip{\gamma^i(x), \tau^i_h} \tau^i_h$. In particular, $\gamma^i(0) \in \ker A(\nu_h)$ which means $\gamma^i(0) =0 $. Testing now the motion equation by $\abs{\pax h^i(0)}\overline{\ip{\gamma^i(x) , \tau^i_h}}\tau^i_h$ we get
\begin{align}
    0=&\sum_{i=1}^d \lambda\abs{\pax h^i(0)}\int_0^{+\infty} \abs{\ip{\gamma^i(x) , \tau^i_h}}^2 \dif x +\int_0^{+\infty} \frac{\abs{\ip{\pax^m\gamma^i(x) , \tau^i_h}}^2}{\abs{\pax h^i(0)}^{2m-1}}\dif x +\\&\qquad +\sum_{j=0}^{m-1} (-1)^{m-j-1} \evaluat{\ip{\frac{\pax^{2m-j-1} \gamma^i}{\abs{\pax h^i(0)}^{2m-j-1}}, \tau_h^i}\overline{\ip{\frac{\pax^{j} \gamma^i}{\abs{\pax h^i(0)}^j},\tau_h^i}}}_{x=0}.
\end{align}
In the boundary term, the $j=0$ term vanishes since $\gamma^i(0)=0$, while the terms for $j>0$ vanish since  $\evaluat{\mathcal{L}_h\Bcnd^*(\gamma)}_{x=0}=0$. In particular, we proved that $\gamma^i(x)= \ip{\gamma^i(x), \tau^i_h} \tau^i_h = 0$.
\end{proof}

\subsubsection{Initial value problem}

Finally, we verify the complementary conditions for the initial data, as prescribed in \cite[p. 12]{solonnikov2}. Adopting the notation therein, we observe that the $2d\times 2d$ matrix $[C_{\alpha j}]$ is the identity matrix. With the choice $\gamma_{\alpha j}=0$ for every $\alpha\in\set{1,\ldots,2d}$ and $j\in\set{1,\ldots,2d}$ we obtain $\rho_\alpha=0$ and $C_0=\id_{\R^{2d}\times \R^{2d}} $. Furthermore, the rows of the matrix representing the principal part of the operator at t=0, namely $\mathcal{D}(x,0,0,p)=p^{2d} \id_{\R^{2d}\times \R^{2d}}$ are linearly independent modulo the polynomial $p^{2d}$.

\subsubsection{Existence and uniqueness of the linearized system}

We verified all the assumptions to apply \cite[Theorem 4.9]{solonnikov2}, we hence get the following existence result.

\begin{theorem}\label{thm:existence_Solonnikov}
Let $T>0$ and $ (f,g,g^\ast,h)$ be a quadruple of funtions with 
\begin{itemize}
    \item $f\in L_p\left((0,T); L_p(0,1)\right)$;
    \item $g=(g_{ij}), g^\ast= (g_{ij}^{\ast})  \in W_p^{\nicefrac{1}{2m}-\nicefrac{1}{2mp}}(0,T)$;
    \item $h$ an admissible initial datum in the sense of \cref{linear-comp}.
\end{itemize}
 Then, there exists a unique solution $\gamma\in W^{1,2m}_p$ to system \cref{linearisedproblem}. Furthermore, there exists a constant $\kst(T)>0$ such that all solutions satisfy  the inequality
\begin{align}\label{stima-dipendente-da-T}
\Vert \gamma\Vert_{W^{1,2m}_p} \leq \kst(T)\left( 
 \Vert f\Vert_{ L_p((0,T); L_p((0,1))}+\Vert (g,g^\ast)\Vert_{ W_p^{\nicefrac{1}{2m}-\nicefrac{1}{2mp}}}+
\Vert h\Vert_{ W_p^{2m-\nicefrac{2m}{p}}}
\right).
\end{align}
\end{theorem}
\begin{proof}
As explained in \cref{parabolicity}, system \cref{linearisedproblem} is parabolic. Furthermore, by \cref{LopatinskiiShapiro}, the  Lopatinski\u{i}-Shapiro condition is satisfied. 
Since $p>2m+1$, the embedding $W_p^{2m-\nicefrac{2m}{p}}\left((0,1)\right)\hookrightarrow C^{2m-1}\left([0,1]\right)$ holds, ensuring that the compatibility conditions are well-defined pointwise. Finally, we showed the complementary condition for the initial datum.  Thus, the theorem follows from \cite[Theorem 5.4]{solonnikov2}.
\end{proof}

\begin{definition}
    Let $T>0$. We define the Banach spaces
    	\begin{align}
	\mathbb{E}_T\coloneqq&\set{\gamma\in \boldsymbol{E}_T
	\;\text{such that for all}\; t\in(0,T)\,\text{it holds}\;\evaluat{B_{0j}(\gamma)}_{\nodes}=0},\\
	\mathbb{F}_T\coloneqq &
    \left\lbrace\begin{array}{l}
       (f,g,g^\ast,h)\in 
	L_p\left((0,T); L_p((0,1)\right)
	\times \boldsymbol{F}_T\times \boldsymbol{F}_T
	\times  W_p^{2m-\nicefrac{2m}{p}}(0,1) \\
    \text{such that $h$ satisfies the linear compatibility \cref{eq:linear-comp} with respect to $(g,g^\ast)$}
    \end{array}\right\rbrace 
	\end{align}
endowed with the norm $\Vert\cdot\Vert_{\mathbb{E}_T}\coloneqq \Vert\cdot\Vert_{\boldsymbol{E}_T}$ and 
\begin{equation}
\Vert (f,g,g^\ast, h)\Vert_{\mathbb{F}_T}\coloneqq\left\lVert f\right\rVert_{L_p\left((0,T);L_p((0,1)\right)}+\Vert g\Vert_{\boldsymbol{F}_T}+\Vert g^\ast\Vert_{\boldsymbol{F}_T}+\left\lVert h\right\rVert_{W_p^{2m-\nicefrac{2m}{p}}\left(0,1\right)}.
\end{equation}
Finally, given a linear operator $\mathcal{A}:\mathbb{F}_T\to\mathbb{E}_T$ we let
\begin{equation}
\Vert{\mathcal{A}}\Vert_{\mathscr{L}\left(\mathbb{F}_T,\mathbb{E}_T\right)}\coloneqq\sup\{\Vert{\mathcal{A}(f,g,g^\ast,h)}\Vert_{\mathbb{E}_T}:(f,g,g^\ast,h)\in\mathbb{F}_T,\Vert(f,g,g^\ast,h)\Vert_{\mathbb{F}_T}\leq 1\}\,.
\end{equation}
\end{definition}

\medskip 
As a consequence of \cref{thm:existence_Solonnikov}, the operator $L_{T}:\mathbb{E}_T\to \mathbb{F}_T$ defined by
\begin{equation}
L_T(\gamma)\coloneqq\left(\left(\gamma_t+(-1)^m\frac{\pax^{2m}\gamma}{\abs{\pax h}^{2m}}\right),\, \evaluat{\mathcal{L}_h\mathcal{B}(\gamma)}_{\nodes},\,\evaluat{\mathcal{L}_h\mathcal{B}^\ast(\gamma)}_{\nodes},\,\evaluat{\gamma}_{t=0}\right)
\end{equation}
is well-defined, linear, continuous, and invertible. We then denote by $L^{-1}_T$ the inverse of $L_T$.

\medskip

In \cref{thm:existence_Solonnikov}, we fix $T>0$, show an existence result in the interval $[0,T]$ and get the estimate \cref{stima-dipendente-da-T} where the constant depends on $T$. Now, once a certain interval of time $(0, T_1]$ with $T_1>0$ is chosen,  for every $T\in (0, T_1]$ it is possible to estimate the norm of $L^{-1}_T$ with a constant independent of $T$. We can promote the estimate  \cref{stima-dipendente-da-T} on an estimate on the operator norm of  $L^{-1}_T$  with a constant independent of $T$ because in \cref{thm:existence_Solonnikov} we can take any $T>0$, and always get existence and uniqueness. The proof is standard and uses extension operators.  

\begin{lemma}\label{lem:bound-L}
    For all $T_1>0$ there exists a constant $\kst(T_1)$ such that
\begin{equation}
    \sup_{T\in (0,\frac{1}{2}T_1]}\Vert L_T^{-1}\Vert_{\mathscr{L}(\mathbb{F}_T,\mathbb{E}_T)}\leq \kst(T_1).
\end{equation}
\end{lemma}

\subsection{Local-in-time-existence}

In the previous section, we ensured the existence and uniqueness of a solution to the underlying linear problem. We should now introduce an operator that encodes the non-linearity specific to our gradient flow. We should take into account that the initial datum $\gamma_0$ is not necessarily zero. As a consequence, the spaces we deal with are affine rather than linear. Since we are going to use the contraction mapping principle, we need complete metric spaces.

\begin{definition}
    Let $T>0$, $\gamma_0 \in W_p^{2m-\nicefrac{2m}{p}}\left(0,1\right)$. We define 
\begin{align}
\mathbb{E}^{\gamma_0}_T\coloneqq&\left\{\gamma\in 
\mathbb{E}_T\,\text{such that }\,\evaluat{\gamma}_{ t=0}=\gamma_0\right\},\\
\mathbb{F}^{\gamma_0}_T\coloneqq&\left\{(f,g,g^\ast)\,\text{such that }\, (f,g,g^\ast,\gamma_0)\in\mathbb{F}_{T} \right\}
\times\left\{\gamma_0\right\}.
\end{align}
\end{definition}

\begin{lemma}\label{lem:nonempty}
Let $\gamma_0$ be an admissible initial datum in the sense of \cref{def:admissible-datum}. Then, the space $\mathbb{E}^{\gamma_0}_T$ is non-empty.
\end{lemma}
\begin{proof}
As $\gamma_0$ is an admissible initial datum, one easily checks that $f\equiv 0$, $g\equiv 0$, $g^\ast\equiv 0$ and $h\equiv \gamma_0$ is an admissible right hand side for system \cref{linearisedproblem}. In other words, $\left(0,0,0,\gamma_0\right)\in\mathbb{F}_T$ and hence \cref{thm:existence_Solonnikov} yields the existence of $\varrho\in\mathbb{E}_T$ with $L_T\varrho=\left(0,0,0,\gamma_0\right)$. In particular, $\varrho_{|t=0}=\gamma_0$ which gives $\varrho\in\mathbb{E}_T^{\gamma_0}$. 
\end{proof}

\begin{definition}
    We introduce the maps
    \begin{align}
        N_T&:\mathbb{E}^{\gamma_0}_T \to \mathbb{F}^{\gamma_0}_T, \quad
         N_T( \gamma)=\left(N_T^1(\gamma),N_T^2(\gamma),N_T^3(\gamma),\evaluat{\gamma}_{t=0}\right)\\
        K_T&:\mathbb{E}^{\gamma_0}_T\cap \overline{B_M}\to \mathbb{E}^{\gamma_0}_T,\quad 
	K_T:=L_T^{-1}\circ N_T
    \end{align}
where the three components $N^1_T,N^2_T,N^3_T$ are  
\begin{align}
N^{1}_T:&
\begin{cases}
\begin{array}{ll}
 \mathbb{E}^{\gamma_0}_T   & \to L_p\left((0,T);L_p((0,1)\right), \\
 \gamma    &  \mapsto f(\gamma),
\end{array}
\end{cases}\\
N^2_{T}:&
\begin{cases}
\begin{array}{ll}
  \mathbb{E}^{\gamma_0}_T  & \to  \boldsymbol{F}_T, \\
 \gamma    & \mapsto g(\gamma),
\end{array}
\end{cases}\qquad 
N^{3}_T:
\begin{cases}
\begin{array}{ll}
  \mathbb{E}^{\gamma_0}_T  & \to  \boldsymbol{F}_T, \\
 \gamma    & \mapsto g^\ast(\gamma),
\end{array}
\end{cases}
\end{align}
with 
\begin{align}
       f(\gamma)(t,x)&\coloneqq\left(\frac{(-1)^m}{\abs{\pax \gamma_0}^{2m}}-\frac{(-1)^m}{\abs{\pax \gamma}^{2m}}\right)\pax^{2m}\gamma+\sum_{j=2m+2}^{6m} \frac{\pol^{2m-1}(\gamma)}{\abs{\pax \gamma}^j}\text{ as defined in \cref{lin-main-eq}}\\
       g(\gamma)(t)&\coloneqq (g_{ij}(\gamma)),\, g^\ast(\gamma)(t)\coloneqq (g^\ast_{ij}(\gamma))\quad\text{where }g_{ij},g^\ast_{ij} \text{ are defined in \cref{lin-bdry}.}
\end{align}
\end{definition}

\medskip

Our final goal is to show that the map $K_T$ is a contraction. 

\begin{lemma}\label{prime-stime}
   Given a time $T_1>0$ and a radius $M>0$, we define 
   \begin{align}
    \overline{B_M}&\coloneqq\left\{\gamma\in\mathbb{E}_{T_1}:\Vert\gamma\Vert_{\mathbb{E}_{T_1}}\leq M\right\},\\
    \boldsymbol{c}&\coloneqq\frac{1}{2}\min_{i\in\{1,\ldots,d\}} \inf_{x\in[0,1]}\abs{\pax\gamma_0^i(x)} >0.
   \end{align}
	Then, there exists a time $\widetilde{T}=\widetilde{T}\left(\boldsymbol{c},M\right)\in (0,T_1]$  such that for all $\gamma\in \mathbb{E}_{T}^{\gamma_0}\cap\overline{B_M}$ with $T\in(0,\widetilde{T}]$  we have
	\begin{equation}\label{eq:c}
	\min_{i\in\{1,\ldots,d\}}\inf_{t\in[0,T],x\in[0,1]}\abs{\pax\gamma^i(t,x)}\geq \boldsymbol{c}.
	\end{equation}
	In particular, the curves $\gamma^i(t)$ are regular for all $t\in [0,T]$. 
\end{lemma}

\begin{proof}
   Let $\theta\in\left(\frac{1+\nicefrac{1}{p}}{2m-\nicefrac{2m}{p}},1\right)$ be fixed and $\beta\coloneqq\left(1-\theta\right)\left(1-\nicefrac{1}{p}\right)$. Given $T\in (0,T_1]$ and $\gamma\in\mathbb{E}_T^{\gamma_0}\cap\overline{B_M}$ we have for $i\in\{1,\ldots,d\}$, $t\in[0,T]$, $x\in[0,1]$,
\begin{equation}
\abs{\pax\gamma^i(t,x)}\geq \abs{\pax\gamma_0^i(x)}-\abs{\pax\gamma^i(t,x)-\pax\gamma^i(0,x)}\geq 2\boldsymbol{c}-\left\lVert\gamma^i(t)-\gamma^i(0)\right\rVert_{\CS^1\left([0,1]\right)}\,.
\end{equation}
\cref{uniformcalphac1} implies for all $t\in[0,T]$,
\begin{equation}
\left\lVert\gamma^i(t)-\gamma^i(0)\right\rVert_{\CS^1\left([0,1]\right)}\leq t^\alpha\left\lVert\gamma^i\right\rVert_{\CS^\alpha\left([0,T];\CS^1\left([0,1]\right)\right)}\leq T^\alpha \kst\left(T_1,p\right)\Vert \gamma\Vert_{\mathbb{E}_T}\leq T^\alpha M \kst\left(T_1,p\right).
\end{equation}
The claim follows by choosing $\widetilde{T}$ sufficiently small such that $\widetilde{T}^\alpha M \kst\left(T_1,p\right)\leq \boldsymbol{c}$. 
\end{proof}

From now on, $T_1>0$, $M>0$ are fixed and $\widetilde{T}$ is the time appearing in \cref{prime-stime}.

\begin{lemma}\label{lem:main_estimates_polinomial_fractions}
    Given two curves $\gamma, \widehat{\gamma} \in \mathbb{E}_T$ and a polynomial $p (\gamma) = p(\pax \gamma, \ldots, \pax^j \gamma)$,  it holds
    \begin{align}
        \abs{|\partial_x \gamma|^{-j}- |\partial_x \widehat{\gamma}|^{-j} }&\leq \frac{\pol^1(\gamma,\widehat{\gamma})}{|\partial_x \widehat{\gamma}|^j|\partial_x {\gamma}|^j} \abs{\partial_x {\gamma}-\partial_x \widehat{\gamma}},\label{lem:fracestimate}\\
        |p(\gamma)- p(\widehat{\gamma})| &\leq \sum_{i=1}^{j}\pol^{j}(\gamma, \widehat{\gamma})\abs{\partial_x^i \gamma -\partial_x^i \widehat{\gamma}}.\label{lem:poliniomialestimate}
    \end{align}
\end{lemma}
\begin{proof}
    The first equality follows from the identity:
    \begin{equation}
        \frac{1}{|\partial_x \gamma|^j} - \frac{1}{|\partial_x \widehat{\gamma}|^j} = \frac{|\partial_x \widehat{\gamma}|^j-|\partial_x \gamma|^j}{|\partial_x \gamma|^j|\partial_x \widehat{\gamma}|^j} 
        = \frac{\left(|\partial_x \widehat{\gamma}|-|\partial_x \gamma|\right)\left(\sum_{i=0}^{j-1}|\partial_x \widehat{\gamma}|^{j-1-i}|\partial_x \gamma|^i\right)}{|\partial_x \gamma|^j|\partial_x \widehat{\gamma}|^j}.
    \end{equation}
    Regarding \cref{lem:poliniomialestimate}, it follows from the fact that given a polynomial $p \in \R[x^1, \ldots, x^j]$ 
    \begin{equation}
        p(x)-p(y)= \sum_{i=1}^j(x^i-y^i) \int_0^1 \partial_i p(y+t(x-y)) \dif t = \sum_{i=1}^{j} (x^i-y^i)q(x,y),
    \end{equation}
    where $q(x,y) \in \R[x,y]$.
\end{proof}

\begin{proposition}\label{prop:contrazioni1}
For every $T\in (0,\widetilde{T}]$ the map $N_T^1$ is well-defined and there exist constants $\alpha\in (0,1)$ and $\kst=\kst(\boldsymbol{c},M)$ such that for all $\gamma, \widehat{\gamma} \in\mathbb{E}_T^{\gamma_0}\cap\overline{B_M}$, there hold
\begin{equation}
\left\lVert N_T^1(\gamma)-N_T^1\left(\widehat{\gamma}\right)\right\rVert_{L_p\left((0,T); L_p(0,1)\right)} \leq \kst\left(\boldsymbol{c},M\right) T^\alpha \Vert\gamma-\widehat{\gamma}\Vert_{\mathbb{E}_T}.
\end{equation}
\end{proposition}

\begin{proof}
    Let $\gamma, \widehat{\gamma} \in \mathbb{E}_T^{\gamma_0} \cap \overline{B_M}$. Recall the structure of the non-linear operator:
    \begin{equation}
        N^1_T(\gamma) = (n(\gamma_0)- n(\gamma))\partial_x^{2m}\gamma + p(\gamma) \abs{\pax \gamma}^{-6m}
    \end{equation}
    where $p(\gamma) = p(\pax \gamma, \ldots, \pax^{2m-1} \gamma)$ is a polynomial and $n(\gamma)= (-1)^{m}\abs{\pax\gamma}^{-2m}$. The difference $ N^1_T(\gamma) - N^1_T(\widehat{\gamma})$ can be rewritten as 
   \begin{equation}\label{eq:zzzN1Texpression}
   \begin{split}
        N^1_T(\gamma) - N^1_T(\widehat{\gamma}) &= (n(\gamma_0)-n(\gamma)) (\pax^{2m}\gamma - \pas^{2m}\widehat{\gamma}) + (n(\widehat{\gamma})- n(\gamma)) \pax^{2m}\widehat{\gamma} \\ &\qquad + p(\gamma) \abs{\pax \gamma}^{-6m} - p(\widehat{\gamma})\abs{\pax \widehat{\gamma}}^{-6m} .
        \end{split}
    \end{equation}
    Using \cref{prime-stime,lem:main_estimates_polinomial_fractions} we obtain
    \begin{equation}\label{eq:zzzN1Testimate}
   \begin{split}
        \abs{N^1_T(\gamma) - N^1_T(\widehat{\gamma})}&\leq \kst \abs{\pax \gamma - \pax \gamma_0} \abs{\pax^{2m}\gamma - \pax^{2m}\widehat{\gamma}} +\kst\abs{\pax \gamma - \pax \widehat{\gamma}} \abs{\pax^{2m}\widehat{\gamma}}\\&\qquad + \kst \sum_{j=1}^{2m-1}\abs{\pol^{2m-1}(\gamma, \hat{\gamma})}\abs{\pax^j \gamma - \pax^j \widehat{\gamma}}
        \end{split}
    \end{equation}
First, we observe that
\begin{align}
     \sup_{t\in [0,T],x\in[0,1]} \abs{\pax \gamma - \pax \gamma_0}\leq \kst T^\beta \norm{\gamma}_{\CS^\beta ([0,T]; \CS^1([0,1]))}\overset{\cref{stima-hölder}}{\leq} \kst T^\beta \norm{\gamma}_{\mathbb{E}_T}, 
\end{align}
and, since $\evaluat{(\gamma-\widehat{\gamma})}_{t=0}=0$, we similarly obtain
\begin{equation}
    \sup_{t\in [0,T],x\in[0,1]} \abs{\pax \widehat{\gamma}- \pax\gamma} \leq \kst T^\beta \norm{\gamma-\widehat{\gamma}}_{\mathbb{E}_T}.
\end{equation}
Moreover,
\begin{equation}
    \norm{\pax^{2m}\widehat{\gamma}}_{L_p((0,T);L_p(0,1))} \leq \kst \norm{\widehat{\gamma}}_{\mathbb{E}_T}\leq \kst(M),
\end{equation}
and
\begin{equation}
    \norm{\pax^{2m}\gamma-\pax^{2m}\widehat{\gamma}}_{L_p((0,T);L_p(0,1))} \leq \kst \norm{\gamma-\widehat{\gamma}}_{\mathbb{E}_T}.
\end{equation}
Combining these estimates, the first two terms in \cref{eq:zzzN1Testimate} can be controlled accordingly. On the other hand,
\begin{align}
\norm{\pax^k \gamma - \pax^k \widehat{\gamma}}_{L_p((0,T)L_p(0,1)} \leq \kst T^{\frac{1}{p}} \norm{\gamma- \widehat{\gamma}}_{\CS^0([0,T]; \CS^{2m-1}([0,1])}\overset{\cref{stima-buc}}{\leq} \kst  T^{\frac{1}{p}}\norm{\gamma - \widehat{\gamma}}_{\mathbb{E}_T},
\end{align}
and
\begin{equation}
    \sup_{t \in [0,T], x \in [0,1]}\abs{\pol^{2m-1}(\gamma, \widehat{\gamma})} \overset{\cref{stima-buc}}{\leq} \kst (\norm{\gamma}_{\mathbb{E_T}}\norm{\widehat{\gamma}}_{\mathbb{E_T}} +1)^\eta\leq \kst(M),
\end{equation}
which conclude the proof.
\end{proof}

\begin{proposition}\label{prop:contrazioni2}
For every $T\in (0,\widetilde{T}]$ the maps $N_T^2$ and $N_T^3$ are well-defined and there exist constants $\sigma\in(0,1)$ and $\kst=\kst(\boldsymbol{c},M)>0$ such that for all $\gamma$, $\widehat{\gamma}\in\mathbb{E}_T^{\gamma_0}\cap\overline{B_M}$,
\begin{align}
\Vert N_T^2(\gamma)-N_T^2\left(\widehat{\gamma}\right)\Vert_{\boldsymbol{F}_T}\leq\kst(\boldsymbol{c},M)T^\sigma \Vert\gamma-\widehat{\gamma}\Vert_{\mathbb{E}_T},\\  
\Vert N_T^3(\gamma)-
N_T^3\left(\widehat{\gamma}\right)\Vert_{\boldsymbol{F}_T}\leq\kst(\boldsymbol{c},M)T^\sigma \Vert\gamma-\widehat{\gamma}\Vert_{\mathbb{E}_T}.
\end{align}

\end{proposition}

\begin{proof}
The structure of each component of the non-linear boundary operators $N^2_T$ and $N^3_T$ is 
\begin{equation}
    (N^\bullet_T)_{ij}(\gamma) = \sum_{l=1}^d \ip{n_{ijl}(\gamma_0) - n_{ijl}(\gamma), \pax^i \gamma^l} + p(\gamma) \prod_{l=1}^{d} \abs{\pax \gamma^l}^{-\beta_l},
\end{equation}
where
\begin{equation}
    n_{ijl}(\gamma)= p_{ijl}(\gamma)\prod_{l=1}^d \abs{\pax \gamma^l}^{-\beta_l}
\end{equation}
for two polynomials $p(\gamma)= p(\pax \gamma, \ldots, \pax^{i-1} \gamma)$ and $p_{ijl}(\gamma)= p_{ijl}(\pax \gamma)$ in $\gamma=(\gamma^1, \ldots, \gamma^d)$.

Fix $\theta \in \left(\frac{p+1}{2m(p-1)}, \frac{2m-1}{2m} \right)$ and let $\alpha=\nicefrac{1}{2m}-\nicefrac{1}{2mp}$, $\beta=(1-\theta)\left(1-\nicefrac{1}{p}\right)$. 

Consider $\gamma,\widehat{\gamma}\in \mathbb{E}^{\gamma_0}_T \cap \overline{B_M}$.  Since $\evaluat{\gamma}_{t=0}=\evaluat{\widehat{\gamma}}_{t=0}=\gamma_0$, the difference at $t=0$ vanishes, and it suffices to estimate $ (N^\bullet_T)_{ij}(\gamma)- (N^\bullet_T)_{ij}(\widehat{\gamma})$ in $W^{\alpha}_p$-norm. Employing \cref{prime-stime,lem:main_estimates_polinomial_fractions} as in \cref{prop:contrazioni1}, the difference $ (N^\bullet_T)_{ij}(\gamma) -(N^\bullet_T)_{ij}(\widehat{\gamma}) $ can be estimated as 
\begin{equation}
    \abs{(N^\bullet_T)_{ij}(\gamma) -(N^\bullet_T)_{ij}(\widehat{\gamma})} \leq \kst\sum_{k=1}^{i}\sum_{l=1}^d \pol^i(\gamma, \widehat{\gamma},\gamma_0)\abs{\pax^k\gamma^l - \pax^k \widehat{\gamma}^l}
\end{equation}
Since $W^{\alpha}_p$ is a Banach algebra, \cref{continuity_boundary_op} and \cref{estensione_controllata} imply
\begin{equation}
    \norm{\pol^i(\gamma, \widehat{\gamma},\gamma_0)}_{W^{\alpha}_p} \leq \kst\left(\norm{\gamma}_{\mathbb{E}_T}\norm{\widehat{\gamma}}_{\mathbb{E}_T}\norm{\gamma_0}_{\mathbb{E}_T}+1\right)^\eta \leq \kst(M) 
\end{equation}
for some positive $\eta$
and similarly
\begin{align}
    \norm{(N^\bullet_T)_{ij}(\gamma) -(N^\bullet_T)_{ij}(\widehat{\gamma})}_{W^\alpha_p}&\leq \kst(M) \sum_{k=1}^{i}\sum_{l=1}^d\norm{\pax^k\gamma^l - \pax^k \widehat{\gamma}^l}_{W^{\alpha}_p}\leq \kst(M)\norm{\gamma - \widehat{\gamma}}_{\mathbb{E}_T}.
\end{align}
Our choice of $\theta$ implies that $\alpha<\beta$ and hence there exists a positive constant $\sigma=\sigma(\alpha,\beta)$ such that
\begin{equation}\label{eq:zzzembeddingsigma}
    \norm{\pax \gamma}_{W^{\alpha}_p} \overset{\cref{embeddingsobolevhölder}}{\leq} \kst T^\sigma \norm{\pax \gamma}_{\CS^\beta([0,T])}.
\end{equation}
Therefore, we get
\begin{align}
     \norm{(N^\bullet_T)_{ij}(\gamma) -(N^\bullet_T)_{ij}(\widehat{\gamma})}_{W^\alpha_p}\overset{\cref{eq:c}}&{\leq} \kst \norm{\pax \gamma}_{W^\alpha_p}\norm{(N^\bullet_T)_{ij}(\gamma) -(N^\bullet_T)_{ij}(\widehat{\gamma})}_{W^\alpha_p}\\
     \overset{\cref{eq:zzzembeddingsigma}}&{\leq} \kst T^\sigma \norm{\pax \gamma}_{\CS^\beta([0,T])}\norm{\gamma -\widehat{\gamma}}_{\mathbb{E}_T}\\
     &\leq  \kst T^\sigma \norm{\pax \gamma}_{\CS^\beta([0,T];\CS^1([0,1]))}\norm{\gamma -\widehat{\gamma}}_{\mathbb{E}_T}\\
     \overset{\cref{stima-hölder}}&{\leq} \kst T^\sigma \norm{ \gamma}_{\mathbb{E_T}}\norm{\gamma -\widehat{\gamma}}_{\mathbb{E}_T} \leq \kst(M)T^\sigma \norm{\gamma- \widehat{\gamma}}_{\mathbb{E}_T}. \qedhere
\end{align}
\end{proof}

\begin{theorem}\label{contrazioni}
Let $\boldsymbol{c}$ defined as in  \cref{eq:c} and $\widetilde{T}$ be the time appearing in \cref{prime-stime}. Then, there exists  $M=M(\boldsymbol{c},\Vert\gamma_0\Vert_{W^{2m-\nicefrac{2m}{p}}_p})>0$ and $\widehat{T}\in (0,\widetilde{T}]$ such that for all $T\in (0,\widehat{T}]$ the map $K_T:\mathbb{E}_T^{\gamma_0}\cap\overline{B_M}\to\mathbb{E}_T^{\gamma_0}\cap\overline{B_M}$ is a contraction. 
\end{theorem}
\begin{proof}

\textit{Step 1: \underline{For every $M>0$, $T\in (0,\widetilde{T}]$ the map $K_T:\mathbb{E}_T^{\gamma_0}\cap\overline{B_M}\to\mathbb{E}_T^{\gamma_0}$ is well-defined.}} The map $N_T(\gamma)=\left(N^1_T(\gamma), N^2_T(\gamma),  N^3_T(\gamma),\evaluat{\gamma}_{t=0}\right)$ sends $\mathbb{E}^{\gamma_0}_T$ into $\mathbb{F}^{\gamma_0}_T$ and $K_T$ sends $\mathbb{E}^{\gamma_0}_T$ into $\mathbb{E}_T$. We should only show that $\evaluat{K_T(\gamma)}_{t=0}=\gamma_0$. Directly from the definition:
\begin{equation}
    \evaluat{K_T(\gamma)}_{t=0}=(N_T(\gamma))_4=\evaluat{\gamma}_{t=0}=\gamma_0.
\end{equation}

\textit{Step 2: \underline{Lipschitzianity.}} By the definition of $K_T$, \cref{lem:bound-L,prop:contrazioni1,prop:contrazioni2} there exist $\alpha\in (0,1)$ and a constant $\kst>0$ independent of $T$ such that
\begin{align}
    \Vert K_T(\gamma)-K_T(\widehat{\gamma})\Vert_{\mathbb{E}_T}
   & =\Vert L_T^{-1}\left(N_T(\gamma)-N_T(\widehat{\gamma})\right)\Vert_{\mathbb{E}_T}\\
    &\leq
    \sup_{T\in (0,\widetilde{T}]}\Vert L^{-1}_T\Vert_{\mathscr{L}(\mathbb{F}_T,\mathbb{E}_T)}\Vert N_T(\gamma)-N_T(\widehat{\gamma})\Vert_{\mathbb{F}_T}\\
&\leq \kst T^{\alpha}\Vert \gamma-\widehat{\gamma}\Vert_{\mathbb{E}_T}.
\end{align}
By choosing $\overline{T}$ small enough, we  get $\kst \overline{T}^{\alpha}<1$.

\textit{Step 3: \underline{There exists $M>0$ such that  $K_T$ is a self-mapping from $\mathbb{E}_T^{\gamma_0}\cap\overline{B_M}$.}} Let $T\in (0,\overline{T}]$. Given the $\gamma_0$ initial datum, we construct $\varrho \in \mathbb{E}_T^{\gamma_0}$ as in \cref{lem:nonempty}.	 Then there exists a time $\hat{T}=\hat{T}(\boldsymbol{c},\Vert\gamma_0\Vert_{W^{2m-\nicefrac{2m}{p}}_p})\in (0,T]$ such that for all $t\in [0, \hat{T}]$, $\varrho^i$ are regular curves to which the operator $N_{\hat{T}}$ can be applied. Let

	\begin{equation}
	M:=2\max\left\{\sup_{T\in(0,\hat{T}]}  \Vert L_T^{-1}\Vert_{\mathscr{L}\left(\mathbb{F}_T,\mathbb{E}_T\right)},1
    \right\}\max\left\{\Vert \varrho\Vert_{\mathbb{E}_{\hat{T}}}, \Vert \left(N_{\hat{T}}^1(\varrho),N_{\hat{T}}^2(\varrho),N_{\hat{T}}^3(\varrho),\gamma_0\right)\Vert_{\mathbb{F}_{\hat{T}}}  \right\}.
	\end{equation}	

With this choice of $M$, $\varrho$ lies in $\mathbb{E}_T^{\gamma_0}\cap\overline{B_M}$ for all $T\in (0,1]$ with  $\norm{K_T\left(\varrho\right)}_{\mathbb{E}_T}\leq \nicefrac{M}{2}$.

 We choose a $\widehat{T}\in (0,\hat{T})$ so small that    
	\begin{equation}
	\Vert K_T\left(\gamma\right)-K_T\left(\varrho\right)\Vert_{\mathbb{E}_T}\leq \kst\left(\boldsymbol{c},M\right)T^\beta\Vert\gamma-\varrho\Vert_{\mathbb{E}_T}\leq \kst\left(\boldsymbol{c},M\right)T^\beta 2M\leq \nicefrac{M}{2}.
	\end{equation}
and we conclude that for all $T\in (0,\widehat{T}]$ and $\gamma\in \mathbb{E}_T^{\gamma_0}\cap\overline{B_M}$,
	\begin{equation}
	\Vert K_T(\gamma)\Vert_{\mathbb{E}_T}\leq \Vert K_T(\gamma)-K_T(\varrho)\Vert_{\mathbb{E}_T}+\Vert K_T(\varrho)\Vert_{\mathbb{E}_T}
    \leq\nicefrac{M}{2}+\nicefrac{M}{2}= M.
	\end{equation}

Combining Step 1 to 3, we get that $K_T:\mathbb{E}_T^{\gamma_0}\cap\overline{B_M}\to\mathbb{E}_T^{\gamma_0}\cap\overline{B_M}$ is a contraction.
\end{proof}

\begin{theorem}\label{short time existence}
Let $\gamma_0$ be an admissible initial datum.
Then, there exists a positive time $\widetilde{T}\left(\gamma_0\right)$ depending on $\min_{i\in\{1,\ldots,d\},x\in[0,1]}\abs{\pax\gamma_0^i(x)}$ and $\norm{\gamma_0}_{W_p^{2m-\nicefrac{2m}{p}}\left(0,1\right)}$ such that for all $T\in (0,\widetilde{T}(\gamma_0)]$ the system \cref{gradient_flow} has a solution in $\mathbb{E}_T$	which is unique in $\mathbb{E}_{T}\cap\overline{B_M}$.
\end{theorem}
\begin{proof}
Choosing appropriately, as in \cref{contrazioni}, $M$ and $\widetilde{T}$, the solution to system  \cref{gradient_flow} are the fixed points of the contraction $K_T:\mathbb{E}_{T}\cap\overline{B_M} \to \mathbb{E}_{T}\cap\overline{B_M}$. Thanks to the  Contraction Mapping Principle, such a fixed point exists and it is unique.
\end{proof}

\begin{theorem}[Local-in-time-existence]\label{short-time}
     Let $\gamma_0$ be an admissible initial datum in the sense of \cref{def:admissible-datum}.  Then, there exists $T>0$ and there exists a  solution $\gamma_{t\in[0,T]}$ to the geometric gradient flow \cref{flow} of $\E_m$ in the class $\mathcal{C}(\gamma)$ in the time interval $[0,T]$ with initial datum $\gamma_0$.  
\end{theorem}
\begin{proof}
Thanks to the previous \cref{short time existence}, there exists a unique solution $\gamma_t$ to the analytic problem \cref{gradient_flow} in the short time interval $[0,\widetilde{T}]$. Such a solution is clearly also a solution to \cref{flow}.
\end{proof}

\subsection{Uniqueness}\label{sec:uniqueness}

While the previous section established the existence of a solution to the geometric gradient flow of $\E_m$, the question of uniqueness remains. As noted previously, because we are dealing with a geometric flow, uniqueness is at best expected up to reparameterization (as formalized in \cref{def:uniqueness}). Ideally, this property holds, and indeed, we show that it does in this section by constructing a solution to an auxiliary
PDE system.

\begin{theorem}[Uniqueness up to reparametrization]\label{uniqueness}
Let $T,\widetilde{T}>0$. Suppose that $\gamma_0$ is an admissible initial datum, $\widetilde{\gamma}$ is the solution to \cref{gradient_flow} in $[0,\widetilde{T}]$, and $\gamma$ is a solution to the geometric gradient flow \cref{flow} in $[0,T]$, both with initial datum $\gamma_0$. Then there exists a positive time $\widehat{T}\leq \min\{T,\widetilde{T}\}$ and a reparametrization $\xi^i:[0,\widehat{T}]\times [0,1]\to [0,1]$ such that for all $t\in [0,\widehat{T}]$ 
\begin{equation}
\widetilde{\gamma}^i(t,\xi^i(t,x))=\gamma^i(t,x).
\end{equation}
In other words, the solution to \cref{flow}  is unique in the sense of \cref{def:uniqueness}.
\end{theorem}

\begin{proof}
    We formally derive an equation for $\xi^i$. We omit the superscript $i$. Differentiating in the time variable, we get
    \begin{equation}\label{eq:zzzfirstequationuniqueness}
        \pat \widetilde{\gamma}(t,\xi(t,x))+ \partial_y \widetilde{\gamma}(t,\xi(t,x)) \pat \xi(t,x)= \pat \gamma(t,x) 
    \end{equation}
    Since the normal velocity is the same for both flows, we have
    \begin{align}
        \pat \gamma(t,x) - \pat \widetilde{\gamma}(t,\xi(t,x))&=  [T(t,x) - \widetilde{T}(t,\xi(t,x))]\frac{\pax \gamma(t,x)}{\abs{\pax \gamma(t,x)}}\\
        &= \left(T(t,x)  + (-1)^{m} \ip{\frac{\partial_y^{2m}\widetilde{\gamma}(t,\xi(t,x))}{\abs{\partial_y \widetilde{\gamma}(t,\xi(t,x))}^{2m}},\frac{\pax \gamma(t,x)}{\abs{\pax \gamma(t,x)}}} \right)\frac{\pax \gamma(t,x)}{\abs{\pax \gamma(t,x)}}.
    \end{align}
    Therefore, by \cref{eq:zzzfirstequationuniqueness} we obtain
    \begin{equation}
        \pat \xi(t,x) = \frac{1}{\abs{\partial_y \widetilde{\gamma}(t,\xi(t,x))}}\left(T(t,x)  + (-1)^{m} \ip{\frac{\partial_y^{2m}\widetilde{\gamma}(t,\xi(t,x))}{\abs{\partial_y  \widetilde{\gamma}(t,\xi(t,x))}^{2m}}, \frac{\pax \gamma(t,x)}{\abs{\pax \gamma(t,x)}}} \right).
    \end{equation}
    To conclude, we need to express the right-hand side in terms of the derivatives of $\xi$. To do so, we will use that
    \begin{equation}\label{eq:zzzrelations}
    \begin{split}
        \pax \gamma(t,x)&=\partial_y \widetilde{\gamma}(t,\xi(t,x)) \pax \xi(x,t)\\
        \pax^2 \gamma(t,x)&= \partial^2_y \widetilde{\gamma}(t,\xi(t,x)) (\pax \xi(x,t))^2 + \partial_y \widetilde{\gamma}(t,\xi(t,x)) \pax^2 \xi(x,t)\\
        & \vdots\\
        \pax^{2m} \gamma(t,x) &= \partial^{2m}_y \widetilde{\gamma}(t,\xi(t,x)) (\pax \xi(x,t))^{2m} + \partial_y \widetilde{\gamma}(t,\xi(t,x)) \pax^{2m} \xi(x,t) + \pol^{2m-1}(\gamma,\xi),
    \end{split}
    \end{equation}
    from which one can express the derivatives of $\widetilde{\gamma}$ in terms of the derivatives of $\xi$, the derivatives of $\gamma$, $\abs{\pax\gamma}^{-1}$ and $\abs{\pax\xi}^{-1}$. Hence, we obtain that
    \begin{equation}
        \pat \xi =(-1)^m \frac{\pax^{2m}\xi}{\abs{\pax \gamma}^{2m}} +\left[(-1)^m \frac{\ip{\pax^{2m}\gamma, \pax\gamma}}{\abs{\pax \gamma}^{2m+2}}
        +\frac{T}{\abs{\pax\gamma}}\right]\pax \xi  + p (\gamma,\xi)
    \end{equation}
    where $p(\gamma,\xi)$ depends only on derivatives of $\gamma$ and $\xi$ up to order $2m-1$, $\abs{\pax \xi}^{-1}$, and $\abs{\pax \gamma}^{-1}$.

    \medskip
    We now have to provide $m$ boundary conditions at each boundary point $\bar{x}=0,1$. The first is $\xi(t,\bar{x})=\bar{x}$. For the remaining $m-1$ conditions, we will again employ  \cref{eq:zzzrelations}. Recall that $\abs{\partial_y \widetilde{\gamma}(t,\bar{x})}=1$. Hence, multiplying \cref{eq:zzzrelations} by $\partial_y \widetilde{\gamma}(t, \bar{x})$ and evaluating at $\bar{x}$ we obtain
    \begin{equation}
    \begin{split}
        \pax \xi(t, \bar{x}) &= \abs{\pax\gamma(t,\bar{x})}\\
        \pax^2 \xi(t,\bar{x}) &= \ip{\pax^2 \gamma(t,\bar{x})-\partial_y^2 \widetilde{\gamma}(t,\bar{x})(\pax \xi(t,\bar{x}))^2,\partial_y \widetilde{\gamma}(t, \bar{x})}\\
        & \vdots\\
        \pax^{m-1} \xi(t,\bar{x}) &= \ip{\pax^{m-1} \gamma(t,\bar{x})-\partial_y^{m-1}\widetilde{\gamma}(t,\bar{x})(\pax \xi(t,\bar{x}))^{m-1},\partial_y \widetilde{\gamma}(t, \bar{x})} + \pol^{m-2}(\gamma,\xi),
    \end{split}
    \end{equation}
    where each derivative of $\xi$ is expressed in terms of derivatives of $\gamma$ and $\widetilde{\gamma}$ and lower order derivatives of $\xi$. By substitution, we derived the following system
    \begin{equation}
        \begin{cases}[@{\ }l]
            \pat \xi = (-1)^m \frac{\pax^{2m}\xi}{\abs{\pax \gamma}^{2m}} +\left[(-1)^m \frac{\ip{\pax^{2m}\gamma, \pax \gamma}}{\abs{\pax \gamma}^{2m+2}}
        +\frac{T}{\abs{\pax\gamma}}\right]\pax \xi  + p (\gamma,\xi)\\
        \xi(t,0)=0\\
        \xi(t,1)=1\\
        \pax^j \xi(t,0)= \pol^{j}(\gamma,\widetilde{\gamma})(t,0)\\
        \pax^j \xi(t,1)= \pol^{j}(\gamma,\widetilde{\gamma})(t,1)\\
        \xi(0,x)=x
        \end{cases}
    \end{equation}

This system has the same structure as the main one we have studied earlier in this paper, and following the steps of that argument, we can prove that there exists a $T'>0$ such that the solution $\xi$ is of class $W^{1,2m}_p((0,T')\times (0,1))$ and for all $t\in [0,T']$, for all $ i \in \set{1, \ldots, n}$ we have that $\xi^i:[0,1]\to [0,1]$ is a $C^1$–diffeomorphism. We conclude by taking $\widehat{T}=\min\{T,\widetilde{T},T'\}$.

We omit further details, we refer to \cite[Theorem 5.4]{garcke-menzel-pluda-20} and \cite[Theorem 4.60]{menzelthesis} for the case $m=2$ of the Willmore functional/elastic flow. 
\end{proof}

\subsection{Maximal solutions and parabolic regularization} With local well-posedness established, we can now define the maximal interval on which a solution to the flow exists.

\begin{definition}[Maximal solutions]\label{def:maximal-solution}
 Let $\gamma_0$ be an admissible initial datum in the sense of \cref{def:admissible-datum}. A time-dependent family of networks $\gamma_{t\in[0,T)}$ with $T\in(0,\infty)\cup\{\infty\}$ is a maximal solution to the geometric gradient flow \cref{flow} of $\E_m$ in the class $\mathcal{C}(\gamma)$ with initial datum $\gamma_0$ if
 \begin{itemize}
    \item For every $\hat{T}<T$, the restriction $\gamma_{t\in[0,\hat{T}]}$ is a solution in the sense of \cref{def:flow};
    \item There is no solution $\widetilde{\gamma}_{t\in[0,\widetilde{T}]}$ to \cref{flow} with $\widetilde{T}\geq T$ such that $\gamma = \widetilde{\gamma}$ on $[0,T)$.
 \end{itemize}
In this case, $T$ is called the \emph{maximal time of existence} and is denoted by $T_{max}$.
\end{definition}

\begin{proposition}[Existence and uniqueness of maximal solutions]\label{prop:emax-sol}
Let $\gamma_0$ be an admissible initial datum. Then, there exists a maximal solution to the geometric gradient flow \cref{flow} of $\E_m$ in the class $\mathcal{C}(\gamma)$ with initial datum $\gamma_0$ which is unique in the sense of \cref{def:uniqueness}.
\end{proposition}
\begin{proof}
We define the maximal time of existence as the supremum
\begin{equation}
T_{max} \coloneqq \sup \left\{ T>0 : \text{there exists a solution } \gamma^T_{t\in [0,T]} \text{ to \cref{flow} with initial datum } \gamma_0 \right\}.
\end{equation}
\cref{short-time} ensures that $T_{max}\in(0,\infty)\cup\{\infty\}$. For any $t\in[0,T_{max})$ we consider a solution $\gamma^T_{t\in [0,T]}$ and set
\begin{equation}
\gamma(t) \coloneqq \gamma^T(t), \quad \text{where } T \in (t, T_{max}),
\end{equation}
which is well-defined on $[0,T_{max})$ as any two solutions $\gamma^{T_1}$ and $\gamma^{T_2}$ with $T_1$, $T_2\in[0,T_{max})$ coincide on their common interval of existence by \cref{uniqueness}. Moreover, $\gamma$ satisfies the properties of a maximal solution. Indeed, suppose that there exists  a solution $\widetilde{\gamma}_{t\in [0,\widetilde{T}]}$ with $\widetilde{T}= T_{max}$. \cref{short-time} would imply the existence of a solution with initial datum $\widetilde{\gamma}(\widetilde{T})$ in a time interval $[0,\delta]$ with $\delta >0$. This would yield the existence of a solution in the time interval $[0,\widetilde{T}+\delta]$ with initial datum $\gamma_0$, contradicting the definition of $T_{max}$. Finally, the uniqueness of this maximal solution is an immediate consequence of \cref{uniqueness}.
\end{proof}

Until now, we have obtained a Sobolev solution to our problem. However, our subsequent analysis of long-time behavior requires estimates on higher-order derivatives; consequently, a higher degree of regularity for the solution is required. The following theorem ensures that the flow instantly smoothens the initial data, providing a parametrization that is smooth for all $t>0$.  

\begin{lemma}\label{lem:smoothing}
Let $T$ be positive and $\gamma_0$ be an admissible initial parametrization. Suppose that $\gamma_{t\in[0,T]}$ is a solution to  the geometric gradient flow \cref{flow} of $\E_m$ in the class $\mathcal{C}(\gamma)$ in the time interval $[0,T]$ with initial datum $\gamma_0$. Then, the solution $\gamma$ is smooth for positive times in the sense that
\begin{equation}
    \gamma\in C^\infty\left([\varepsilon,T]\times[0,1]\right)
\end{equation} 
for every $\varepsilon\in(0,T)$.
\end{lemma}

The proof of this result is relatively standard in the literature,
exploiting the parabolic smoothing effects inherent to gradient flows of this type. To this end, one can apply the classical theory presented in~\cite{solonnikov2} provided that the boundary conditions are quasilinear. Alternatively, the core idea of the proof relies on the so-called `parameter trick' due to Angenent \cite{angen3} (see also \cite{daprato-grisvard}). This method has been generalized to several settings that, again, fully cover the case of quasilinear boundary conditions~\cite{lunardi1,lunardi2,Prusssimonett}. In the current paper, there is only one possible boundary condition that is fully nonlinear and is not treated in the aforementioned results: the angle condition that involves $\frac{\pax\gamma}{\abs{\pax \gamma}}$. An adaptation of the parameter trick that allows for the treatment of fully nonlinear boundary terms of the mentioned type is presented both in \cite[Chapter 6.6]{goesswein2019Dissertation} and \cite[Section 4]{GoMePl}.

\bigskip

\begin{proof}[Proof of \cref{thm:buona-positura}]
The theorem follows combining \cref{short-time} with \cref{uniqueness} and \cref{lem:smoothing}.
\end{proof}

\section{Long-time behavior}

In this section $\gamma_t$ is a maximal smooth solution to the geometric gradient flow \cref{flow} in the maximal time interval $[0,T_{\max})$.

\subsection{Special variations}
We consider three special variations that are particularly useful during this section. The first is the time variation, which is $\varphi= \pat \gamma_t$. The second is the tangential variation $\varphi= T \tau$ and the last is the normal variation $\varphi =V \nu$. 

\medskip

Since $\E_m$ smoothly depends on $\gamma$, we have
\begin{align}
    \delta_{\pat \gamma_t}\E_m(\gamma_t)& = \frac{\dd}{\dd t } \E_m(\gamma_t).
\end{align}
Similarly, $\delta_{\pat \gamma} \pas^j \gamma = \pat \pas^j \gamma$ and we know how $\pat$ and $\pas$ commute thanks to \cref{eq:commuteDpas}. It is now easy to prove the following lemma.

\begin{lemma}\label{lem:patNet_is_admissible}
    Let $\gamma_t$ be a maximal solution to the geometric gradient flow \cref{flow} in the class $\mathcal{C}(\Bcnd)$.
    Then, $\pat \gamma_t$ is an admissible variation in the sense of \cref{def:admissible-variation}. 
\end{lemma}

\begin{proof}
    Let $P_t\in \nodes_t$. Observe that $t\mapsto \evaluat{B_{ij}(\gamma_t)}_{P_t} \equiv 0$, for every $t \in [0,T_{\max})$,  $ i \in \set{0, \ldots, m-1}$ and $ j \in \set{1, \ldots, d}$. Hence, $\evaluat{\pat B_{ij}(\gamma_t)}_{P_t}=0$ for all $ i \in \set{0, \ldots, m-1}$ and $j \in \set{1, \ldots, d}$, i.e. $\pat \gamma_t$ is an admissible variation.
\end{proof}

\begin{corollary}\label{co:energy-decrease}
    If $\Bcnd(\gamma,\psi)$ is admissible for $\gamma_t$, the energy $t \mapsto \E(\gamma_t)$ is decreasing and 
    \begin{equation}
         \pat \E(\gamma_t) = - \int_{\gamma_t} V^2 \dif \gamma_t.
    \end{equation}
    In particular, $\E(\gamma_t) \leq \E(\gamma_0)<+ \infty$.
\end{corollary}

\begin{proof}
Since $\pat \gamma_t = V \nu + T \tau$, by \cref{lem:patNet_is_admissible} and \cref{thm:gradient_flow}, we have
   \begin{equation}
        \pat \E(\gamma_t) = -\int_{\gamma_t} \ip{\pat\gamma_t, V\nu} \dif \gamma_t = -\int_{\gamma_t} V^2 \dif \gamma_t.\qedhere
    \end{equation}
\end{proof}

The tangential variation takes into account the reparameterization of the curve due to the tangential component of the velocity.
\begin{lemma}
   For $j\in\N$ we have
\begin{equation}\label{commutationT}
    \delta_{T \tau} \pas^j \gamma = T \pas^{j+1} \gamma.
\end{equation}
Moreover, 
    \begin{equation}
        \delta_{T \tau}(\dd \gamma) = \pas T \dif \gamma.
    \end{equation}
\end{lemma}
\begin{proof}
    We prove \cref{commutationT} by induction.     By definition $\delta_{T \tau}\gamma = T \tau = T \pas \gamma$. Now, assume the statement is true for some $j$, then
    \begin{equation}
        \delta_{T \tau} \pas^{j+1} \gamma = \pas \delta_{T \tau} \pas^{j} \gamma - \pas T \pas^{j+1} \gamma = \pas T \pas^{j+1}\gamma + T \pas^{j+2} \gamma -\pas T \pas^{j+1} \gamma = T \pas^{j+2} \gamma. \qedhere
    \end{equation}
\end{proof}

The normal variation is the most geometric one. Indeed, it only sees the evolution of the support of the curves since $V\nu = \pat \gamma - T\tau$.

\begin{lemma}\label{lem:Vnuexchange}
For $j\geq 1$ we have
\begin{equation}\label{commutationV}
     \delta_{V\nu} \pas^j \gamma   = \pas \delta_{V\nu} \pas^{j-1} \gamma + \ip{V\nu , \pas^2 \gamma}\pas^{j} \gamma.    
\end{equation}
Moreover, 
\begin{equation}
          \delta_{V\nu} (\dd \gamma) = \ip{\pas(V\nu), \pas \gamma}\dd \gamma.
\end{equation}
\end{lemma}
\begin{proof}
We combine  \cref{eq:commuteDpas} with \cref{commutationT}:
\begin{align}
    \delta_{V\nu} \pas^j \gamma  &= \delta_{\partial_t\gamma} \pas^j \gamma -  \delta_{T\tau} \pas^j \gamma = \pas \pat \pas^{j-1}\gamma -\ip{\pas \gamma,\pas\pat\gamma}\pas^j\gamma-T\pas^{j+1}\gamma\\
      &= \pas \pat \pas^{j-1}\gamma +V\ip{\nu,\pas^2\gamma }\pas^{j}\gamma -\pas T\pas^j\gamma-T\pas^{j+1}\gamma\\
      &=\pas \left(\pat \pas^{j-1}\gamma-T\pas^j\gamma\right)+V\ip{\nu,\pas^2\gamma }\pas^{j}\gamma\\
      &=\pas \delta_{V\nu} \pas^{j-1} \gamma + V\ip{\nu , \pas^2 \gamma}\pas^{j} \gamma . \qedhere
    \end{align}
\end{proof}

Applying repeatedly \cref{commutationV}
for every $j\geq1 $ and $ k \in \set{0, \ldots, j}$
\begin{equation}
        \delta_{V\nu} \pas^j \gamma   = \pas^k \delta_{V\nu} \pas^{j-k} \gamma + \sum_{i=0}^{k-1}\pas^i\left(\ip{V\nu , \pas^2 \gamma}\pas^{j-i} \gamma \right).
\end{equation}
In the following we will apply often the lemma in this exact form.

\bigskip

In \cref{co:energy-decrease} we get control of the energy from its monotonicity along the flow. Similarly, we aim to gain control over the velocity by examining its evolution along the flow. As a first step, we derive its analytical expression in a concise yet still meaningful way.

\begin{lemma}\label{lem:derivata_normale_V}
It holds
\begin{equation}\label{eq:derivata_normale_V}
\begin{split}
    \frac{\dd}{\dd t}\int_{\gamma_t} V^2 \dif {\gamma_t} &=-\int_{\gamma_t} \abs{\pas^m V}^2 + (\pas^m V)\pol^{3m-1}_{3m-1}(\gamma_t) + \pol^{3m-1}_{6m-2}(\gamma_t)\dif {\gamma_t} \\&\qquad+\evaluat{TV^2}_{\nodes_t}+\sum_{j=0}^{m-1} (-1)^{j}\evaluat{\ip{\delta_{V\nu} \pas^{m-j-1} \gamma_t, \pas^{j} \delta_{V\nu} \pas^m \gamma_t}}_{\nodes_t}.
    \end{split}
\end{equation}
\end{lemma}

\begin{proof} Observe that
    \begin{align}
        \pat \int_{\Net_t} V^2 \dif {\Net_t} &= \delta_{V \nu}\int_{\gamma_t} V^2 \dif {\gamma_t} + \delta_{T\tau}\int V^2 \dif \gamma_t\\
        &= \delta_{V \nu}\int_{\gamma_t} V^2 \dif {\gamma_t} + \int_{\gamma_t}T \pas (V^2) + V^2\pas T\dif \gamma_t \\
           &= \delta_{V \nu}\int_{\gamma_t} V^2 \dif {\gamma_t} + \int_{\gamma_t} \pas (TV^2)\dif \gamma_t \\
        &= \delta_{V \nu}\int_{\gamma_t} V^2 \dif {\gamma_t} + \evaluat{TV^2}_{\nodes_t}.
    \end{align}
Hence, we need just to compute the variation of the velocity in direction $V \nu$. By Leibniz rule
\begin{align}
     \delta_{V\nu} \int_{\gamma_t} V^2 \dif \gamma_t &=\int_{\Net_t} 2 \ip{V \nu, \delta_{V\nu} (V \nu )} - V^3\ip{\nu, \pas^2 \gamma_t} \dif \gamma_t\\&= \int_{\gamma_t} 2 \ip{V \nu, \delta_{V\nu} (V \nu )} + \pol^{2m}_{6m-2}(\gamma_t) \dif \gamma_t
\end{align}
Recall that $V \nu = \vec{V} (\gamma_t) $, where $\vec{V}$ is defined in \cref{prop:first_variation_energy}. By \cref{lem:Vnuexchange} we have
\begin{align}
    \delta_{V \nu} \vec{V}(\gamma_t) &= (-1)^{m+1} \pas^m  \delta_{V\nu} \pas^m \gamma_t- \sum_{j=0}^{m-1} (-1)^j \pas \left[ \ip{\pas^{m-j}  \gamma_t, \pas^j  \delta_{V\nu} \pas^m \gamma_t} \pas \gamma_t\right] \\&\qquad +(\pas^m V)\pol^{3m-1}_{m}(\gamma_t)+ \pol^{3m-1}_{4m-1}(\gamma_t)
\end{align}

By Leibniz rule since $\ip{\nu, \pas \gamma_t} =0$ and \cref{commutationV}, we get
\begin{align}
    \ip{ V \nu, \pas\left[ \ip{\pas^{m-j} \gamma_t, \pas^j \delta_{V \nu} \pas^m \gamma_t} \pas \gamma_t\right]} &= - \ip{\pas( V \nu),  \pas \gamma_t}\ip{\pas^{m-j} \gamma_t, \pas^j \delta_{V \nu} \pas^m \gamma_t}\\& = \ip{\delta_{V \nu} \pas^{m-j} \gamma_t - \pas\delta_{ V \nu} \pas^{m-j-1} \gamma_t, \pas^j \delta_{V \nu} \pas^m \gamma_t}.
\end{align}
Integrating the last term by part and rearranging, we get
\begin{align}
    &\int_{\gamma_t}\ip{ V \nu, \pas\left[ \ip{\pas^{m-j} \gamma_t, \pas^j \delta_{V \nu} \pas^m \gamma_t} \pas \gamma_t\right]} -\ip{ \delta_{ V \nu} \pas^{m-j-1} \gamma_t,\pas^{j+1} \delta_{V \nu} \pas^m \gamma_t } \dif \gamma_t  \\&\qquad = \int_{\gamma_t}  \ip{\delta_{V \nu} \pas^{m-j} \gamma_t,\pas^j \delta_{V \nu} \pas^m \gamma_t} \dif \gamma_t - \ip{\delta_{V\nu} \pas^{m-j-1}\gamma_t, \pas^{j} \delta_{V\nu} \pas^{m} \gamma_t}.
\end{align}
Applying the formula for $ j \in \set{0, \ldots, m-1}$, we get
 \begin{align}
     \int_{\gamma_t}\ip{V \nu, \delta_{V\nu} (V \nu )} &= -\int_{\gamma_t} \abs{\delta_{V\nu} \pas^m \gamma}^2 + (\pas^m V) \pol^{3m-1}_{3m-1}(\gamma_t)+ \pol^{3m-1}_{6m-2}(\gamma_t) \dif \gamma_t \\&\qquad +\sum_{j=0}^{m-1} (-1)^{j}\ip{\delta_{V\nu} \pas^{m-j-1} \gamma_t, \pas^{j} \delta_{V\nu} \pas^m \gamma_t},
 \end{align}
 and therefore the conclusion since
 \begin{equation}
        \abs{\delta_{V\nu} \pas^m \gamma}^2 =\abs{\pas^m V}^2 + (\pas^m V) \pol^{3m-1}_{3m-1}(\gamma_t) + \pol^{3m-1}_{6m-2}(\gamma_t). \qedhere
 \end{equation}
\end{proof}

\subsection{Boundary terms}

In this section, we focus on the boundary contribution in \cref{eq:derivata_normale_V}, keeping track of the tangential part of the velocity $T$. While the normal velocity is defined by the problem, we do not have control over the tangential velocity. However, we will show a geometric control over $T$ at the boundary.

\begin{lemma}\label{lem:trattamento_bordo}
    Assume $\det\ndA(\nu_t)\neq 0$ where $\ndA(\nu_t)$ is defined in \cref{eq:nondegeneracy_system}. There holds
    \begin{align}\label{eq:final_bondary_term}
        \sum_{i=0}^{m-1} (-1)^i \ip{\delta_{V\nu} \pas^{m-i-1} \gamma_t, \pas^i \delta_{V\nu} \pas^{m} \gamma_t} = \pol^{3m-1}_{6m-3}(\gamma_t) + \pol^{3m-1}_{4m-2}(\gamma_t)T+\pol^{3m-1}_{2m-1}(\gamma_t) T^2
    \end{align}

\end{lemma}
\begin{proof}
By \cref{lem:Vnuexchange} and Leibniz rule
    \begin{equation}
       \ip{V\nu, \pas^{m-1}\delta_{V\nu}\pas^m \gamma_t} - \ip{V \nu, \delta_{V\nu} \psi^0(\gamma_t)} = \pol^{3m-1}_{6m-3} (\gamma_t^j).
    \end{equation}
and
\begin{equation}
    \ip{\delta_{V\nu} \pas^{m-i-1} \gamma_t, \pas^i\delta_{V\nu} \pas^m \gamma_t}- \ip{\delta_{V\nu} \pas^{m-i-1} \gamma_t, \delta_{V\nu}\pas^{m+i} \gamma_t}= \pol^{3m-1}_{6m-3} (\gamma_t).
\end{equation}
Hence, the left-hand side of \cref{eq:final_bondary_term} can be rewritten as
\begin{equation}\label{eq:zzz_boundary_commutedV}
     \ip{V\nu, \delta_{V\nu} \psi^0(\gamma_t)} +\sum_{i=0}^{m-2} (-1)^i \ip{\delta_{V\nu} \pas^{m-i-1} \gamma_t,  \delta_{V\nu} \pas^{m+i} \gamma_t} + \pol^{3m-1}_{6m-3}(\gamma_t).
\end{equation}
Recall the definition of $\psi^0$ in \cref{prop:first_variation_energy}. Straightforward computations give
\begin{equation}
    \ip{\pas \gamma_t, \delta_{V\nu} \psi^0(\gamma_t)}+\sum_{i=0}^{m-2} (-1)^i \ip{\pas^{m-i}\gamma_t, \delta_{V\nu} \pas^{m+i} \gamma_t}  =  \pol^{3m-1}_{4m-2} (\gamma_t).
\end{equation}
Multiplying the above expression by $T$ and summing the result to \cref{eq:zzz_boundary_commutedV}, we get
\begin{equation}\label{eq:zzz_boundary_addedtangential}
    \underbrace{\ip{\pat\gamma_t, \delta_{V\nu} \psi^0(\gamma_t)}\vphantom{\sum_{i=0}^{m-2}}}_{\mathrm{(I)}} +\underbrace{\sum_{i=0}^{m-2} (-1)^i \ip{\pat \pas^{m-i-1} \gamma_t,  \delta_{V\nu} \pas^{m+i} \gamma_t}}_{\mathrm{(II)}} +  \pol^{3m-1}_{6m-3}(\gamma_t)+T \pol^{3m-1}_{4m-2} (\gamma_t).
\end{equation}
We now analyze the two terms separately. Since $\det\ndA\neq 0$, we have
\begin{align}
    \mathrm{(I)}&=\sum_{j,i,k=1}^{d}\ip{\pat \gamma^j_t, \ip{\ndA^{-1}_{ij} \pi_{ik}\delta_{V\nu} \psi^0(\gamma^k_t),\nu^j}\nu^j}\\
    & =\sum_{j=1}^{d}\ip{\pat \gamma^j_t, \delta_{V\nu}B_{(2m-1)j}(\gamma_t)\nu^j}   + \pol^{3m-1}_{4m-2}(\gamma_t) T^j+ \pol^{3m-1}_{6m-2}(\gamma_t). 
\end{align}
Since $\pat B_{(2m-1)j}=0$ and $\pas B_{(2m-1)j} = \pol^{2m}_{2m-1}(\gamma_t)$ we conclude
\begin{equation}
    \mathrm{(I)} = T^2 \pol^{3m-1}_{2m-1}(\gamma_t) + T \pol^{3m-1}_{4m-2}(\gamma_t) +\pol^{3m-1}_{6m-3}(\gamma_t)
\end{equation}

The second part is a little more delicate. First observe that $\pat\pas^l \gamma_t \in \ker b$ since
\begin{equation}
    \sum_{l=1}^{m-1} b_{kl}(\gamma_t)\pat \pas^{l} \gamma_t = \frac12 \pat \pas^{k-1} \abs{\pas \gamma}^2 =0.
\end{equation}
 By \cref{lem:complement_of_b},  we then have
\begin{equation}\label{eq:zzz_boundary_complement_b}
        \mathrm{(II)}=\sum_{i=1}^{m-1}\sum_{k=1}^{m-1}\ip{\pat \pas^i \gamma_t,\nu}\ip{(-1)^{m-k-1} b_{ik}^\perp \delta_{V\nu} \pas^{2m-k-1} \gamma_t ,\nu}  + \pol^{3m-1}_{6m-3}(\gamma_t).
\end{equation}
Recall the expression of $\psi^k(\gamma_t)$ presented in \cref{psi}. Observe that
\begin{align}
     (-1)^{m-k-1} b_{ik}^\perp \delta_{V\nu} \pas^{2m-k-1} \gamma_t 
     = \delta_{V\nu}\psi^k(\gamma_t) + \pol^{3m-1}_{4m-i-2}(\gamma_t),
\end{align}
that substituting in \cref{eq:zzz_boundary_complement_b} gives
\begin{equation}
    \mathrm{(II)}= \sum_{i=1}^{m-1} \ip{\pat \pas^i \gamma_t,\nu}\ip{\delta_{V\nu} \psi^i( \gamma_t),\nu} + \pol^{3m-1}_{6m-3}(\gamma_t) + \pol^{3m-1}_{4m-2}(\gamma_t) T.
\end{equation}
Since $\pat B_{ij} = 0$, $\ip{\pat \pas^i \gamma_t^j, \nu^j}$ belongs to $\ker a$. Then,
\begin{align}
    \mathrm{(II)} &= \sum_{i,k=1}^{m-1} \sum_{j,l=1}^{d} a_{ijkl}^\perp(\gamma_t)\ip{\pat \pas^i \gamma^j_t, \nu_t^j}\ip{ \delta_{V\nu}\psi^k(\gamma^l_t), \nu^l}\\
    &=\sum_{i=1}^{m-1} \sum_{j=1}^{d} \ip{\pat \pas^i \gamma^j_t, \nu_t^j}\delta_{V\nu}B_{(2m-i-1)j}(\gamma_t) + \pol^{3m-1}_{6m-2}(\gamma_t) + \pol^{3m-1}_{4m-2}(\gamma_t) T^j
\end{align}
Using that $\pat B_{(2m-i-1)j} =0$ and $\pas B_{(2m-i-1)j} = \pol^{2m-i}_{2m-i-1}$, we obtain that
\begin{equation}
    \mathrm{(II)}=  T^2 \pol^{3m-1}_{2m-1}(\gamma_t) + T \pol^{3m-1}_{4m-2}(\gamma_t) +\pol^{3m-1}_{6m-3}(\gamma_t).
\end{equation}
which concludes the proof.
\end{proof}

\begin{definition}[Uniform non--degeneracy]\label{def:uniform-non-deg}
Let $I\subseteq \R$ be an interval. A solution $\gamma_t$ to the geometric gradient flow \cref{flow} in the class $\mathcal{C}(\Bcnd)$ in the time interval $I$ is \emph{uniformly non-degenerate} if
\begin{enumerate}
    \item the length of each curve $\gamma^i_t \in \gamma_t$ is uniformly bounded from below, and 
    \item there exists $\rho>$ such that at every $d$-junction $P$, we have 
    	\begin{equation}
	\inf_{t\in I}\det\ndA(\nu_t)\geq \rho>0.
	\end{equation}
where $\ndA(\nu_t)$ is defined in \cref{eq:nondegeneracy_system}.
\end{enumerate}
\end{definition}

\begin{lemma}\label{lem:combinazione_lineare}
Let $\gamma_t$ be a uniformly non-degenerate solution to the flow in the time interval $I$. Then, for every $t\in I$ the tangential velocity $T^i$ of the $l$-th curve at a $d$-junction can be expressed as a linear combination with coefficients uniformly bounded in time of the normal velocities $V^1,\ldots,V^d$ of all the curves meeting at the junction.
\end{lemma}
\begin{proof}
    Since 
    \begin{equation}
        \sum_{k=1}^d a_{jk}\pat \gamma_t^k = \pat B_{0j}(\gamma_t) = 0
    \end{equation}
    for all $ j \in \set{1, \ldots, d}$, $\pat \gamma_t \in \ker a$. Hence 
    \begin{equation}
        \pat \gamma^l_t = \sum_{j,k=1}^d\ndA^{-1}_{lj} \diag(\nu \nu^\dagger)_{jk} \pat \gamma^k_t = \sum_{j=1}^d\ndA^{-1}_{lj} V^j\nu^j
    \end{equation}
    Recalling that $\pat \gamma_t^l = V^l\nu^l + T^l \tau^l$, we conclude
    \begin{equation}
        T^l = \sum_{j=1}^d\ip{[\ndA^{-1}_{lj} -\delta_{lj}] V^j\nu^j, \tau^l} \qquad \text{with} \qquad \abs{\ndA^{-1}_{lj} -\delta_{lj}} \leq 1 + \frac{1}{\rho}\max_{j,k} \abs{a_{jk}}.\qedhere
    \end{equation}
\end{proof}

\begin{lemma}\label{lem:evolutionV2}
It holds
\begin{equation}
     \frac{\dd}{\dd t} \int_{\gamma_t} V^2 \dif {\gamma_t} =-\int_{\gamma_t} (\pas^m V)^2 + (\pas^m V)\pol^{3m-1}_{3m-1}(\gamma_t) + \pol^{3m-1}_{6m-2}(\gamma_t)\dif {\gamma_t}+\evaluat{\pol^{3m-1}_{6m-3}(\gamma_t)}_{\nodes_t}.\label{eq:derivata_normale_V_finale}
\end{equation}
\end{lemma}

\begin{proof}
    The result follows directly combining \cref{lem:derivata_normale_V} with \cref{lem:trattamento_bordo} and \cref{lem:combinazione_lineare}.
\end{proof}

\subsection{Gagliardo-Nirenberg type estimate}
We now derive some Gagliardo-Nirenberg type estimates. It is noteworthy that Gagliardo-Nirenberg's inequality is applied for derivatives with respect to the arclength parameter. Hence, the constants appearing in this section also depend on the inverse of the length of each curve. We will omit this dependency when the estimate is for a fixed time.

\begin{lemma}\label{lem:boundedness_of_low_derivatives}
    For every $ j \in \set{2, \ldots, m-1}$ we have
    \begin{equation}
         \norm{\pas^j \gamma}_{L^\infty} \leq \kst (\norm{\pas^m \gamma}_{L^2}^{\frac{2(j-1)}{2m-3}} + 1 ).
    \end{equation}
\end{lemma}

\begin{proof}
    By Gagliardo-Nirenberg inequality, we have
    \begin{equation}
        \norm{\pas^j \gamma}_{L^\infty} \leq \kst \left( \norm{\pas^{m} \gamma}_{L^2}^{\frac{2(j-1)}{2m-3}} \norm{\pas \gamma}_{L^\infty}^{\frac{2(m-j)-1}{2m-3}} + \norm{\pas \gamma}_{L^\infty}\right).\qedhere
    \end{equation}
\end{proof}

\begin{lemma}\label{lem:Lpboundpolinomial}
    Let $\sigma \geq0$, $k\geq m$, and $p\in[1,+\infty]$, then
    \begin{equation}
        \norm{\pol^{k}_\sigma(\gamma)}_{L^p}\leq\kst \left(\norm{\pas^m\gamma}_{L^2}+1\right)^{\xi}\left(\norm{\pas^{k+1}\gamma}_{L^2}^{\frac{(k-m+\frac{1}{2})\sigma}{(k-1)(k-m+1)}-\frac{1}{p(k-m+1)}} +1\right)
    \end{equation}  
    for some finite constant $\kst >0$ depending on $m$, $k$, $\sigma$, and $p$ and some $\xi $ depending only on $m$, $k$, and $\sigma$. In particular, if
    \begin{equation}
        \frac{(k-m+\frac{1}{2})\sigma}{(k-1)(k-m+1)}-\frac{1}{p(k-m+1)}<2,
    \end{equation}
    for every $\varepsilon>0$ there exists a constant $\kst>0$ depending on $\varepsilon$, $m$, $k$, $\sigma$, and $p$ such that
    \begin{equation}\label{eq:PPlpbound}
         \norm{\pol^{k}_\sigma(\gamma)}_{L^p}\leq \varepsilon\norm{\pas^{k+1}\gamma}_{L^2}^2 + \kst\left(\norm{\pas^m\gamma}_{L^2}+1\right)^{\xi},
    \end{equation}
    where $\xi$ can now also depend on $p$.
\end{lemma}

\begin{proof}
    By \cref{lem:boundedness_of_low_derivatives}, derivatives up to $(m-1)$-order are uniformly bounded. Without loss of generality, we can assume that
    \begin{equation}
        \abs{\pol^{k}_{\sigma}} \leq \prod_{j=m}^{k} \abs{\pas^j \gamma}^{\beta_j} \qquad \text{with } \sum_{j=m}^{k} (j-1)\beta_j\leq \sigma.
    \end{equation}
    If $p<+\infty$ we first have to use H\"older inequality to get
    \begin{equation}
        \norm{\pol^{k}_\sigma}_{L^p} \leq \prod_{j=m}^{k}\norm{\abs{\pas^j \gamma}^{p\beta_j}}^{\frac{1}{p}}_{L^{k-m+1}} = \prod_{j=m}^{k}\norm{\pas^j \gamma}^{\beta_j}_{L^{(k-m+1)p \beta_j}}.
    \end{equation}
    Fix $j$ and apply Gagliardo-Nirenberg inequality  to get
    \begin{equation}
        \norm{\pas^j \gamma}^{\beta_j}_{L^{(k-m+1)p \beta_j}}\leq\kst \left( \norm{\pas^{k+1} \gamma}^{\theta_j}_{L^2} \norm{\pas^{m}\gamma}_{L^2}^{1- \theta_j} + \norm{\pas^m \gamma}_{L^2} \right)^{\beta_j},
    \end{equation}
    where 
    \begin{equation}
        (k-m+1)\theta_{j} = j-m + \frac12- \frac1{(k-m+1)p\beta_j}.
    \end{equation}
    Since $m\geq 2$
    \begin{align}
        (k-m+1)\sum_{j=m}^{k} \beta_j \theta_j &= \sum_{j=m}^k(j-1)\beta_j+\left(\frac{3}{2}-m\right) \sum_{j=m}^k\beta_j - \frac{1}{(k-m+1)p}\\
        &\leq \sigma + \frac{\frac32-m}{k-1}\sigma -\frac{1}{(k-m+1)p}=\frac{k-m+\frac{1}{2}}{k-1}\sigma - \frac{1}{(k-m+1)p}
    \end{align}

    The lemma follows, since 
    \begin{equation}
        \max\left\{x^{1- \theta_j}, x \right\}\leq 1 + x
    \end{equation}
    and
    \begin{equation}
        \prod_{j=m}^{k} (1+ x^{\theta_j})^{\beta_j} \leq \kst \left( 1+ x^{\frac{(k-m+\frac{1}{2})\sigma}{(k-1)(k-m+1)}-\frac{1}{p(k-m+1)}}\right)
    \end{equation}
    for $x\geq 0 $. \cref{eq:PPlpbound} follows by Peter-Paul's inequality.
\end{proof}

\begin{proposition}\label{lem:comparsionderivativesgammaV} Let $(\gamma_t)_{t \in[0, T_{\max})}$ be a uniformly non-degenerate maximal solution of the geometric gradient flow \cref{flow}. For every $j\geq 0$, we have
    \begin{equation}
        \norm{\pas^{j+2m} \gamma}_{L^2}^2 \leq \kst \left(\norm{\pas^j V}_{L^2}^2 + 1 \right).
    \end{equation}
    where the constant depends on $m$, $j$, $\E_m(\Net_0)$ and the inverse of the length of each curve.
\end{proposition}
\begin{proof}
    Notice that
    \begin{equation}
        \pas^{j+2m}\gamma-\pas^kV \nu = \pas^j(\pas^{2m} \gamma - V\nu) +\pol^{j+2m-1}_{j+2m-1} = \pas^j(\ip{\pas^{2m}\gamma, \pas\gamma}\pas\gamma) +\pol^{2m+j-1}_{2m+j-1} =\pol^{2m+j-1}_{2m+j-1},
    \end{equation}
    since
    \begin{equation}
        0= \pas^{2m-2}\ip{\pas^2 \gamma, \pas \gamma} = \ip{\pas^{2m}\gamma,\pas\gamma}+ \pol^{2m-1}_{2m-1}.
    \end{equation}
    Now
    \begin{equation}
        \norm{\pas^{j+2m} \gamma}_{L^2}^2 \leq 4 \norm{\pas^j V}_{L^2}^2 + 4\norm{\pas^{j+2m}\gamma-\pas^j V\nu}_{L^2}^2=4 \norm{\pas^j V}_{L^2}^2 + 4\norm{\pol^{2m+j-1}_{4m+2j-2}}_{L^1}.
    \end{equation}
    By \cref{lem:Lpboundpolinomial} with $k=2m+j-1$, $p=1$ and $\sigma=4m+2j-2$, we get
    \begin{equation}
        \norm{\pas^{j+2m} \gamma}_{L^2}^2\leq 4 \norm{\pas^j V}_{L^2}^2 + \varepsilon^2\norm{\pas^{j+2m} \gamma}_{L^2}^2 + \kst\left(\E_m(\gamma_t) + 1\right)^\xi.
    \end{equation}
    The lemma corollary follows by \cref{co:energy-decrease} and choosing $\varepsilon\leq \frac{1}{2}$.
\end{proof}

\begin{theorem}\label{thm:boundednessevolutionV2}
    Let $(\gamma_t)_{t \in[0, T_{\max})}$ be a uniformly non-degenerate maximal solution of the geometric gradient flow \cref{flow}.
    We have
    \begin{equation}
        \frac{\dd}{\dd t } \int_{\Net_t} V^2 \dif \Net_t \leq \kst
    \end{equation}
    where $\kst$ only depends on $m$, $\E_m(\Net_0)$, and the inverse of the length of each curve.
\end{theorem}

\begin{proof}
By the Peter-Paul inequality, we have
\begin{align}
    \frac{\dd}{\dd t} \int_{\Net_t} V^2 \dif {\Net_t} &=-\int_{\Net_t} (\pas^m V)^2 + (\pas^m V)\pol^{3m-1}_{3m-1}(\Net_t) + \pol^{3m-1}_{6m-2}(\Net_t)\dif {\Net_t}+\evaluat{\pol^{3m-1}_{6m-3}(\Net_t)}_{\nodes_t}\\
    &\leq -\frac{1}{2}\int_{\Net_t} (\pas^m V)^2 \dif \Net_t+ \norm{\pol^{3m-1}_{6m-2}(\Net_t)}_{L^1} + \norm{\pol^{3m-1}_{6m-3}(\Net_t)}_{L^\infty}.
\end{align}
Applying \cref{lem:Lpboundpolinomial} with $k=3m-1$, $\sigma =6m-2$, and $p=1$ for the second  term, and $k=3m-1$, $\sigma=6m-3$, and $p=\infty$ for the third term we get
\begin{align}
    \frac{\dd}{\dd t} \int_{\Net_t} V^2 \dif {\Net_t} &\leq -\frac{1}{2}\int_{\Net_t} (\pas^m V)^2 + \varepsilon \abs{\pas^{3m} \gamma}^2\dif \Net_t+\kst \left( \E(\Net_t) + 1\right)^\xi.\\
    &\leq\left(-\frac{1}{2}+ \varepsilon \kst\right)\int_{\Net_t} (\pas^m V)^2 \dif \Net_t + \kst \left( \E(\Net_t) + 1\right)^\xi.
\end{align}
In the second inequality, we used \cref{lem:comparsionderivativesgammaV}.
By \cref{lem:evolutionV2}. We conclude by \cref{co:energy-decrease} and choosing $\varepsilon$ so that $\varepsilon\kst < \frac{1}{4}$.    
\end{proof}

\subsection{Long-time behavior}
From now on, we will consider the geometric gradient flow of the functional
\begin{equation}
    \E = \E_m+ \E_1 = \int_{\Net} \frac{1}{2}\abs{\pas^m \Net}^2 \dif \Net + \frac{1}{2}\ell(\Net).
\end{equation}

 \cref{thm:boundednessevolutionV2} also holds for this flow, replacing the velocity with $V=V_m + V_1$, defined in \cref{eq:normalvelocity}. As a consequence of this choice, we also have that
\begin{equation}\label{eq:lengthbounded}
    \ell(\Net_t) \leq 2\E(\Net_t) \leq 2\E(\Net_0)
\end{equation}
by \cref{co:energy-decrease}.
\bigskip

We now rephrase \cref{thm:longtime-intro} in a more precise way.

\begin{theorem}\label{thm:longtime}
Let  $\Net_0$ be an admissible initial network.  Suppose that $\left(\Net_t\right)_{t\in [0,T_{\max})}$ is a maximal solution to the geometric gradient flow of $\E$ in the class $\mathcal{C}(\Bcnd)$ with initial datum $\gamma_0$. 
Then, 
\begin{equation}
    T_{\max}=+\infty
\end{equation}
or at least one of the following happens:
\begin{enumerate}
\item  there exists a curve $\gamma^i_t$ such that
\begin{equation}
    \liminf_{t \to T^{-}_{\max}} \ell(\gamma^i_t) =0;
\end{equation}
\item there exists a $d$-junction such that
\begin{equation}
    \liminf_{t\to T_{\max}^-} \det\ndA(\nu_t)=0,
\end{equation}
where $\ndA(\nu_t)$ is defined in \cref{eq:nondegeneracy_system}.
\end{enumerate}
\end{theorem}

\begin{proof}[Proof of \cref{thm:longtime-intro}]
    Let  $\Net_t$ be a maximal solution to the flow in $[0,T_{\max})$. Suppose by contradiction that $T_{\max}$ is finite and assume that for all $t\in [0,T_{\max})$: 
    \begin{enumerate}
        \item there exists a positive constant $\kst_1$ independent of time such that  \begin{equation}\label{eq:bound_uniforme_lunghezza}
            \ell(\gamma^i_t)\geq \kst_1 >0\qquad\text{for all }i;
        \end{equation}
        \item there exists a positive constant $\kst_2$ independent of time such that
        \begin{equation}
            \det\ndA(\nu_t) \geq \kst_2 >0 \qquad \text{ at every $d$-junction},
        \end{equation}
    where $\ndA(\nu_t)$ is defined in \cref{eq:nondegeneracy_system}
    \end{enumerate}
    Then, $\gamma(t)$ is uniformly non-degenerate in $[0,T_{\max})$.
    
    Fix $\varepsilon \in (0, T_{\max}/10)$ and $\delta \in (0,T_{\max}/4)$. We can apply \cref{thm:boundednessevolutionV2} with a constant uniformly bounded in $[0,T_{\max})$: 
    \begin{equation}
        \frac{\dd}{\dd t}\int_{\Net_t} V^2\dif \gamma_t \leq \kst.
    \end{equation}

Integrating on the interval $(\varepsilon,t)\subseteq (\varepsilon, T_{\max}- \delta)$, and recalling the definition of $V$, we get 
    \begin{equation}\label{eq:zzzL2Vintegrated}
        \int_{\gamma_t} V^2 \dif \gamma_t \leq \int_{\gamma_\varepsilon} V^2 \dif \gamma_\varepsilon + \kst (T_{\mathrm{max}}- \varepsilon).
    \end{equation}    

    By \cref{lem:comparsionderivativesgammaV} for $j=0$, we also have
    \begin{equation}\label{eq:zzzgammavsV}
        \int_{\gamma_t} \abs{\pas^{2m} \gamma_t}^2 \dif \gamma_t \leq \kst\left(\int_{\gamma_t} V^2  \dif \gamma_t  +1 \right)
    \end{equation}
    for some constant $\kst>0$ depending only on $m$, $\E(\gamma_0)$ and the inverse of the length of each curve. Hence, \cref{eq:zzzgammavsV,eq:zzzL2Vintegrated} give $\pas^{2m}\gamma_t\in L^\infty\left((\varepsilon,T_{\max}-\delta);L^2(0,\ell(\gamma_t))\right)$.
    By interpolation 
    \begin{equation}
        \norm{\pas \gamma_t}_{W^{2m-1}_{2}(0, \ell(\gamma_t))}\leq \kst
    \end{equation}
    for every $t\in(\varepsilon, T_{\max} -\delta)$, where $\kst$ depends only on $\varepsilon$ and $T_{\max}$.
    
    We reparametrize each curve $\gamma_t^i \in \Net_t$ so that $\abs{\pax \gamma_t^i} = \ell^i(\gamma_t)$, that is bounded away from zero for the hypothesis (1) and bounded from the gradient flow structure in the entire interval $(\varepsilon, T_{\max}-\delta)$. Since $\pas= \pax/\ell(\gamma_t)$
    we also get 
    \begin{equation}
        \norm{\pax\gamma_t}_{W^{2m-1}_2(0,1)} \leq \kst
    \end{equation}
    for every $t\in(\varepsilon, T_{\max} -\delta)$, where $\kst$ depends only on $\varepsilon$ and $T_{\max}$.
    Since $p\in (2m+1,4m+2)$, Sobolev embedding theorem \cref{eq:sobolevembeddingclassico} yields
    \begin{equation}\label{eq:sobolevboundlongtime}
    \norm{\pax\gamma_t}_{W_p^{2m-1-\nicefrac{2m}{p}}\left(0,1\right)}\leq C,
    \end{equation}
    where $\kst$ depends only on $\varepsilon$ and $T_{\max}$.
    By \cref{short-time}, we can restart the flow $\Net_t$ for all times $\tau \in [0,2T]$ where $T<T_{\max}/4$ only depends on the constant $\kst$ in \cref{eq:sobolevboundlongtime}. Taking $\delta =T$, $\Net_t$ exists for every $t \in [0,T_{\max} + T)$ contradicting the maximality of $T_{\max}$.
\end{proof}

\section{Application to the Willmore functional}

We now consider the Willmore functional, i.e.
\begin{equation}
\mathcal{W}(\gamma) = \E_2(\gamma) + \E_1(\gamma)= \int_\gamma \frac{1}{2} \abs{\pas^2\gamma}^2 +\frac{1}{2} \dif \gamma.
\end{equation}
Thanks to \cref{eq:normalvelocity}, the velocity of the flow is given by
\begin{equation}
    V\nu = \underbrace{- (\pas^4 \gamma )^\perp + \ip {\pas^3 \gamma, \pas\gamma} \pas^2\gamma -\frac{1}{2}\abs{\pas^2 \gamma}^2   \pas^2 \gamma}_{V_2\nu} + \underbrace{\frac{1}{2}(\pas^2 \gamma)^\perp}_{V_1 \nu}.
\end{equation}

At any $d$-junction $P\in \nodes$ the boundary term is 
\begin{equation}\label{eq:boundary_willmore}
    \sum_{j=1}^{d} \ip{\varphi^j,\psi_2^0(\gamma^j)+\psi^0_1(\gamma^j)}+\ip{\delta_\varphi \pas \gamma^j, \nu^j}\ip{\nu^j,\psi^1_2(\gamma^j)}.
\end{equation}
Observe that in this specific case
\begin{equation}
    \psi^1_2( \gamma^j) = \ip{\pas^2 \gamma^j, \nu^j}\nu^j.
\end{equation}

\subsection{Topological conditions}
For simplicity, we assume that $\gamma_0 = (\gamma^1_0,\gamma^2_0,\gamma^3_0)$ is a star network with a $3$-junction $P_0=\gamma^1_0(1)=\gamma^2_0(1)=\gamma^3_0(1)$ and $3$ fixed endpoints $\gamma^1_0(0)=P^1_0$, $\gamma^2_0(0)=P^2_0$, and $\gamma^3_0(0)=P^3_0$.

At each endpoint the matrix $a_{11} = \id_{\R^2}$ and the non-degeneracy condition is automatically satisfied. The resulting maximal order condition will then be $B_{3j}=0$ since $\pi_{11}=0$.

At the junction, the non-degeneracy condition is satisfied provided $\lspan\{\nu^1, \nu^2,\nu^3\}=\R^2$. Moreover, the maximal order condition will be
\begin{equation}
    B_{3j} = \sum_{i,k=1}^d \ip{ \ndA(\nu)^{-1}_{ji} \pi_{ik} [\psi^0_1(\gamma^k) + \psi^0_2(\gamma^k)], \nu^j}.
\end{equation}
Since $\{\nu^1, \nu^2, \nu^3\}$ generates  $\R^2$, $\det\ndA(\nu)\neq 0$ and $\pi_{ik}= \frac{1}{3} \id_{\R^2}$, thus $B_{3j}=0$ is equivalent to
\begin{equation}
    \sum_{k=1}^3 \psi^0_1( \gamma^k) + \psi^0_2(\gamma^k)=0.
\end{equation}

\subsection{Natural condition at the endpoints}
If we do not impose any condition of order $1$ at the endpoint of the curve $\gamma^i$, the resulting $\R^{1 \times 1}$ matrix will be $a_{1111} = 0$, then $a^\perp_{1111} =1$. Hence, the corresponding Neumann second order condition will be
\begin{equation}
    B_{21} =  k^i(0).
\end{equation}
\subsection{Natural conditions at the junction}
If we do not impose any condition of order $1$ at the junction, the $\R^{3 \times 3}$ matrix will be $a_{1i1j}=0$ and therefore $a^\perp_{1i1j} = \delta_{ij}$. Hence the second order condition will be
\begin{equation}
    B_{2i} = k^i(1).
\end{equation}

\medskip
The gradient flow of the Willmore functional with natural conditions is given by
\begin{equation}\label{eq:willmore_natural}
\begin{cases}[@{\ }l@{\kern.4cm}l]
    (\pat \gamma^i)^\perp = \left[- \pas^2 (k^i) -\frac{1}{2}(k^i)^3 + \frac{1}{2}(k^i)\right]\nu^i& \text{normal velocity} \\[.2cm]
    
    \gamma^i(t,0)=P^i& \text{fixed endpoint condition}\\
    k^i(t,0)=0 & \text{second-order condition}\\[.2cm]
    
    \gamma^1(t,1)=\gamma^2(t,1)=\gamma^3(t,1) & \text{concurrency condition} \\
    {k^i}(t,1) = 0 &\text{second-order conditions}\\
    \sum_{i=1}^3 [\pas k^i \nu^i +\frac12 (k^i)^2\tau^i-\frac12\tau^i](t,1) =0& \text{third-order condition}\\[.2cm]

    \gamma^i(0,x)=\gamma^i_0(x) & \text{initial condition}
\end{cases}
\end{equation}

\subsection{Preserved angles condition}
In this case we are adding a first order condition on the unit tangent vectors at the junction which are
\begin{equation}
    B_{1j}= \ip{\tau^j,R_{j,j-1}\tau^{j-1}}
\end{equation}
for $j=2,3$. $R_{j,j-1}$ is the fixed rotation of the plane which sends $\tau^{j-1}_0(1)$ to $\nu^j_0(1)$. The matrix $a_{1i1j} = (\delta_{ij} -\delta_{(i-1)j})$ for $i=2,3$ and $a_{111j}=0$. Hence the matrix $a^\perp$ is given by $a^\perp_{111j}=1$ and $a_{1i1j}=0$ for $i=2,3$. We then get
\begin{align}
    B_{2j}&=\sum_{l=1}^{3} a_{1j1l}^\perp k^l = 0 & j=2,3\\
    B_{21}&=\sum_{l=1}^{3} a_{111l}^\perp k^l=\sum_{l=1}^3 k^l
\end{align}

The gradient flow of the Willmore functional with preserved angle conditions is given by

\begin{equation}\label{eq:willlmorepreserved}
\begin{cases}[@{\ }l@{\kern.4cm}l]
    (\pat \gamma^i)^\perp = \left[- \pas^2 (k^i) -\frac{1}{2}(k^i)^3 + \frac{1}{2}(k^i)\right]\nu^i& \text{normal velocity} \\[.2cm]

    \gamma^i(t,0)=P^i& \text{fixed endpoint condition}\\
    k^i(t,1)=0 & \text{second-order condition}\\[.2cm]
    \gamma^1(t,1)=\gamma^2(t,1)=\gamma^3(t,1) & \text{concurrency condition} \\
    \ip{\tau^i, R_{i,i+1} \tau^{i+1}} (t,1)=0 & \text{preserved angle condition}\\
    \sum_{i=1}^{3}{k^i}(t,1) = 0 &\text{second-order conditions}\\
    \sum_{i=1}^3 [\pas k^i \nu^i +\frac12 (k^i)^2\tau^i-\frac12\tau^i](t,1) =0& \text{third-order condition}\\[.2cm]
    \gamma^i(0,x)=\gamma^i_0(x) & \text{initial condition}
\end{cases}
\end{equation}

\subsection{Gradient flow of the Willmore functional}

The following theorem holds

\begin{theorem}
    Let $\gamma_0$ an admissible network of class $W^{4-\nicefrac{4}{p}}_p$.
    There exists a $T_{\max}>0$ and a maximal unique (up to reparametrization) solution $(\gamma_t)_{t \in [0, T_{\max})}$ to the problem \cref{eq:willmore_natural} or \cref{eq:willlmorepreserved} with 
    \begin{equation}
        \gamma^i_t \in W^{1}_p ( (0, T_{\max}); L_p(0,1)) \cap L_p((0,T_{\max}); W^{4}_p(0,1)).
    \end{equation}
    Moreover, either      $T_{\max}=+\infty$
    or at least one of the following happens:
    \begin{enumerate}
        \item  there exists a curve $\gamma^i_t$ such that
    \begin{equation}
        \liminf_{t \to T^{-}_{\max}} \ell(\gamma^i_t) =0;
    \end{equation}
    \item there exists a $d$-junction such that
    \begin{equation}
        \liminf_{t\to T_{\max}^-} \det\ndA(\nu_t)=0,
    \end{equation}
    where $\ndA(\nu_t)$ is defined in \cref{eq:nondegeneracy_system}.
    \end{enumerate}
\end{theorem}

\begin{proof}
    Since the conditions are admissible, solution and uniqueness follows from \cref{short time existence} and \cref{uniqueness}. The long time behavior directly follows from \cref{thm:longtime}.
\end{proof}

\begin{remark}
    Scenario (2) never happens for solutions to the problem \cref{eq:willlmorepreserved}.
\end{remark}

\begin{remark}
    Problem \cref{eq:willmore_natural} coincides with the problem studied and solved in \cite{garcke-menzel-pluda-20} and \cite{DaChPo19,DaChPo20}.
\end{remark}

\begingroup
\setlength{\emergencystretch}{1em}
\printbibliography
\endgroup
\end{document}